\documentclass[times,sort]{elsarticle}

\journal{Journal of Computational Physics}

\usepackage{yfonts}

\usepackage[margin=2.0cm]{geometry}
\usepackage{graphicx}
\usepackage{amssymb}
\usepackage{amsmath}
\usepackage[ruled,vlined]{algorithm2e}
\DontPrintSemicolon
\SetAlFnt{\small}

\usepackage{mathrsfs}
\usepackage{epstopdf}
\usepackage{stmaryrd}
\usepackage{subfig}
\usepackage{amsthm}
\usepackage{lmodern}
\usepackage{enumitem}
\usepackage{amsfonts}
\usepackage{mathtools}
\usepackage{booktabs}
\usepackage{xcolor}
\usepackage{stackrel}
\theoremstyle{plain}
\newtheorem{thm}{Theorem}
\newtheoremstyle{proofpartstyle}
  {1pc} 
  {1pc} 
  {} 
  {2.5mm} 
  {\itshape} 
  {} 
  {-0.1em} 
  {} 
\theoremstyle{proofpartstyle} \newtheorem*{proofpart}{Proof of Part}
\makeatletter
\@addtoreset{proofpart}{theorem}
\makeatother
\theoremstyle{plain}
\newtheorem{Lem}{Lemma}
\newtheorem{Prop}{Proposition}
\theoremstyle{remark}
\newtheorem{rem}{Remark}

\allowdisplaybreaks

\usepackage{placeins}

\numberwithin{equation}{section}

\newcommand{\pderivative}[2]{\frac{\partial #1}{\partial #2}}
\newcommand{\mytext}[1]{\textnormal{\textbf{#1}}}

\newcommand{\dmat}{\mat{D}}
\newcommand{\dmatT}{\mat{D}^T}
\newcommand{\mmat}{\mat{{M}}}
\newcommand{\qmat}{\mat{{Q}}}
\newcommand{\diffmat}{\mmat^{-1}\boldsymbol{\Delta}}

\newcommand{\smat}{\mat{{S}}}
\newcommand{\zmat}{\mat{{Z}}}

\newcommand{\average}[1]{\left\{\!\left\{ #1\right\}\!\right\}}
\newcommand{\jump}[1]{\llbracket #1 \rrbracket}

\newcommand{\half}{\frac{1}{2}}
\newcommand{\fourth}{\frac{1}{4}}
\newcommand{\mat}[1]{\mathbf{#1}}
\newcommand\statevec[1]{\vec{#1}}
\newcommand\spacevec[1]{\vec{ #1}}
\newcommand{\arry}[1]{\mathbf{ #1}}

\newcommand{\hmat}{\mat{{h}}}

\definecolor{redcol}{rgb}{1.,0.,0.0}

\begin{document}
\begin{frontmatter}

\title{An Entropy Stable Nodal Discontinuous Galerkin Method for the Two Dimensional Shallow Water Equations on Unstructured Curvilinear Meshes with Discontinuous Bathymetry}
\author[unikoeln]{Niklas~Wintermeyer}
\author[unikoeln]{Andrew R.~Winters\corref{cor1}}
\cortext[cor1]{Corresponding author.}
\ead{awinters@math.uni-koeln.de}
\author[unikoeln]{Gregor J.~Gassner}
\author[fsu]{David A.~Kopriva}

\address[unikoeln]{Mathematisches Institut, Universit\"{a}t zu K\"{o}ln, Weyertal 86-90, 50931 K\"{o}ln, Germany}
\address[fsu]{Department of Mathematics, The Florida State University, Tallahassee, FL 32306, USA}

\begin{abstract}
We design an arbitrary high-order accurate nodal discontinuous Galerkin spectral element approximation for the nonlinear two dimensional shallow water equations with non-constant, possibly discontinuous, bathymetry on unstructured, possibly curved, quadrilateral meshes. The scheme is derived from an equivalent flux differencing formulation of the split form of the equations. We prove that this discretisation exactly preserves the local mass and momentum. Furthermore, combined with a special numerical interface flux function, the method exactly preserves the mathematical entropy, which is the total energy for the shallow water equations. By adding a specific form of interface dissipation to the baseline entropy conserving scheme we create a provably entropy stable scheme. That is, the numerical scheme discretely satisfies the second law of thermodynamics. Finally, with a particular discretisation of the bathymetry source term we prove that the numerical approximation is well-balanced. 
We provide numerical examples that verify the theoretical findings and furthermore provide an application of the scheme for a partial break of a curved dam test problem.
\end{abstract}


\end{frontmatter}
\section{Introduction}\label{Intro}

Fluid flows in lakes, rivers, and near coastlines are of interest in oceanography and climate modeling. For such flows the vertical scales of motion are much smaller than the horizontal scales. From this and the assumption of hydrostatic balance \cite{whitham1974}, the Euler equations can be simplified to the shallow water equations. If the fluid flows over a non-constant bottom topography the shallow water equations may be written as a hyperbolic system of balance laws
\begin{equation}\label{2DBalLaw}
\statevec{w}_t + \statevec{f}_x + \statevec{g}_y = \statevec{s}.
\end{equation}
It is well-known that solutions of the balance laws \eqref{2DBalLaw} may develop discontinuities in finite time, independent of the smoothness of the initial data. Hence, we consider solutions of the balance laws \eqref{2DBalLaw} in a weak sense that are well-defined provided the source term $\statevec{s}$ remains uniformly bounded, i.e., weak solutions of \eqref{2DBalLaw} are well-defined under the assumption that the function used to model the bottom topography is in the space $W^{1,\infty}(\mathbb{R})$, see e.g. \cite{evans2010}.

The design of numerical methods to approximate (\ref{2DBalLaw}) is driven by the need for stable, accurate and robust behaviors. For instance, the preservation of steady-state solutions is critical in problems with non-constant bottom topographies. Preserving steady solutions discretely is particularly troublesome for discontinuous bottom topographies where special discretisations of the source term are required, e.g. \cite{fjordholm2011,xing2014}. One steady-state constraint for the shallow water equations is the ``lake at rest'' condition \cite{audusse2004,leveque1998,xing2014}, since the relevant waves in a flow can be viewed as small perturbations of the lake at rest, see \cite{audusse2004}. A good numerical method for the shallow water equations should accurately capture both steady states and their small perturbations (quasi-steady flows) so as to diminish the appearance of unphysical waves with magnitude proportional to the mesh size 
(a so-called ``numerical storm'' \cite{Chertock2014}), that are normally present for numerical schemes that cannot preserve the ``lake at rest'' condition. A numerical method that exactly preserves the ``lake at rest'' steady state property is said to be \emph{well-balanced}, see e.g. \cite{audusse2004,fjordholm2011,gassner2015,leveque1998}.

Another critical requirement of the numerics is robustness and the ability of the method to remain stable and accurate, particularly the removal of aliasing errors that can drive nonlinear instabilities, and maintenance of stability even if discontinuities develop. Recent work has appeared on the use of high-order discontinuous Galerkin (DG) approximations to create robust numerical methods for the solution of systems of conservation laws, e.g., \cite{carpenter_esdg,gassner2015,kopriva2014}. These robust high-order DG methods may be derived from the perspective of mathematical entropy conservation, e.g. \cite{carpenter_esdg,merriam1989,tadmor2003}, or reformulations of the PDE into a split formulation to maintain conservation, e.g. \cite{gassner2015,kopriva2014}. The motivations behind the two approaches are similar \cite{tadmor1984}.

The split form of an equation is usually found by averaging its conservative form and non-conservative advective form. This is problematic as it is not obvious that discretisations of the split form remain conservative, yet conservation is desired for the numerical solution to exhibit correct shock speeds. Recent success has been had using diagonal norm summation-by-parts (SBP) finite difference operators to discretise the spatial derivatives in the split formulation of the equations \cite{skew_sbp2,gassner_skew_burgers,gassner_kepdg,gassner2015}. Fisher et al. \cite{skew_sbp2} show that split form operators derived from SBP derivative matrices are consistent and conservative in the Lax-Wendroff sense. There is now a known link between SBP finite difference operators and the discontinuous Galerkin spectral element approximation with Gauss-Lobatto points, e.g. \cite{gassner_skew_burgers}. This link was used in \cite{gassner2015} to derive an entropy conserving discontinuous Galerkin spectral element method (DGSEM) for the one dimensional shallow water equations. This paper exploits the links further and extends in a non-trivial way the previous work found in \cite{gassner2015} to multiple dimensions and possible discontinuous bottom topographies.

In this paper we present an entropy stable, high-order discontinuous Galerkin spectral element approximation for the shallow water equations with a discontinuous bottom topography for unstructured and curved quadrilateral grids. The DGSEM is naturally discontinuous at element boundaries, so we ensure high-order (spectral) accuracy by placing element boundaries at discontinuities in the bottom topography. The ability to do so allows one to model realistic bottom topographies appearing in oceanography. The scheme presented here is also well-balanced, an attribute difficult to guarantee in curvilinear coordinates. We find that the numerical satisfaction of the metric identities \cite{koprivabook} (referred to in \cite{fisher2012} as the geometric conservation law) is critical to prove that the baseline scheme remains entropy conservative and well-balanced on arbitrary meshes. 

Our approach is to use results of Fisher \cite{fisher2012} and Fisher and Carpenter \cite{fisher2013} to derive an entropy conserving approximation, and from that an entropy stable one, which is possible because we can reformulate the spectral element approximation of the split form of the shallow water equations into an equivalent flux differencing structure. We use the flux differencing reformulation to prove the underlying properties of the entropy stable DGSEM, as well as to highlight how an existing DGSEM code can be altered to incorporate entropy stability.


The paper is organised as follows: in Sec. \ref{sec:skew-symmetric-swe} we begin with a brief description of the entropy analysis of the two dimensional shallow water equations. We outline the discontinuous Galerkin spectral element method with the summation-by-parts (SBP) property in Sec. \ref{sec:DG-Disc}. This section also introduces the important reformulation of the DGSEM into an equivalent flux differencing framework, which is the critical equivalence that allows us to use existing theory. We provide in Sec. \ref{sec:ECCurvilinear} a discretisation of the two dimensional shallow water equations using the flux differencing formulation that is conservative and entropy conservative on curvilinear meshes. We also provide a detailed proof that the approximation remains well-balanced. Then in Sec. \ref{sec:ent-stab} additional dissipation is added to the scheme to ensure that the approximation remains valid for flow regimes that may contain shocks. Numerical results in Sec. \ref{NumResults} demonstrate and underline our theoretical findings. Our conclusions are presented in Sec.~\ref{sec:conclusion}. Finally, \ref{sec:efficiency} provides algorithms and implementation details of how a standard DGSEM code can be altered to incorporate the newly proposed entropy stable fluxes.

\section{Shallow water equations}\label{sec:skew-symmetric-swe}

We begin with the balance law form of the two-dimensional shallow water equations
\begin{equation}
\begin{aligned}\label{sw-standard1}
 h_t + (hu)_x + (hv)_y &= 0,\\
(hu)_t + (h\,u^2+g\,h^2/2)_x + (huv)_y &= -g\,h\,b_x,\\
(hv)_t + (huv)_x + (h\,v^2+g\,h^2/2)_y &= -g\,h\,b_y,
\end{aligned}
\end{equation}
which includes the continuity equation and the momentum balances. The quantity $h=h(x,y,t)$ denotes the water height measured from the bottom topography $b=b(x,y)$ with the total height given by $H=h+b$. Additionally the constant $g$ is the gravitational acceleration. The fluid velocities are given by $u=u(x,y,t)$ in the $x-$direction and $v=v(x,y,t)$ in the $y-$direction.
The SW model \eqref{sw-standard1} is compactly written as a system of balance laws (\ref{2DBalLaw})
with $\statevec{w}=(h,hu,hv)^T$, the fluxes 
\begin{equation}\label{eq:physicalFluxes}
\statevec{f}=(hu,h\,u^2+g\,h^2/2,huv)^T,\quad \statevec{g}=(hv,huv,h\,v^2+g\,h^2/2)^T,
\end{equation}
and the source term $\statevec{s} = (0,-ghb_x,-ghb_y)^T$.

Since the system \eqref{sw-standard1} is nonlinear, we must define in what sense our numerical approximation will be stable. The extension of the usual $L^2$ stability for linear problems is the so-called \emph{entropy stability} \cite{tadmor2003}, where a (generalized, mathematical) entropy function rather than the $L^2$ norm of the solution is non-increasing in time.
To this end we impose the entropy condition as an additional admissibility criterion on the system. 

The entropy condition states that for smooth solutions the entropy of the system is conserved and for discontinuous solutions the entropy decays. Numerical approximations that satisfy the entropy condition discretely are referred to as \emph{entropy stable}. The balance law \eqref{2DBalLaw} for the shallow water equations is equipped with a convex mathematical entropy function $e=e(\statevec{w})$ in the form of the total energy \cite{fjordholm2011}
\begin{equation}
e :=\half h\left(u^2+v^2\right) + \half gh^2 + ghb.
\end{equation}

To develop the conservation law for the entropy function $e(\statevec{w})$ we define the set of entropy variables, $\statevec{q}=(q_1,q_2,q_3)^T$, by
\begin{equation}
\label{entropyVars}
\begin{aligned}
&q_1:=\pderivative{e}{w_1}= g(h+b) - \half u^2 - \half v^2, \\
&q_2:=\pderivative{e}{w_2}= u, \\
&q_3:=\pderivative{e}{w_3}= v .
\end{aligned}
\end{equation}
To determine the entropy fluxes, $\mathcal F(\statevec{w})$ and $\mathcal G(\statevec{w})$, we use two compatibility relations that must hold between them, the entropy variables and the physical fluxes \cite{tadmor2003}
\begin{equation}\label{eq:entComp}
\mathcal F_{\statevec{w}} = \statevec{q}^{\,T}\statevec{f}_{\statevec{w}},\quad \mathcal G_{\statevec{w}} = \statevec{q}^{\,T}\statevec{g}_{\statevec{w}}. 
\end{equation}
We pre-multiply the balance law \eqref{2DBalLaw} with the entropy variables \eqref{entropyVars} and apply the conditions \eqref{eq:entComp} to find the conservation law for the entropy function 
\begin{equation}
\label{EnergyEquation2D}
e_t + \mathcal{F}_x + \mathcal{G}_y = 0.
\end{equation}
Explicitly, the entropy fluxes of the shallow water equations are
\begin{equation}
\begin{aligned}
\label{EnergyFluxes}
&\mathcal F:= \half (hu^3  +  huv^2)  + g(hu(h+b)) ,\\
&\mathcal G:= \half (hv^3 + hu^2v) + g(hv(h+b)) .
\end{aligned}
\end{equation}
In the presence of discontinuities the entropy conservation law \eqref{EnergyEquation2D} becomes the entropy inequality
\begin{equation}
\label{EnergyInequality2D}
e_t + \mathcal F_x + \mathcal G_y \leq 0.
\end{equation}
We then define an entropy conserving approximation of the nonlinear shallow water equations, (\ref{2DBalLaw}), to be one that discretely satisfies (\ref{EnergyEquation2D}) and entropy stable if (\ref{EnergyInequality2D}) is satisfied \cite{tadmor2003}.

Finally, we note that systems of balance laws have important steady state configurations where the flux and source terms are in balance. For the shallow water equations, one such steady state solution is the ``lake at rest'' condition defined by
\begin{equation}
\label{lakeAtRest}
\begin{aligned}
&h+b=\textrm{const},\\
&u=v=0.\\
\end{aligned}
\end{equation}
A numerical method that preserves the ``lake at rest'' state is said to be well-balanced. If a method is not well-balanced spurious waves on the order of the mesh size truncation error can be generated and pollute the approximation.

\section{Nodal discontinuous Galerkin spectral element method}\label{sec:DG-Disc}

The entropy stable method that we propose is a form of nodal Discontinuous Galerkin Spectral Element Method (DGSEM). In this section, we introduce the basic construction of the DGSEM on curvilinear quadrilateral grids. A more complete discussion can be found in \cite{koprivabook}. Also in this section, we provide details about the relationship of DG methods to summation-by-parts (SBP) operators, which allows us to write the approximation as an equivalent sub-cell flux differencing formulation (FDF) \cite{fisher2013}. The FDF is useful for theoretical purposes, particularly the proof of local conservation and satisfaction of the Lax-Wendroff condition. Therefore, the DGSEM has local and global conservation. Also, special choices of the flux functions in the FDF generate a DG discretisation of a split form of the original PDE. Additionally, if the underlying flux functions in the FDF are two-point entropy conserving fluxes then the FDF remains high-order and entropy conservative \cite{fisher2013}. Because of the equivalence between the FDF and the DGSEM with the SBP property this automatically generates an entropy conservative DGSEM. 


\subsection{Conservation law in curvilinear coordinates}
\label{sec:mapping}

The DGSEM approximates the system of conservation laws (\ref{2DBalLaw}) defined on a domain $\Omega$ on an unstructured mesh of quadrilateral elements. 
To simplify the discussion, we work with the components of the system (\ref{sw-standard1}) written as
\begin{equation}
\label{genericEq}
{(w_k)}_t + {(f_k(\statevec{w}))}_x+ {(g_k(\statevec{w}))}_y= s_k(w),\quad k = 1,2,3,
\end{equation}
where $w_1=h$, $w_2=hu$ and $w_3=hv$. 

We decompose $\Omega$ into non-overlapping quadrilateral elements $G$ and for computational efficiency map each element to the computational reference element $E=[-1,1]^2$. A commonly used transformation between the reference square and an arbitrary curve-sided quadrilateral element is transfinite interpolation with linear blending \cite{koprivabook}. The mapping between the coordinates of the reference square $(\xi, \eta)$ and the physical coordinates $\spacevec x = (x, y)$ is 
\begin{align}\label{metricterms}
{\spacevec x}(\xi,\eta)&=\frac{1}{2}\bigl[(1-\xi)\spacevec\Gamma_4(\eta)+(1+\xi)\spacevec\Gamma_2(\eta)+(1-\eta)\spacevec\Gamma_1(\xi)+(1+\eta) \spacevec\Gamma_3(\xi)\bigr] \nonumber \\
& -\frac{1}{4}\bigl[(1-\xi)\{(1-\eta)\spacevec\Gamma_1(-1)+(1+\eta)\spacevec\Gamma_3(-1)\}\\
& +(1+\xi)\{ (1-\eta)\spacevec\Gamma_1(1)+(1+\eta)\spacevec\Gamma_3(1)\}\bigr], \nonumber
\end{align} 
where we assume that each element is bounded by four curves $\spacevec\Gamma_j$, $j=1,2,3,4$.

Under the transformation the conservation law on $\Omega$ remains a conservation law on $E$. We transform the two dimensional balance law \eqref{genericEq} from physical space to the reference space by rewriting derivatives using the chain rules
\begin{equation}
\label{PhysicalDerivativeGeneral}
\begin{aligned}
\pderivative{w}{x} = \pderivative{w}{\xi}\pderivative{\xi}{x}+\pderivative{w}{\eta}\pderivative{\eta}{x},\\
\pderivative{w}{y} = \pderivative{w}{\xi}\pderivative{\xi}{y}+\pderivative{w}{\eta}\pderivative{\eta}{y},
\end{aligned}
\end{equation}
for some $w(x,y)$. Applying the chain rule \eqref{PhysicalDerivativeGeneral} to \eqref{genericEq} gives us the transformed conservation law in reference space 
\begin{equation}\label{newConsLaw} 
\mathcal J{(w_k)}_t + (\tilde{f}_k)_{\xi} + (\tilde{g}_k)_{\eta} = 0,
\end{equation}
where the element mapping \eqref{metricterms} is used to compute the Jacobian, $\mathcal J$, and the contravariant flux components $\tilde{f}_k,\,\tilde{g}_k$, $k=1,2,3$ according to
\begin{equation}
\label{ContravariantFluxes}
\begin{aligned}
\mathcal J&=x_{\xi}\,y_{\eta} - x_{\eta}\,y_{\xi},\\
\tilde{f}_k(\statevec{w})&=y_{\eta}\,f_k(\statevec{w}) - x_{\eta}\,g_k(\statevec{w}),\\
\tilde{g}_k(\statevec{w})&=-y_{\xi}\,f_k(\statevec{w}) + x_{\xi}\,g_k(\statevec{w}),
\end{aligned}
\end{equation}
for $k=1,2,3$.

\subsection{Polynomial approximation on the reference element}
\label{sec:poly_approximation}

The DGSEM is constructed by approximating the variables ${w}_k$ and contravariant fluxes ${\tilde{f}}_k$, ${\tilde{g}}_k$ in the reference space by polynomials of degree $N$ in each direction.  We use a nodal form of the interpolation with nodes defined at the Legendre-Gauss-Lobatto (LGL) points $\{\xi_i\}_{i=0}^N$ and $\{\eta_j\}_{j=0}^N$ in the reference square $E=[-1,1]^2$. The Lagrange basis functions for the interpolant are
\begin{equation}
\label{lagrange_basis}
 \ell_j(\xi)=\prod\limits_{i=0,i\neq j}^N\frac{\xi - \xi_i}{\xi_j-\xi_i},\qquad j=0,\ldots,N,
\end{equation}
and satisfy the cardinal property
\begin{equation}
\label{cardinal}
\ell_j(\xi_i)=\delta_{ij}, \qquad i,j=0,\ldots,N,
\end{equation}
where $\delta_{ij}$ denotes Kronecker's symbol with $\delta_{ij}=1$ for $i=j$ and $\delta_{ij}=0$ for $i\neq j$. We write the element-wise polynomial approximation (e.g for the components of $\statevec{w}$) as
\begin{equation}
\label{poly_approx}
 w_k(x,y,t)\big|_{G}=w_k(x(\xi,\eta),y(\xi,\eta),t)\approx W_k(\xi,\eta,t) :=\sum\limits_{i=0}^N\sum\limits_{j=0}^N W^{i,j}_{k}(t)\,\ell_i(\xi)\,\ell_j(\eta), \quad k=1,2,3,
\end{equation}
where $\{W_k^{i,j}(t)\}_{i=0, \, j=0}^{N,N}$ are the time dependent nodal degrees of freedom. 
The nodally represented polynomial (\ref{poly_approx}) is equivalent to an orthogonal Legendre polynomial expansion used in a modal spectral method, but is more convenient to use in the approximation of nonlinear equations like the shallow water equations. 

We use the idea of collocation throughout this work to approximate quantities derived from the $W_{k}$. For instance the velocity $u$ is approximated by a polynomial of degree $N$ in each direction \eqref{poly_approx} as well, where its nodal values are computed as
\begin{equation}
U^{i,j}:=\frac{W_2^{i,j}}{W_1^{i,j}},\quad i,j=0,\ldots,N.
\end{equation}
This collocation strategy also applies to the contravariant fluxes, where we interpolate the metric terms at the same nodes. For instance
\begin{equation}
\tilde{F}_k^{i,j}=y_{\eta}(\xi_i,\eta_j)\,F_k({\statevec{W}}^{i,j}) - x_{\eta}(\xi_i,\eta_j)\,G_k({\statevec{W}}^{i,j}), \qquad i,j=0,\ldots,N,\quad k=1,2,3
\end{equation}
where the three flux components for the shallow water equations are defined in \eqref{eq:physicalFluxes}. Similarly, the Jacobian of the transformation is approximated by the polynomial of degree $N$ with nodal values $J^{i,j} =\mathcal J(\xi_i,\eta_j)$.

For spectral approximations, the derivative is approximated elementwise directly from the derivative of the polynomial approximation, e.g.,  
\begin{equation}
\frac{\partial }{\partial \xi}W_k(\xi,\eta,t) = \sum\limits_{i=0}^N\sum\limits_{j=0}^N W_k^{i,j}(t)\,\frac{\partial }{\partial \xi}\ell_i(\xi)\,\ell_j(\eta),
\end{equation}
where $k=1,2,3$. We introduce the polynomial derivative matrix $\dmat$ with entries
\begin{equation}
\label{DmatDef}
D_{ij}:=\frac{\partial\ell_j}{\partial\xi}\Bigg|_{\xi=\xi_i},\qquad i,j=0,\ldots,N,
\end{equation}
which is used to calculate the derivative with respect to $\xi$.
Since we use the same polynomial ansatz in $\xi$ and $\eta$ direction, the derivative operator is identical in each direction. 

We can reuse the 1D operator (\ref{DmatDef}) for the 2D scheme if we store the individual slices in a 2D array. Entries in one column refer to the nodal values at constant $\eta$, entries in one row share the same $\xi$
\begin{equation}
\label{2DArrayStorage}
\begin{aligned}
\arry{W}_k  :=\begin{pmatrix}
(W_k)^{\xi=-1,\eta=-1} & \cdots & (W_k)^{\xi=-1,\eta=+1}\\
 \vdots & \ddots &  \vdots \\
(W_k)^{\xi=+1,\eta=-1} & \cdots & (W_k)^{\xi=+1,\eta=+1}
\end{pmatrix}=\begin{pmatrix}
W_k^{0,0} & \cdots & W_k^{0,N}\\
 \vdots & \ddots &  \vdots \\
W_k^{N,0} & \cdots & W_k^{N,N}
\end{pmatrix}.
\end{aligned}
\end{equation}
By storing the nodal values of variables, fluxes and metric terms in such 2D arrays, we can multiply the 1D derivative operator $\mat{D}$ from the left  to represent taking the $\xi$ derivative at each constant $\eta$. If we multiply with $\mat{D}^T$ from the right, we obtain the $\eta$ derivative. This notation allows us to write the scheme in a compact matrix-vector notation while using unmodified 1D operators.

With notation \eqref{2DArrayStorage}, the nodal values of the derivative of a nodal approximation $w$ are given by
\begin{equation}
\label{eq:discrete_derivative}
(\arry{W}_k)_\xi = \mat{D}\,\arry{W}_k\qquad\text{and}\qquad(\arry{W}_k)_\eta = \arry{W}_k\,\mat{D}^T,
\end{equation}
or in index notation
\begin{equation}
\label{eq:discrete_derivative2}
({W}_k^{i,j})_\xi = \sum_{l=0}^N{D}_{il}\,{W}^{l,j}_k\qquad\text{and}\qquad({W}^{i,j}_k)_\eta = \sum_{l=0}^ND_{li}{W}^{j,l}_k,
\end{equation}
where $i,j=0,\ldots,N$.

We demand that the discretisation preserves free-stream solutions, i.e. constant solutions of the balance law \eqref{newConsLaw} should remain constant for all times. A necessary and sufficient condition for constant state preservation is that the metric identities 
\begin{equation}\label{metricIDs}
\frac{\partial}{\partial\xi}\mathcal J\spacevec{a}^1+\frac{\partial}{\partial\eta}\mathcal J\spacevec {a}^2 = \vec{0},
\end{equation}
are satisfied, where the volume weighted contravariant basis vectors, $\mathcal J\spacevec{a}^{\,i},i=1,2$ are
\begin{equation}\label{metricTerms}
\mathcal J\spacevec{a}^1 = \left(y_\eta,-x_\eta\right)^T,\quad \mathcal J\spacevec{a}^2 = \left(-y_\xi,x_\xi\right)^T.
\end{equation}

The metric identities are not automatically satisfied for the discretisation, which, using the array notation \eqref{2DArrayStorage} for the nodal values of the metric terms and \eqref{eq:discrete_derivative} for the derivative can be expressed as
\begin{equation}\label{metricIDsDiscrete}
\begin{aligned}
\dmat \mat{y}_\eta  - \mat{y}_\xi \dmatT&= \mat{0}\\
-\dmat \mat{x}_\eta  + \mat{x}_\xi \dmatT&= \mat{0}.
\end{aligned}
\end{equation} 

Kopriva \cite{Kopriva:2006er} proved that free-stream preservation is guaranteed for the DGSEM when the linear blending formula (\ref{metricterms}) is used and boundaries of the quadrilateral elements are approximated by polynomials with an order equal to (or lower than) the polynomial order of the approximate solution. 
Thus, we use an isoparametric approximation in which each boundary curve $\spacevec \Gamma_j(\zeta)$, $\zeta\in[-1,1]$ of an element $G$ is approximated by a polynomial of order $N$. We use the same Lagrange basis functions \eqref{lagrange_basis} to approximate the boundary curves 
\begin{equation}\label{Gammas}
\spacevec \Gamma = \sum_{j=0}^N  \spacevec\Gamma(\zeta_j)\ell_j(\zeta),
\end{equation}
where due to their robust interpolation properties \cite{koprivabook}, the nodes $\{\zeta_j\}_{j=0}^N$ are typically chosen to be the Chebyshev-Gauss-Lobatto or Legendre-Gauss-Lobatto nodes. The polynomial boundary curve approximations \eqref{Gammas} are used to construct the mapping \eqref{metricterms} for each element. As the mapping is a polynomial in $\xi$ and $\eta$, the derivatives necessary to obtain the metric terms and the normal vectors are computed directly in the discrete derivative sense \eqref{eq:discrete_derivative}. Further details of the isoparametric polynomial approximation of boundary curves can be found in \cite{Kopriva:2006er,koprivabook}.

\subsection{Discontinuous Galerkin spectral element method (DGSEM)}
\label{sec:dgsem}

Following a standard approach we derive a nodal discontinuous Galerkin scheme, e.g. \cite{koprivabook} or \cite{hesthaven_warburton}. We will discuss the specific form of the approximation of the source term in Sec.~\ref{sec:ECDGSEM2D}. Omitting the source term, the nodal discontinuous Galerkin method in weak form of the transformed conservation law \eqref{newConsLaw} reads
\begin{equation}
\label{genericIntFormWeak2}
\int_{E,N}\, J\,(W_k)_t \,\varphi\,d{\xi}d{\eta} - \int_{E,N}\,\tilde{F}_k \varphi_{\xi}\,d{\xi}d{\eta}- \int_{E,N}\,\tilde{G}_k\varphi_{\eta}\,d{\xi}d{\eta} = -\oint_{\partial E,N} \left(\tilde{F}_k^*(\statevec W^+,\statevec W^-),\tilde{G}_k^*(\statevec W^+,\statevec W^-)\right)\cdot\hat{n}\,\varphi\,dS,
\end{equation}
for $k=1,2,3$ and in the equivalent strong form \cite{KoprivaGassner_GaussLob}
\begin{equation}
\label{genericIntFormStrong2}
\int_{E,N}\,\left(J\,(W_k)_t+ (\tilde{F}_k)_{\xi} + (\tilde{G}_k)_{\eta} \right)\,\varphi\,d{\xi}d{\eta}  = -\oint_{\partial E,N} \left(\tilde{F}_k^*(\statevec W^+,\statevec W^-)- \tilde{F}_k,\tilde{G}_k^*(\statevec W^+,\statevec W^-)-\tilde{G}_k\right)\cdot\hat{n} \varphi\,dS,
\end{equation}
for $k=1,2,3$, where $\varphi$ is a polynomial of degree $N$ or less in each space direction. We also use the notation $\int_{E,N}$  to represent Legendre-Gauss-Lobatto Quadrature, which is the tensor product of one space dimension quadrature
\begin{equation}
\int\limits_{-1,N}^{1}f(\xi)\,d\xi\approx\sum\limits_{j=0}^N f(\xi_j)\omega_j,
\end{equation}
with $\{\omega_j\}_{j=0}^N$ being the LGL quadrature weights. The quadrature is exact if the integrand $f$ is a polynomial of degree $2N-1$ or less.
The numerical fluxes $\tilde{F}_k^*$ and $\tilde{G}_k^*$, in the normal direction, couple neighbouring elements. We indicate this coupling by the dependence on the ``outer'' and ``inner'' values $\statevec W^+,\statevec W^-$ along the normal vector $\hat{n}$. The boundary integrals in \eqref{genericIntFormWeak2} and \eqref{genericIntFormStrong2} describe the integration along the four edges of the element $E$. 

The two forms (\ref{genericIntFormWeak2}) and (\ref{genericIntFormStrong2}) of the DGSEM are algebraically equivalent because the quadrature satisfies a \emph{summation-by-parts} rule \cite{KoprivaGassner_GaussLob}. In one space dimension, and for any two polynomials $U\left(\xi\right)$ and $V\left(\xi\right)$, exactness or the quadrature implies that
\begin{equation}
\sum\limits_{j = 0}^N {{U_j}{{V'}_j}{\omega _j}}  = \int_{ - 1}^1 {U{V_\xi }d\xi }  = \left. {UV} \right|_{ - 1}^1 - \int_{ - 1}^1 {{U_\xi }Vd\xi }  = \left. {UV} \right|_{ - 1}^1 - \sum\limits_{j = 0}^N {{{U'}_j}{V_j}{\omega _j}} .
\end{equation}
In other words,
\begin{equation}
\int_{ - 1,N}^1 {U{V_\xi }d\xi }  = \left. {UV} \right|_{ - 1}^1 - \int_{ - 1,N}^1 {{U_\xi }Vd\xi } .
\label{eq:SBPIntegral}
\end{equation}
The result extends to two and three space dimensions \cite{KoprivaGassner_GaussLob}. 


The formal statements of the DGSEM, (\ref{genericIntFormWeak2}) and (\ref{genericIntFormStrong2}) can be reduced to a pointwise form \cite{koprivabook}, which in turn can be represented in a matrix form where the nodal values are represented as arrays. 
To get equations for the nodal degrees of freedom, we take $\varphi=\ell_i(\xi)\ell_j(\eta)$. Then, for example,
\begin{equation}
\int_{E,N} {J{(W_k)_t}\varphi d\xi d\eta }  = \sum\limits_{n,m = 0}^N {{\omega _n}{\omega _m}{J^{n,m}}(W^{n,m}_k)_t{\ell _i}\left( {{\xi _n}} \right){\ell _j}\left( {{\eta _m}} \right)}  = {\omega _i}{\omega _j}{J^{i,j}}(W^{i,j}_k)_t\quad i,j=0,1,2,\ldots,N,
\label{eq:PtwWiseTimeIntegral}
\end{equation}
for $k=1,2,3$.

We can represent each component (\ref{eq:PtwWiseTimeIntegral}) in terms of the matrix-array notation introduced in (\ref{2DArrayStorage}) and (\ref{eq:discrete_derivative}) as 
\begin{equation}
\mat M \arry J\circ\left( \arry W_{k}\right)_{t} \mat M,
\end{equation}
where $\mat M$ is the diagonal matrix of the quadrature weights  (the \emph{mass matrix}), 
\begin{equation}
\mmat := \textrm{diag}(\omega_0,\ldots,\omega_N),
\end{equation}
and we introduce the notation of a component-wise Hadamard product of two matrices with the same dimension
\begin{equation}
\label{DefHadamard}
\arry{A}\circ\arry{B}=\arry{C} ,\quad \textrm{with} \quad  c_{ij}=a_{ij}\,b_{ij},\quad i,j=0,\ldots,N.
\end{equation}
Similarly, the integral of the $k-th$ component of the $\xi$ contravariant flux in strong form is
\begin{equation}
\int_{E,N} {{{\left( {{\tilde{F}_k}} \right)}_\xi }\varphi d\xi d\eta }  = \sum\limits_{n,m = 0}^N {{\omega _n}{\omega _m}\left( {\sum\limits_{l = 0}^N {\tilde{F}_k^{l,m}{D_{nl}}} } \right){\ell _i}\left( {{\xi _n}} \right){\ell _j}\left( {{\eta _m}} \right)}  = {\omega _i}{\omega _j}\left( {\sum\limits_{l = 0}^N {\tilde{F}_k^{l,j}{D_{il}}} } \right),
\label{eq:VolTermForm1}
\end{equation}
so we see that 
\begin{equation}
\int_{E,N} {{(\tilde {F}_k)_\xi }\varphi d\xi d\eta }\rightarrow \mmat \dmat \arry{\tilde F}_{k}\mat M.
\label{eq:MDF}
\end{equation}
On the other hand, for the weak formulation,
\begin{equation}
\int_{E,N} {{\tilde {F}_k}{\varphi _\xi }d\xi d\eta }  = \sum\limits_{n,m = 0}^N {{\omega _n}{\omega _m}\tilde {F}_k^{n,m}{{\ell '}_i}\left( {{\xi _n}} \right){\ell _j}\left( {{\eta _m}} \right)}  = \left( {\sum\limits_{n = 0}^N {{\omega _n}{\omega _j}\tilde {F}_k^{n,j}D_{in}^T} } \right),
\end{equation}
so
\begin{equation}
\int_{E,N} {{\tilde {F}_k}{\varphi _\xi }d\xi d\eta } \rightarrow \dmat^{T}\mmat\arry{\tilde F}_{k}\mmat,
\end{equation}
with a similar result for the $\eta$ direction flux.

The boundary quadratures for the weak form, (\ref{genericIntFormWeak2}), are
\begin{equation}
\begin{split}
\oint_{\partial E,N} {\left( {\tilde F_k^*,\tilde G_k^*} \right) \cdot \hat nd\xi \varphi d\eta }  &= \int_{ - 1,N}^1 {\tilde F_k^*\left( {1,\eta } \right)\varphi d\eta }  - \int_{ - 1,N}^1 {\tilde F_k^*\left( { - 1,\eta } \right)\varphi d\eta }  \\&\quad+ \int_{ - 1,N}^1 {\tilde G_k^*\left( {\xi ,1} \right)d\xi }  - \int_{ - 1,N}^1 {\tilde G_k^*\left( {\xi , - 1} \right)\varphi d\xi. } 
\end{split}\end{equation}
Each term reduces to pointwise values in the same way. The first term, for instance, is
\begin{equation}
\int_{ - 1,N}^1 {\tilde F_k^*\left( {1,\eta } \right)\varphi d\eta }  = \int_{ - 1,N}^1 {\tilde F_k^*\left( {1,\eta } \right){\ell _i}\left( 1 \right){\ell _j}\left( \eta  \right)d\eta }  = \omega_{j}{\left( {\tilde F_k^*} \right)^{N,j}}.
\end{equation}
We represent the numerical fluxes $\mat{\tilde{F}}_k^*$ and $\mat{\tilde{G}}_k^*$ consistent to the 2D array notation \eqref{2DArrayStorage} this time so that the only non-zero entries correspond to the respective interfaces, i.e. 
\begin{equation}
\begin{aligned}
\label{NumFluxMatrixForm}
&\mat{\tilde{F}_k^*} :=\begin{pmatrix}
(\tilde{F}_k^*)^{0,0} & \cdots & (\tilde{F}_k^*)^{0,N}\\
0 & \cdots & 0 \\
\vdots &  & \vdots \\
0 & \cdots & 0 \\
(\tilde{F}_k^*)^{N,0} & \cdots & (\tilde{F}_k^*)^{N,N}
\end{pmatrix},\quad
\mat{\tilde{G}^*} :=\begin{pmatrix}
(\tilde{G}^*_k)^{0,0} & 0 & \cdots & 0 & (\tilde{G}^*_k)^{0,N}\\
\vdots & \vdots &  & \vdots & \vdots\\
(\tilde{G}^*_k)^{N,0} & 0 & \cdots & 0 & (\tilde{G}^*_k)^{N,N}
\end{pmatrix},
\end{aligned}
\end{equation}
so that $\mat{\tilde{F}}_k^*$ only appears at $\xi=\pm1$ and $\mat{\tilde{G}}_k^*$ appears at $\eta = \pm1$ on the reference element $E$. 

To write the approximations in a compact form, we define the matrix operators 
\begin{equation}
\begin{aligned}
\mat{\hat{D}} &:= -\mmat^{-1}\dmat^T\mmat & \text{ (scaled derivative matrix)},\\
\mat{S} &:= \textrm{diag}\left(\frac{1}{\omega_0},0,\ldots,0,-\frac{1}{\omega_N}\right) & \text{ (surface matrix)},
\end{aligned}
\end{equation}
with the derivative matrix $\mat{D}$ defined in \eqref{DmatDef}. Then we can rewrite the equations for the nodal degrees of freedom of the DGSEM approximations (\ref{genericIntFormWeak2}) and (\ref{genericIntFormStrong2}) in the algebraically equivalent forms
\begin{equation}
\label{genericWeakDGSEM}
\arry{J}\circ(\arry{W}_k)_t + \hat{\dmat}\,\arry{\tilde{F}}_k + \arry{\tilde{G}}_k\,\hat{\dmat}^T = \mat{S}\,\arry{\tilde{F}}_k^* + \arry{\tilde{G}}_k^*\,\mat{S},
\end{equation}
and
\begin{equation}
\label{genericStrongDGSEM}
\mat{J}\circ(\mat{W}_k)_t + \dmat\,\mat{\tilde{F}}_k + \mat{\tilde{G}}_k\,\dmatT = \mat{S}\left(\mat{\tilde{F}}_k^*-\mat{\tilde{F}}_k\right) + \left(\mat{\tilde{G}}_k^*-\mat{\tilde{G}}_k\right)\mat{S},
\end{equation}
for $k=1,2,3$. We note that to obtain the results \eqref{genericWeakDGSEM} and \eqref{genericStrongDGSEM} we multiplied by the inverse of $\mmat$ on the left and right.

\subsection{An equivalent subcell flux differencing form}\label{sec:FluxDiff}

 The most important property the DGSEM operators constructed with the LGL quadrature nodes have is that they are summation-by-parts (SBP) operators for all polynomial orders, (\ref{eq:SBPIntegral}). This property can be represented in the form of SBP-SAT finite difference operators \cite{gassner_skew_burgers}, and we collect relevant results here in Lemma \ref{Lem:Lemma1}.
 
 \begin{Lem}[SBP-Properties] \label{Lem:Lemma1}Let the matrix $\qmat:=\mmat\dmat$, which represents the mass weighed derivative, as seen in (\ref{eq:MDF}).  The matrix $\qmat$ has the SBP-property 
 \[
 \qmat + \qmat^T=\mat{B}:=\textrm{diag}(-1,0,\ldots,0,1)\quad (SBP).
 \]
 Furthermore, the SBP-property can be used to obtain alternative expressions for the derivative matrix
\begin{equation}
\label{DGidentities}
\begin{aligned}
&\dmat=\mmat^{-1}\qmat = \mmat^{-1}(\mat{B}-\qmat^T)=-\smat-\mmat^{-1}\qmat^T,\\
&\dmatT=(\mmat^{-1}\qmat)^T =  -\smat-\qmat\mmat^{-1}.
\end{aligned}
\end{equation}
So, the derivative matrix $\hat{\dmat}$ of the weak DG formulation
\begin{equation}
\label{DefDHat}
\hat{\dmat}=-\mmat^{-1}\qmat^T,
\end{equation}
satisfies the relations
\begin{equation}\label{IntByPartsDisc}
\begin{aligned}
&\dmat=-\smat+\hat{\dmat},\\
&\dmatT=  -\smat+\hat{\dmat}^T.
\end{aligned}
\end{equation}
\end{Lem}
\begin{proof}
See, for example, \cite{kreiss1,sbp1,gassner_skew_burgers,carpenter_esdg,gassner_kepdg}.
\end{proof}
\begin{rem}
Although the finite difference and spectral element approximations differ (e.g. the spectral element approximate solution is known everywhere, including in between the nodes), the fact that the nodal equations can be written in the same form will allow us to simply use results proved for SBP finite difference methods as needed.
\end{rem}

The LGL-based DGSEM operators listed in Lemma  \ref{Lem:Lemma1} are in the sub-class of SBP operators with diagonal norm matrix $\mat{M}$. For this class of diagonal norm SBP operators, Fisher and Carpenter \cite{fisher2013} proved an astounding relationship: Such operators can always be re-written into an algebraically equivalent subcell finite volume type differencing formulation. As an example we rewrite the contravariant flux in the $x-$direction into a telescoping flux form (the contravariant flux in the $y-$direction has an analogous form):
\begin{equation}
\dmat{\arry{\tilde{F}}}_k = \mmat^{-1}\qmat{\arry{\tilde{F}}}_k = \diffmat\overline{\arry{\tilde{F}}}_k,
\end{equation}
where $\boldsymbol{\Delta}$ is the $N\times N+1$ differencing matrix 
\begin{equation}
\boldsymbol{\Delta} = 
\begin{pmatrix}
-1 &  1 &  0 & 0 & 0 & 0 \\
 0 & -1 &  1 & 0 & 0 & 0 \\
 0 &  0 & \ddots & \ddots & 0 & 0 \\
 0 &  0 & 0 & -1 & 1 & 0 \\
 0 & 0 & 0 & 0 & -1 & 1
\end{pmatrix}.
\end{equation}
The new flux functions, denoted with an overbar, can be viewed as subcell fluxes on a complementary staggered subcell grid \cite{fisher2013,skew_sbp2}. The contravariant flux functions on the complementary grid remain consistent and high-order when they are computed according to \cite[Section 4.5 and Appendix A.3]{fisher2012}
\begin{equation}\label{highOrderFluxCurvilinear}
\left\{
\begin{aligned}
\bar{\tilde{F}}_k^{0,j} &= \tilde{F}_k^{0,j}, \\
\bar{\tilde{F}}_k^{i,j} &= \sum_{m=i}^{N}\sum_{\ell=0}^{i-1}2\,Q_{\ell m}\left(F_k^{vol}(\statevec W^{\ell,j}, \statevec W^{m,j})\average{Ja^1_1}_{(\ell,m),j}+G_k^{vol}(\statevec W^{\ell,j}, \statevec W^{m,j})\average{Ja^1_2}_{(\ell,m),j}\right),\quad i = 1,\ldots,N,\\
\bar{\tilde{F}}_k^{N+1,j} &= \tilde{F}_k^{N,j},
\end{aligned}
\right.
\end{equation}
for $k=1,2,3$, a fixed point $j$ in the $y-$direction and for some two point, symmetric flux functions $F_k^{vol}$ and $G_k^{vol}$, e.g.,
\begin{equation}
F_k^{vol}(\statevec W^{\ell,j}, \statevec W^{m,j}) = F_k^{vol}(\statevec W^{m,j}, \statevec W^{\ell,j}).
\end{equation}
The $\tilde{F}_k^{0,j}$ and $\tilde{F}_k^{N,j}$ are the typical contravariant fluxes. The metric terms are given in \eqref{metricTerms}, and the arithmetic mean is defined as
\begin{equation}
\average{\cdot}_{(\ell,m),j} := \half\left((\cdot)^{\ell,j}+(\cdot)^{m,j}\right). 
\end{equation}

For computational efficiency we generalise a previous result of Fisher and Carpenter \cite[Eq. (3.13)]{fisher2013} and rewrite the flux differencing volume term on curvilinear meshes.
\begin{Prop}[Flux Diffferencing with Metric Terms]\label{prop1}
One can use the structure of the SBP matrix $\qmat$ and the fluxes on the complimentary grid to eliminate one of the sums in \eqref{highOrderFluxCurvilinear} to write the flux difference in the $x-$direction, $\diffmat\overline{\mat{\tilde{F}}}_k$, in the indicial form 
\begin{equation}\label{eq:fluxDiffForm}
\frac{\bar{\tilde{F}}^{{i+1},j}_k - \bar{\tilde{F}}_k^{i,j}}{\omega_i} = \frac{1}{\omega_i}\sum_{m=0}^N2\,Q_{i m}\left(F_k^{vol}(\statevec{W}^{i,j},\statevec{W}^{m,j})\average{Ja^1_1}_{(i,m),j}+G_k^{vol}(\statevec{W}^{i,j},\statevec{W}^{m,j})\average{Ja^1_2}_{(i,m),j}\right),
\end{equation}
for $ i = 0,\ldots,N$. The flux difference in the $y$-drection, $\overline{\mat{\tilde{G}}}_k\boldsymbol{\Delta}^T\mmat^{-1}$, can be expressed in a similar indicial form
\begin{equation}\label{eq:fluxDiffFormG}
\frac{\bar{\tilde{G}}^{i,{j+1}}_k - \bar{\tilde{G}}_k^{i,j}}{\omega_j} = \frac{1}{\omega_j}\sum_{m=0}^N2\,Q_{j m}\left(F_k^{vol}(\statevec{W}^{i,j},\statevec{W}^{i,m})\average{Ja^2_1}_{i,(j,m)}+G_k^{vol}(\statevec{W}^{i,j},\statevec{W}^{i,m})\average{Ja^2_2}_{i,(j,m)}\right),
\end{equation}
for $ j = 0,\ldots,N$.
\end{Prop}
\begin{proof}
The details to derive \eqref{eq:fluxDiffForm} and \eqref{eq:fluxDiffFormG} are straightforward and can be found in \ref{sec:fluxDiffFormProof}.
\end{proof}

With the definition of the flux differencing formulation complete, we must select the specific form for the volume fluxes $\statevec{F}^{vol}$ and $\statevec{G}^{vol}$. Depending on the choice of the volume flux it is possible for the flux differencing scheme to recover discretisations of alternative split forms of the PDE. The split form of a PDE, often found by averaging the advective and conservative form of the equations, is known to have stabilisation properties for non-linear PDE discretisations \cite{blaisdell1996effect,ducros2000,kennedy2008}. But, due to their form, it is often unclear if the approximation remains globally conservative in the sense of Lax-Wendroff. We see from the telescoping flux difference formulations \eqref{eq:fluxDiffForm} and \eqref{eq:fluxDiffFormG} that it is trivial to show conservation of the scheme while maintaining the flexibility and positive stabilisation properties of the split form. 

In particular, a split form for the derivative of a product of two quantities is
\begin{equation}\label{eq:two_split}
\begin{split}
(a\,b)_x &= \frac{1}{2}\,(a\,b)_x + \frac{1}{2}\left(a_x\,b + a\,b_x\right),
\end{split}
\end{equation}
To demonstrate the split form property of the flux differencing form we select the form for the first component of the volume flux, say $F_1^{vol}(\statevec{W}^{i,j},\statevec{W}^{m,j})$, to be 
\begin{equation}\label{eq:anAverage}
F_1^{vol}(\statevec{W}^{i,j},\statevec{W}^{m,j}) = \half\left((W_1)^{i,j} +(W_1)^{m,j} \right) = \average{W_1}_{(i,m),j}.
\end{equation}
We substitute \eqref{eq:anAverage} into the first term on the right of \eqref{eq:fluxDiffForm} and after some manipulation obtain
\begin{equation}
\label{ExtensionEq}
\begin{aligned}
\frac{1}{\omega_i}\sum_{m=0}^N2\,Q_{i m}\average{W_1}_{(i,m),j}\average{Ja^1_1}_{(i,m),j} &= \frac{1}{2\omega_i}\left(\sum_{m=0}^N Q_{im}(W_1)^{m,j}(Ja_1^1)^{m,j} + (W_1)^{i,j}\sum_{m=0}^NQ_{im}(Ja_1^1)^{m,j}\right.\\
&\qquad\qquad+\left.(Ja_1^1)^{i,j}\sum_{m=0}^NQ_{im}(W_1)^{m,j}\right),
\end{aligned}
\end{equation}
which is the $i-$th row for the discretisation of \eqref{eq:two_split}
\begin{equation}\label{eq:fullSplitForm}
\left(Ja_1^1W_1\right)_x\approx \half\left(\dmat\left(\mat{Ja}_1^1\circ\mat{W}_1\right) + \mat{Ja}_1^1\circ\dmat\mat{W}_1 + \mat{W}_1\circ\dmat\mat{Ja}_1^1\right).
\end{equation}
That is, the product of two averages in \eqref{eq:fluxDiffForm} creates a discretisation of the standard split form of a quadratic product. In a similar fashion the product of three averages creates a discretisation of the standard split form of a triple product proposed by Kennedy and Gruber \cite{kennedy2008}. This is a remarkable property of the flux differencing form \eqref{eq:fluxDiffForm}. By inserting an arithmetic mean or products of arithmetic means into the flux differencing scheme creates a discrete version of a particular split form of the equation. Complete details and proofs of this property of \eqref{eq:fluxDiffForm} can be found in \cite{gassner2016}.

So, rewriting the volume contributions of the DGSEM into the flux differencing form grants us additional flexibility to construct an approximation that discretises alternative split forms of the PDE. By using the alternative split form of the shallow water equations it is possible to create an entropy conservative numerical approximation \cite{gassner2015}. This gives us the motivation to select the internal volume fluxes in such a way that the total energy of the numerical scheme will be conserved discretely. We note that the only alteration needed to change a standard DGSEM code to an entropy stable one is to change the volume contributions to the flux differencing form and select appropriate volume fluxes, see \ref{sec:efficiency}.

To summarize, we have rewritten the volume contributions of the standard strong form DGSEM \eqref{genericStrongDGSEM} into an equivalent flux differencing framework
\begin{equation}\label{fluxDiffStrong}
\mat{J}\circ(\arry{W}_k)_t + \diffmat\overline{\mat{\tilde{F}}}_k+ \overline{\mat{\tilde{G}}}_k\boldsymbol{\Delta}^T\mmat^{-1} = \mat{S}\left(\mat{\tilde{F}}_k^*-\arry{\tilde{F}}_k\right) + \left(\arry{\tilde{G}}_k^*-\arry{\tilde{G}}_k\right)\mat{S},\quad k =1,2,3,
\end{equation}
where the flux differencing is computed in each direction according to the form \eqref{eq:fluxDiffForm} or \eqref{eq:fluxDiffFormG}. The flux differencing representation guarantees that with a chosen pair of symmetric volume fluxes the approximation \eqref{fluxDiffStrong} remains high-order and conservative. Additionally, if the volume fluxes in \eqref{eq:fluxDiffForm} and the surface fluxes in \eqref{genericStrongDGSEM} are carefully constructed then the approximation is also provably entropy conservative \cite{fisher2013}. Note that one can independently select an entropy conservative volume flux that has a different form than an entropy conservative surface flux. We will show in the next section that the additional flexibility to select different fluxes in the volume and at the surface allow the construction of an entropy conservative approximation for the shallow water equations that is also well-balanced.

\section{Entropy stable DGSEM for the 2D shallow water equations}\label{sec:ECDGSEM2D}

In this section, we construct the volume and surface fluxes, following the ideas of Tadmor, e.g. \cite{tadmor1984,tadmor2003}, that discretely conserve the entropy. We denote the entropy conservative approximation as the \emph{ECDGSEM}. Entropy conservation is only valid, however, for smooth solutions and not discontinuous ones (shocks). 
In Sec. \ref{sec:ent-stab}, we use the entropy conserving scheme as the starting point and add dissipation in a controlled way to guarantee that entropy is always dissipated at shocks, resulting in an entropy stable approximation denoted by \emph{ESDGSEM}.

\subsection{ECDGSEM on curvilinear meshes}\label{sec:ECCurvilinear}
The flux differencing formulation of the strong form DGSEM on curvilinear meshes \eqref{fluxDiffStrong} provides a compact notation for the ECDGSEM for which the previous results of Fisher and Carpenter \cite{fisher2013} for conservation and entropy conservation apply. To ensure that the approximation remains well-balanced we use the extra flexibility of the flux differencing form that allows us to select different entropy conservative volume and surface fluxes. The discretisation of the source term is also divided into volume and surface contributions. 
The surface parts depend on averages, $\average{\cdot}$, and jumps, $\jump{\cdot}$, across the interface. We denote the jumps for an arbitrary nodal quantity $\mathcal{W}$ consistent to notation \eqref{2DArrayStorage}: 
\begin{equation}
\begin{aligned}
\label{SourceJmpApproxMatrices}
&\jump{\mat{\mathcal{W}}}_\xi :=\begin{pmatrix}
\jump{\mathcal{W}}_\xi^{0,0} &  \cdots & \jump{\mathcal{W}}_\xi^{0,N} \\
0 & \cdots & 0 \\
\vdots &  & \vdots \\
0 & \cdots & 0 \\
\jump{\mathcal{W}}_\xi^{N,0} &  \cdots & \jump{\mathcal{W}}_\xi^{N,N}
\end{pmatrix},
&\jump{\mat{\mathcal{W}}}_\eta :=\begin{pmatrix}
\jump{\mathcal{W}}_\eta^{0,0} & 0 & \cdots & 0 & \jump{\mathcal{W}}_\eta^{0,N}\\
\vdots & \vdots &  & \vdots & \vdots\\
\jump{\mathcal{W}}_\eta^{N,0} & 0 & \cdots & 0 & \jump{\mathcal{W}}_\eta^{N,N}
\end{pmatrix},
\end{aligned}
\end{equation}
where we distinguish between jumps in $\xi$ direction, $\jump{\mathcal{W}}_\xi$, and $\eta$ direction, $\jump{\mathcal{W}}_\eta$. The averages of a nodal quantity are defined analogously. Whereas the local average operators are symmetric and hence don't prefer a direction, we define the local jumps according to the $\xi$ and $\eta$ coordinate directions: at the $\xi=-1$ and the $\eta=-1$ interfaces we compute the local jumps as the ``inner'' value minus the ``outer'' value, whereas for the $\xi=1$ and $\eta=1$ interfaces the local jumps are computed as the ``outer'' value minus the ``inner'' value. 

\begin{thm}[Curvilinear ECDGSEM]\label{thm:ECDGSEM}
The semi-discrete flux difference form of the two dimensional ECDGSEM for the shallow water equations on curvilinear grids
%
%
\begin{equation}\label{Eq:CurvilinearECDGSEM}
\mat{J}\circ(\mat{W}_k)_t + \diffmat\overline{\mat{\tilde{F}}}_k+ \overline{\mat{\tilde{G}}}_k\boldsymbol{\Delta}^T\mmat^{-1} = \mat{S}\left(\mat{\tilde{F}}_k^{*,ec}-\mat{\tilde{F}}_k\right) + \left(\mat{\tilde{G}}_k^{*,ec}-\mat{\tilde{G}}_k\right)\mat{S}+\mat{source}_k,\quad k =1,2,3,
\end{equation}
where for the flux differencing components \eqref{eq:fluxDiffForm} and \eqref{eq:fluxDiffFormG}, we use the entropy conserving volume fluxes 
\begin{equation}\label{volFlux2D} 
\begin{aligned}
\statevec{F}^{vol}(\statevec{W}^{i,j},\statevec{W}^{m,j}) &= \begin{pmatrix}
\average{hu}_{(i,m),j} \\[0.1cm]
\average{hu}_{(i,m),j}\average{u}_{(i,m),j} +g\average{h}_{(i,m),j}^2-\half g \average{h^2}_{(i,m),j} \\[0.1cm]
\average{hu}_{(i,m),j}\average{v}_{(i,m),j}
\end{pmatrix},\\
\\
\statevec{G}^{vol}(\statevec{W}^{i,j},\statevec{W}^{i,m}) &= \begin{pmatrix}
\average{hv}_{i,(j,m)} \\[0.1cm] 
\average{hv}_{i,(j,m)}\average{u}_{i,(j,m)} \\[0.1cm]
\average{hv}_{i,(j,m)}\average{v}_{i,(j,m)} +g\average{h}_{i,(j,m)}^2-\half g \average{h^2}_{i,(j,m)}
\end{pmatrix},
\end{aligned}
\end{equation}
in combination with the entropy conserving surface numerical fluxes
\begin{equation}\label{eq:surfaceFluxes}
\begin{aligned}
\statevec{F}^{*,ec}(\statevec{W}^{+},\statevec{W}^{-}) = \begin{pmatrix}
\average{h}\average{u}\\[0.1cm]
\average{h}\average{u}^2 + \frac{1}{2}\,g\,\average{h^2}\\[0.1cm]
\average{h}\average{u}\average{v}
\end{pmatrix},\\[0.2cm]
\statevec{G}^{*,ec}(\statevec{W}^{+},\statevec{W}^{-}) = \begin{pmatrix}
\average{h}\average{v}\\[0.1cm]
\average{h}\average{u}\average{v}\\[0.1cm]
\average{h}\average{v}^2 + \frac{1}{2}\,g\,\average{h^2}
\end{pmatrix},
\end{aligned}
\end{equation}
where in (\ref{eq:surfaceFluxes}) the $\average{\cdot}$ indicates the average of the two neighbouring states $\statevec{W}^+$ and $\statevec{W}^{\_}$ with the source term discretisation
\begin{equation}\label{eq:sourceTerm}
\resizebox{1.0\textwidth}{!}{$
\mat{source} = \begin{pmatrix}
0\\[0.1cm]
-\frac{g}{2}\hmat\circ\left[\mat{y}_{\eta}\circ\dmat\mat{b} + \dmat(\mat{y}_{\eta}\circ\mat{b}) - \mat{y}_\xi\circ\mat{b}\dmatT - (\mat{y}_\xi\circ\mat{b})\dmatT\right] + \frac{g}{2}\mmat^{-1}\left( \mat{y}_{\eta}\circ\average{\mat{h}}_{\xi}\circ\jump{\mat{b}}_{\xi} \right) - \frac{g}{2}\left( \mat{y}_{\xi}\circ\average{\mat{h}}_{\eta}\circ\jump{\mat{b}}_{\eta} \right)\mmat^{-1}\\[0.1cm]
-\frac{g}{2}\hmat\circ\left[-\mat{x}_{\eta}\circ\dmat\mat{b} - \dmat(\mat{x}_{\eta}\circ\mat{b}) + \mat{x}_\xi\circ\mat{b}\dmatT + (\mat{x}_\xi\circ\mat{b})\dmatT\right]-\frac{g}{2}\mmat^{-1}\left( \mat{x}_{\eta}\circ\average{\mat{h}}_{\xi}\circ\jump{\mat{b}}_{\xi} \right) + \frac{g}{2}\left( \mat{x}_{\xi}\circ\average{\mat{h}}_{\eta}\circ\jump{\mat{b}}_{\eta} \right)\mmat^{-1}
\end{pmatrix},$}
\end{equation}
has the following properties:
   \begin{enumerate}[label={\thethm.\arabic*}]
       \item Discrete conservation of the mass and discrete conservation of the momentum if the bottom topography is constant.\label{Thm2partone}
       \item Discrete conservation of the total energy, which is an entropy function for the shallow water equations. Hence it preserves the entropy of the system. \label{Thm2parttwo}
       \item Discrete well-balanced property for arbitrary bottom topographies. \label{Thm2partthree}
   \end{enumerate}
\end{thm}
\begin{proof}
We prove the result in three parts.
\begin{proofpart}~~\ref{Thm2partone}: The discrete conservation of the numerical scheme follows directly from the telescoping flux differencing form of the approximation \cite{fisher2013}.
\end{proofpart}

\begin{proofpart}~~\ref{Thm2parttwo}: If the volume and surface fluxes are provably entropy conservative then the global entropy conservation of the approximation is retained by the flux differencing form \cite{fisher2013}. The surface fluxes \eqref{eq:surfaceFluxes} are known to be entropy conserving \cite{fjordholm2011}. We demonstrate in \ref{sec:ec_proof}  that the volume fluxes \eqref{volFlux2D} are also entropy conservative.
\end{proofpart}

\begin{proofpart}~~\ref{Thm2partthree}:  The complete proof of well-balancedness for the curvilinear ECDGSEM is provided in \ref{sec:wb_proof}. However, it is necessary to describe the construction of the bottom topography discretisations. For the two dimensional problem in general coordinates we require discrete approximations for $b_x$ and $b_y$, or more compactly, $\nabla b$. From \eqref{PhysicalDerivativeGeneral} we know the explicit form of the gradient in computational coordinates is
\begin{equation}\label{gradB}
J\nabla b = \left[y_{\eta}\pderivative{b}{\xi} - y_{\xi}\pderivative{b}{\eta}\;,\;-x_\eta\pderivative{b}{\xi} + x_\xi\pderivative{b}{\eta}\right]^T.
\end{equation}

We treat each piece of the source term as a quadratic split form \eqref{eq:two_split}. For example the first component from \eqref{gradB} is
\begin{equation}\label{eq:splitFormB}
b_x = \half\left(\pderivative{(y_{\eta}b)}{\xi} + y_{\eta}\pderivative{b}{\xi} +b\pderivative{y_{\eta}}{\xi} \right) -\half\left(\pderivative{(y_{\xi}b)}{\eta} + y_{\xi}\pderivative{b}{\eta} +b\pderivative{y_{\xi}}{\eta} \right).
\end{equation}
We then approximate the bottom contribution from the discrete approximation of the quadratic split form \eqref{eq:fullSplitForm}. So, we have the approximations for the derivative of the bottom topography
\begin{equation}\label{eq:botdiscretisation}
\begin{aligned}
b_x&\approx \half\left(\dmat(\mat{y}_{\eta}\circ\mat{b}) + \mat{y}_{\eta}\circ\dmat\mat{b} + \mat{b}\circ\dmat\mat{y}_\eta - (\mat{y}_\xi\circ\mat{b})\dmatT - \mat{y}_\xi\circ\mat{b}\dmatT - \mat{b}\circ\mat{y}_\xi\dmatT\right), \\[0.1cm]
b_y&\approx \half\left(-\dmat(\mat{x}_{\eta}\circ\mat{b})-\mat{x}_{\eta}\circ\dmat\mat{b}-\mat{b}\circ\dmat\mat{x}_{\eta} + (\mat{x}_{\xi}\circ\mat{b})\dmatT+ \mat{x}_{\xi}\circ\mat{b}\dmatT+ \mat{b}\circ\mat{x}_{\xi}\dmatT \right).
\end{aligned}
\end{equation}
The bottom topography discretisation could immediately be treated in a flux differencing way with the result \eqref{eq:fluxDiffForm}. However, we know that the approximation satisfies the metric identities \eqref{metricIDs}. So we cancel extraneous terms in \eqref{eq:botdiscretisation} and obtain a more computationally efficient form of the source term volume contributions in \eqref{Eq:CurvilinearECDGSEM}
\begin{equation}\label{eq:botdiscretisation2}
\begin{aligned}
b_x&\approx \half\left(\dmat(\mat{y}_{\eta}\circ\mat{b}) + \mat{y}_{\eta}\circ\dmat\mat{b} - (\mat{y}_\xi\circ\mat{b})\dmatT - \mat{y}_\xi\circ\mat{b}\dmatT\right), \\[0.1cm]
b_y&\approx \half\left(-\dmat(\mat{x}_{\eta}\circ\mat{b})-\mat{x}_{\eta}\circ\dmat\mat{b} + (\mat{x}_{\xi}\circ\mat{b})\dmatT+ \mat{x}_{\xi}\circ\mat{b}\dmatT \right).
\end{aligned}
\end{equation}
The surface contributions follow from a similar logic in that we require the derivative of $b$ at the boundary (in this case the jump). \end{proofpart}
\end{proof}

\begin{rem}[One-dimensional bottom formulation]
In one space dimension we use the following approximation of the source term for an element $G$
\begin{equation}
\label{JumpApprox1D}
g\,h\,b_x \approx g\,\mat{h}\,\dmat\,\statevec{b} + \half g \average{h}^G_{\xi=-1} \jump{b}^G_{\xi=-1} \frac{1}{\omega_0} \statevec{e}_1 + \half g \average{h}^G_{\xi=+1} \jump{b}^G_{\xi=+1} \frac{1}{\omega_N} \statevec{e}_{N+1}.
\end{equation}
Here $\average{h}^G_{\xi=-1}$ is the average water height at the left interface node of element $G$ and the right interface node of element $G-1$, $\average{h}^G_{\xi=+1}$ is the average water height at the right interface node of element $G$ and the left interface node of element $G+1$:
\begin{equation}
\begin{aligned}
&\average{h}_{\xi=-1}^G = \half(h_0^G + h_N^{G-1}),\\
&\average{h}_{\xi=+1}^G = \half(h_N^G + h_0^{G+1}).
\end{aligned}
\end{equation}
In the same fashion the jump in bottom topography is defined as
\begin{equation}
\begin{aligned}
&\jump{b}_{\xi=-1}^G = b_0^G - b_N^{G-1},\\
&\jump{b}_{\xi=+1}^G = b_0^{G+1}-b_N^G .
\end{aligned}
\end{equation}
The quadrature weights are denoted by $\omega_0$ and $\omega_N$ and the $N+1$ dimensional unit vectors are denoted by $\statevec{e}_1$ and $\statevec{e}_{N+1}$.
\end{rem}

\subsection{Entropy stable DGSEM}\label{sec:ent-stab}
The ECDGSEM presented so far exactly conserves the discrete entropy. However, the solution of the shallow water equations may develop discontinuities (shocks) in finite time even for smooth initial data. We know in the presence of discontinuities that the conservation law for the entropy function \eqref{EnergyEquation2D} must be replaced by the entropy inequality \eqref{EnergyInequality2D} \cite{tadmor1984}. Thus, we must add numerical dissipation to the ECDGSEM so that the entropy function is guaranteed to be dissipated (or conserved for smooth well resolved solutions), thereby ensuring that a discrete version of the entropy inequality holds. 

A typical way to add dissipation in a discontinuous Galerkin approximation is via the definition of the numerical flux function. Most often, those numerical flux functions are (approximate) Riemann solvers that inherently create dissipation at shocks (or where the solution is otherwise underresolved). We follow this basic idea and add dissipation in the spirit of Riemann solvers at the element interfaces to add dissipation to the ECDGSEM. 

To derive the dissipative numerical flux, we note that the physical fluxes \eqref{eq:physicalFluxes} have the associated flux Jacobians
\begin{equation}
\label{fJacobi}
\mat{A}_f=\statevec{f}_{\statevec{w}} =\begin{pmatrix}
0 & 1 & 0\\
gh-u^2 & 2u & 0 \\
-uv& v & u
\end{pmatrix},
\end{equation}
and
\begin{equation}
\label{gJacobi}
\mat{A}_g=\statevec{g}_{\statevec{w}} =\begin{pmatrix}
0 & 0 & 1\\
-uv & v & u \\
gh-v^2& 0 & 2v
\end{pmatrix}.
\end{equation}
The eigenvalues of \eqref{fJacobi} are $u+c$, $u$, $u-c$ and for \eqref{gJacobi} $v+c$, $v$, $v-c$ with the wave speed $c=\sqrt{gh}$. The matrices of eigenvectors of \eqref{fJacobi} and \eqref{gJacobi} are
\begin{equation}
\mat{R}_f =\begin{pmatrix}
1 & 0 & 1\\
u+c & 0 & u-c \\
v& 1 & v
\end{pmatrix},
\end{equation}
and
\begin{equation}
\mat{R}_g =\begin{pmatrix}
1 & 0 & 1\\
u & 1 & u \\
v+c& 0 & v-c
\end{pmatrix},
\end{equation}
respectively.

The dissipation term will also require the entropy Jacobian matrix $\mat{H}=\statevec{q}_{\statevec{w}}$. To obtain $\mat{H}$, we now express the conservative variables $\statevec{w}=(h,hu,hv)^T$ in terms of the entropy variables $\statevec{q}=(g(h+b) - \half u^2 - \half v^2,u,v)^T$ :
\begin{equation}
\begin{aligned}
\label{ConservativeByEntropy}
&w_1=\frac{1}{g}q_1-b+\frac{1}{2g}(q_2^2+q_3^2), \\
&w_2=\frac{1}{g}q_1 q_2-b q_2+\frac{1}{2g}(q_2^3+q_2 q_3^2), \\
&w_3=\frac{1}{g}q_1 q_3-b q_3+\frac{1}{2g}(q_2^2 q_3+q_3^3) .
\end{aligned}
\end{equation}
Differentiating \eqref{ConservativeByEntropy} directly leads to the entropy Jacobian matrix
\begin{equation}
\mat{H} =\frac{1}{g}\begin{pmatrix}
1 & u & v\\
u & gh+u^2 & uv \\
v& uv & gh+v^2
\end{pmatrix}.
\end{equation}

With an appropriate scaling for the right eigenvectors we obtain the set of entropy scaled eigenvectors \cite{merriam1989}
\begin{equation}
\mat{H} = (\mat{R}\mat{T})(\mat{R}\mat{T})^T,
\end{equation}
which relates the right eigenvectors to the entropy Jacobian. For the scaling, we consider the matrix
\begin{equation}
\mat{T}=\textrm{diag}(\sqrt{s_1},\sqrt{s_2},\sqrt{s_3}),
\end{equation}
 with scaling parameters on the diagonal only. We define $\zmat=\mat{T}^2$ and have the identity in a new form
\begin{equation}
\mat{H} = \mat{R}\zmat\mat{R}^T.
\end{equation}
For the eigenvectors of the $f$ flux Jacobian, $\mat{R}_f$, we find
\begin{equation}
\begin{aligned}
&s_1=\frac{1}{2g}, \quad
&s_2=h, \quad
&s_3=\frac{1}{2g} .
\end{aligned}
\end{equation}
A straightforward calculation shows that the same scaling works for the eigenvectors of the flux Jacobian in the $y-$direction as well.

Now we have all the necessary components to define the entropy stable numerical flux functions. We subtract the dissipation term required for dissipation in the $x-$direction
\begin{equation}
\label{fStabilized}
\statevec{F}^{*,es} =\statevec{F}^{*,ec} - \half \mat{R}_f  \, \big|\mat{\Lambda}_f \big|\, \zmat \, \mat{R}_f^T \jump{\,\statevec{q}\,},
\end{equation}
and the $y-$direction
\begin{equation}
\label{gStabilized}
\statevec{G}^{*,es} =\statevec{G}^{*,ec} - \half \mat{R}_g \, \big|\mat{\Lambda}_g \big|\, \zmat\, \mat{R}_g^T \jump{\,\statevec{q}\,},
\end{equation}
where $\mat{\Lambda}_f$ and $\mat{\Lambda}_g$ are the diagonal matrices containing the eigenvalues previously computed. We use the arithmetic average values at an element interfaces to compute the right eigenvector, scaling, and eigenvalue matrices in \eqref{fStabilized} and \eqref{gStabilized}. 

It is important that the dissipation terms depend on the jumps of the entropy variables and not on the jump of the conserved quantities as would be common in standard Riemann solver-based numerical flux functions. If we compute the discrete entropy equation by multiplying the scheme with the entropy variables, we get contributions of the form $ -\half\,\jump{\,\statevec{q}\,}\cdot \mat{R}_f  \, \big|\mat{\Lambda}_f \big|\, \zmat \, \mat{R}_f^T \jump{\,\statevec{q}\,}$ at each interface that are guaranteed to be negative due to the positivity of the matrix $\mat{R}_f \, \big|\mat{\Lambda}_f \big|\, \zmat \, \mat{R}_f^T$. Writing in terms of the jumps in the entropy variables ensures that entropy is dissipated when the jump in entropy variables across interfaces is large (e.g. shocks) and is nearly preserved when the jumps are small for well resolved smooth solutions.

We finally present the main contribution of the present work, an an entropy stable DGSEM (ESDGSEM) for the shallow water equations.
\begin{thm}[Curvilinear ESDGSEM]\label{thm:ESDGSEM}
The semi-discrete form of the two dimensional ESDGSEM formulation for the shallow water equations on curvilinear grids is given by
%
%
\begin{equation}\label{Eq:CurvilinearESDGSEM}
\mat{J}\circ(\mat{W}_k)_t + \diffmat\overline{\mat{\tilde{F}}}_k+ \overline{\mat{\tilde{G}}}_k\boldsymbol{\Delta}^T\mmat^{-1} = \mat{S}\left(\mat{\tilde{F}}_k^{*,es}-\mat{\tilde{F}}_k\right) + \left(\mat{\tilde{G}}_k^{*,es}-\mat{\tilde{G}}_k\right)\mat{S}+\mat{source}_k,\quad k =1,2,3,
\end{equation}
where the flux differencing components use the volume fluxes \eqref{volFlux2D} in combination with the entropy stable numerical fluxes \eqref{fStabilized} and \eqref{gStabilized} and the source term discretisation \eqref{eq:sourceTerm}. The approximation has the following properties:
   \begin{enumerate}[label={\thethm.\arabic*}]
       \item Discrete conservation of the mass and discrete conservation of the momentum if the bottom topography is constant.
       \item Discrete entropy stability.
       \item The well-balanced property for arbitrary bottom topographies.
   \end{enumerate}
\end{thm}
\begin{proof}
The ESDGSEM follows directly from the curvilinear ECDGSEM presented in Thm.~\ref{thm:ECDGSEM}. To guarantee entropy stability we replace the entropy conserving numerical fluxes \eqref{eq:surfaceFluxes} at element interfaces with the entropy stable numerical fluxes \eqref{fStabilized}, \eqref{gStabilized}. For the ``lake at rest'' initial conditions the jump in entropy variables is zero, $\jump{\statevec{q}\,}=0$, so the additional dissipation term vanishes and does not affect the well-balanced property of the scheme.
\end{proof}

\section{Numerical results}\label{NumResults}

In this section, we use the entropy conserving \eqref{Eq:CurvilinearECDGSEM} and entropy stable \eqref{Eq:CurvilinearESDGSEM} numerical schemes on several test cases to numerically verify the theoretical findings from Thms.~\ref{thm:ECDGSEM} and \ref{thm:ESDGSEM}. To integrate the systems in time we use the five stage, fourth order Runge-Kutta time integrator of Carpenter and Kennedy \cite{Carpenter&Kennedy:1994}. First, to verify the convergence, conservation and well-balancedness of the approximations we use a structured curvilinear mesh depicted in Figure \ref{fig:PeriodicMesh}. Elements are numbered by counting from left to right and bottom to top. We also simulate a bore-shear interaction as well as the numerical generation of potential vorticity generated from the passage of a non-uniform bore. After all theoretical findings are verified numerically, we present a simulated partial dam break from a parabolic dam with a discontinuous bottom topography in the downstream region of the flow. The partial dam break problem serves to exercise each component of the ESDGSEM approximation.  
\begin{figure}[!ht]
\begin{center}
  \makebox[\textwidth]{\input{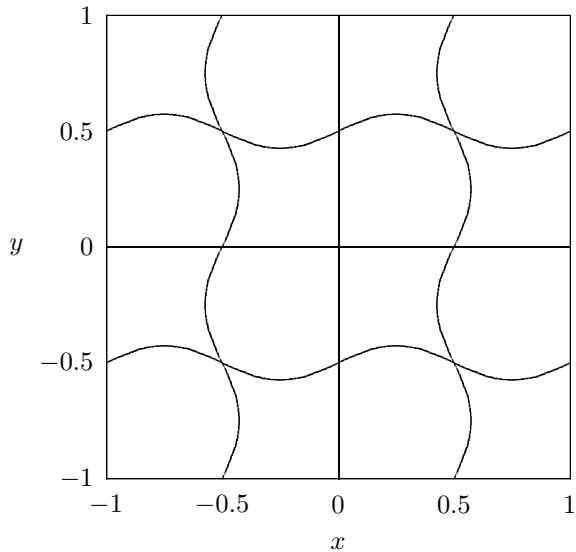}}
\caption{The curvilinear mesh used for verification of convergence, conservation, and well-balancedness.}
\label{fig:PeriodicMesh}
\end{center}
\end{figure}
\subsection{Convergence}\label{sec:convergence}
\label{sec:convTest}
We first test the convergence of the entropy conserving and entropy stable approximations with a smooth solution test problem. We use the method of manufactured solutions to create an analytic solution
\begin{equation}\label{eq:manufacturedSol}
\begin{aligned}
&H(x,y,t)=h(x,y,t)+b_1(x,y)=8+\cos(x)\sin(y)\cos(t),\\
&u(x,y,t)=0.5,\\
&v(x,y,t)=1.5 .
\end{aligned}
\end{equation}
The manufactured solution \eqref{eq:manufacturedSol} introduces additional source terms to the equations of the form
\begin{equation}
\begin{aligned}
&s_1:=H_t+u(H_x-(b_1)_x)+v(H_y-(b_1)_y),\\
&s_2:=uH_t+u^2(H_x-(b_1)_x)+uv(H_y-(b_1)_y)+H_x(H-b_1),\\
&s_3:=vH_t+uv(H_x-(b_1)_x)+v^2(H_y-(b_1)_y)+H_y(H-b_1) ,
\end{aligned}
\end{equation}
where derivatives regarding $H$ and $b_1$ are computed analytically. We solve this problem on the domain $[-1,1]^2$ with the smooth bottom topography
\begin{equation}
\label{ConvBottom}
b_1(x,y)=2+0.5\sin\left(2 \pi x\right)+0.5\cos\left(2 \pi y\right).
\end{equation}
The gravitational constant is set to $g=1$. 

We vary the polynomial degree on the mesh given in Fig. \ref{fig:PeriodicMesh} and observe exponential convergence up to $N=16$ ($N=15$ for ESDGSEM) for $\Delta t=1/2000$ and $N=17$ ($N=16$ for ESDGSEM) for $\Delta t=1/4000$, when the errors introduced by the time integrator become dominant. We present semi-log plots in Fig. \ref{fig:SWSpectral_EC} for the entropy conserving scheme and Fig. \ref{fig:SWSpectral_ES} for the entropy stable scheme. As previously observed, e.g. \cite{skew_sbp2,gassner_skew_burgers,gassner_kepdg,gassner2015}, we find a suboptimal order of convergence for odd polynomial degree $N$ for the purely entropy conserving scheme. However, both the ECDGSEM and ESDGSEM are spectrally accurate for smooth problems.
\begin{figure}[!ht]
   \centering
    {
        \includegraphics[scale=0.65,trim=0 0 0 0, clip]{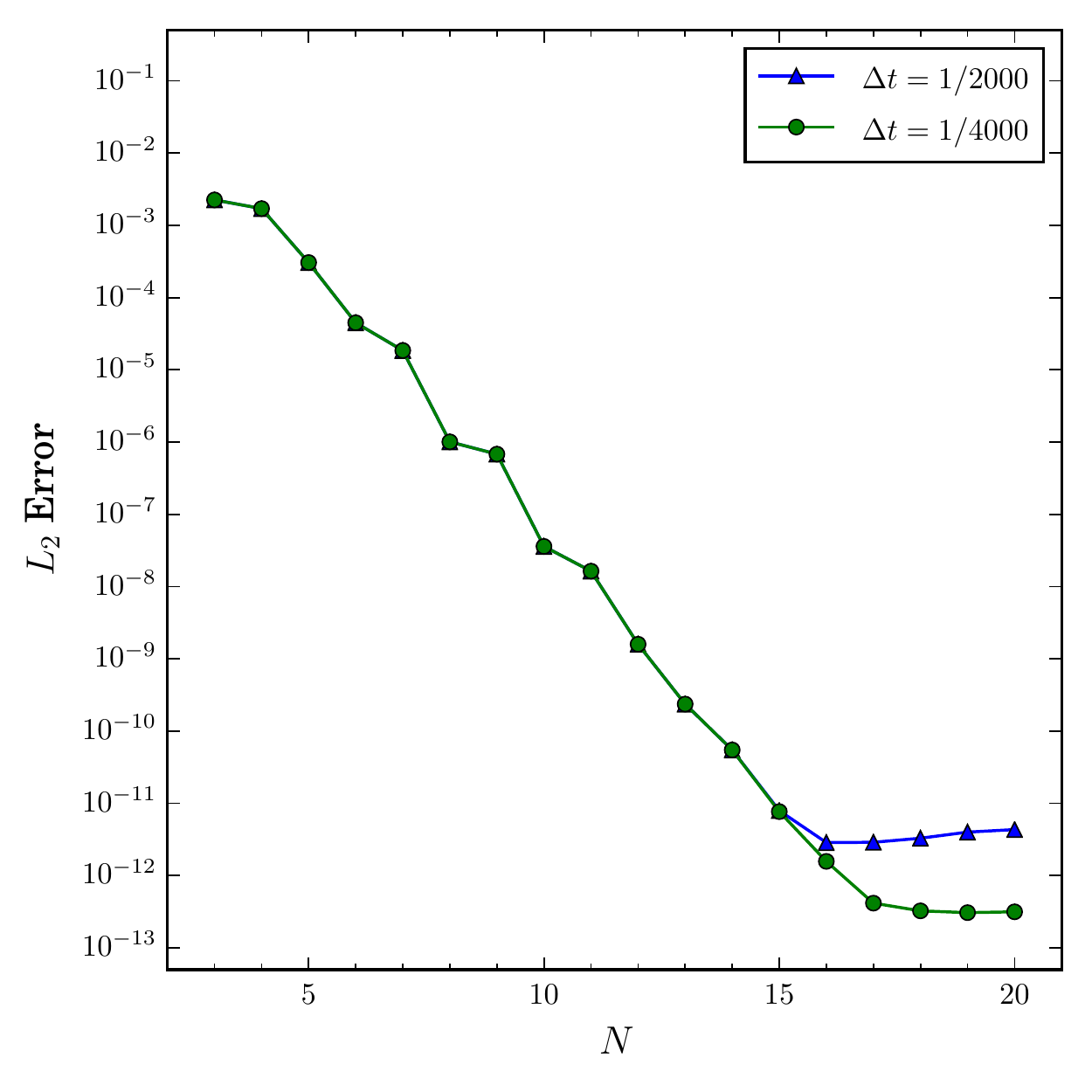}
    }
    \caption{Semi-log plot shows the spectral convergence in space and fourth order accuracy in time for of the ECDGSEM scheme applied to a smooth solution.}
    \label{fig:SWSpectral_EC}
\end{figure}
\begin{figure}[!ht]
   \centering
    {
        \includegraphics[scale=0.65,trim=0 0 0 0, clip]{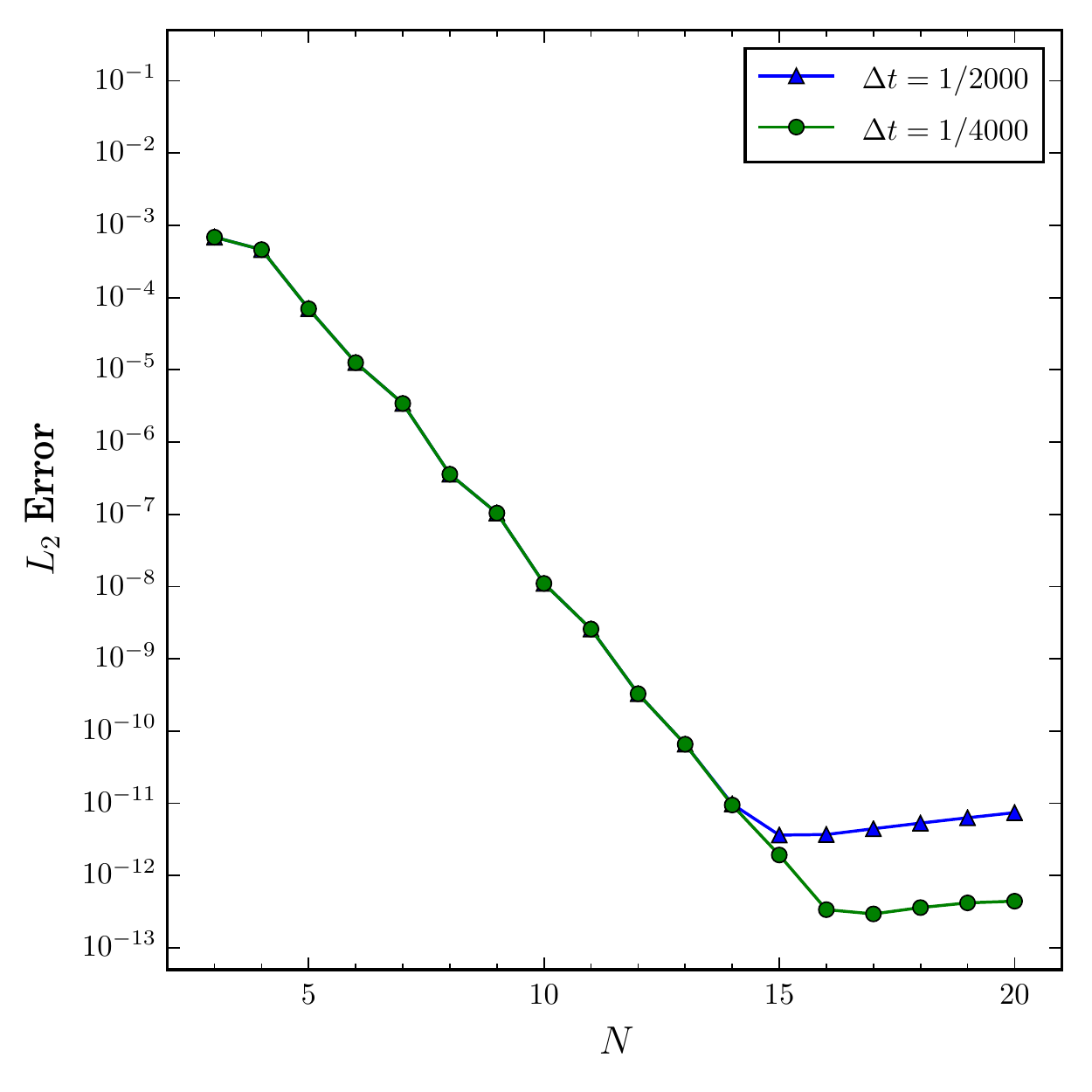}
    }
    \caption{Semi-log plot shows the spectral convergence in space and fourth order accuracy in time for the ESDGSEM scheme applied to a smooth solution.}
    \label{fig:SWSpectral_ES}
\end{figure}
\subsection{Conservation of mass and momentum}
We first numerically verify that mass and momentum are conserved, Property~\ref{thm:ECDGSEM}.1 of Thm.~\ref{thm:ECDGSEM} using a constant bottom topography. Additionally, with the specific numerical volume fluxes \eqref{volFlux2D} and surface fluxes \eqref{eq:surfaceFluxes} the approximation will conserve the total energy (modulo dissipative effects of the time integrator). Also, for a non-constant bottom topography the momentum equations become balance laws, and we show that mass and entropy are still conserved discretely, even for discontinuous bottom topographies. 

The problem that we choose to model is a dam break on the domain $\Omega=[-1,1]^2$. The dam break is initialised along the vertical line $x=0$ on the curvilinear mesh shown in Fig. \ref{fig:PeriodicMesh} with periodic boundary conditions and a polynomial degree of $N=5$. The gravitational constant is again set to $g=1$. The dam break problem uses the initial conditions
\begin{equation}
\label{eq:dambreakCons}
h(x,y,0) = \left\{
\begin{aligned}
&5-b(x,y) \quad\textrm{if }x < 0 \\
&4-b(x,y) \quad\textrm{if }x > 0
\end{aligned}
\right. ,
\quad u(x,y,0) = v(x,y,0) = 0.
\end{equation}
\subsubsection{Dam break over a flat bottom}
We demonstrate the entropy conservative properties of the ECDGSEM scheme, Property~\ref{thm:ECDGSEM}.2 of Thm.~\ref{thm:ECDGSEM}. We consider a flat bottom topography, $b\equiv 0$, and solve the dam break on the curvilinear mesh Fig. \ref{fig:PeriodicMesh}. The differences in mass, momentum, and total energy are listed in Table \ref{tab:ConservationErrors1}. The error in the discrete energy reflects the dissipative influence of the time integrator. Otherwise, we see that the conservation in mass and momentum is on the order of machine precision for each time step value considered. By shrinking the time step we see that we can drive the error in the discrete total energy to machine precision, converging at the fourth order accuracy of the time integrator. 
\begin{table}[!ht]
\begin{center}
\begin{tabular}{ccccccc}
\toprule
$\Delta t$  & $\Delta$Mass & $\Delta$MomentumX & $\Delta$MomentumY & $\Delta$Energy & Temporal Order \\
\toprule
1/1000   & 3.55E-14 & 2.66E-13 & 4.32E-17 & 4.79E-08 & -- \\\midrule
1/2000  & 2.49E-14 & 2.66E-13 & 9.95E-16 & 3.01E-09 & 3.99 \\\midrule
1/4000  & 3.20E-14 & 2.66E-13 & 1.71E-15 & 1.89E-10 & 3.99 \\\midrule
1/8000  & 3.20E-14 & 2.66E-13 & 1.46E-15 & 1.18E-11 & 4.00\\
\bottomrule
\end{tabular}
\end{center}
\caption{Errors in the conserved quantities, mass, momentum, and total energy, at $T=1$ over a constant bottom topography. In the energy conservation results we observe the temporal accuracy of the time integrator.}
\label{tab:ConservationErrors1}
\end{table}%

\subsubsection{Dam break over a discontinuous bump}
Next we examine the conservative properties of the numerical scheme for a discontinuous bottom topography. Momentum will no longer be conserved. But the mass should be conserved and the error in the total energy should reduce as the time step is refined. These properties are demonstrated in the numerical test presented in Table \ref{tab:ConservationErrors2}. 
For the discontinuous bottom topography we used
\begin{equation}
\label{eq:discBottomBox}
b_2(x,y) = \left\{
\begin{aligned}
&2+0.5\sin\left(2 \pi x\right)+0.5\cos\left(2 \pi y\right), \quad\textrm{if $element_{ID}$} = 6 \\
&0, \quad\textrm{otherwise}
\end{aligned}
\right.,
\end{equation}
which is the bottom topography \eqref{ConvBottom} restricted to a single element.
If we use the entropy conserving scheme without added stabilisation and periodic boundaries, we expect the entropy (total energy) to be conserved in the scheme. Table \ref{tab:ConservationErrors2} shows that the error in the total energy shrinks with the fourth order accuracy of the time integrator as the time step is refined. Also, we see that mass is conserved to machine precision for all time steps. 
\begin{table}[!ht]
\begin{center}
\begin{tabular}{cccc}
\toprule
$\Delta t$  & $\Delta$Mass & $\Delta$Energy & Temporal Order\\
\toprule
1/1000  & 5.33E-14 & 2.16E-08 & -- \\\midrule
1/2000  & 1.78E-14 & 1.35E-09 & 4.00\\\midrule
1/4000  & 2.84E-14 & 8.48E-11 & 3.99 \\\midrule
1/8000  & 3.55E-15 & 5.32E-12 & 3.99 \\
\bottomrule
\end{tabular}
\end{center}
\caption{Errors in the mass and total energy at $T=1$ for the discontinuous bottom topography \eqref{eq:discBottomBox}. In the energy conservation results we observe the temporal accuracy of the time integrator.}
\label{tab:ConservationErrors2}
\end{table}%
\subsection{Well-balancedness over a discontinuous bottom}\label{sec:wellBalance}
Next we demonstrate numerically that the curvilinear entropy conserving numerical scheme \eqref{Eq:CurvilinearECDGSEM} is well-balanced, numerically demonstrating Property~\ref{thm:ECDGSEM}.3 of Thm.~\ref{thm:ECDGSEM}. We focus, particularly, on a discontinuous bottom topography. So, we configure a ``lake at rest'' test problem as in \eqref{lakeAtRest}
\begin{equation}
\begin{aligned}
&h+b_2(x,y)=5,\\
&u=v=0,
\end{aligned}
\end{equation}
with the discontinuous bottom topography \eqref{eq:discBottomBox}. The boundary conditions are set to be periodic. We use the curvilinear mesh in Fig. \ref{fig:PeriodicMesh} on the domain $\Omega=[-1,1]^2$ and vary the polynomial degree $N$. The time step is fixed at $\Delta t = 1/1000$. Table \ref{tab:L2ErrorsLakeAtRest} shows that the $L_2$-error of the approximate total water height, $H=h+b_2$, is of the magnitude of round-off errors for both the ECDGSEM and the ESDGSEM.
\begin{table}[!ht]
\begin{center}
\begin{tabular}{ccc}
\toprule
N & $L_2$-error ECDGSEM & $L_2$-error ESDGSEM  \\
\toprule
3 & 8.84E-15 & 5.37E-15 \\\midrule
4 & 8.75E-15 & 5.02E-15 \\\midrule
5 & 1.85E-14 & 1.55E-14 \\
\bottomrule
\end{tabular}
\end{center}
\caption{$L_2$-error of the approximate total water height, $H=h+b$, to the ``lake at rest'' problem on the curvilinear mesh shown in Fig. \ref{fig:PeriodicMesh} at $T=1$.}
\label{tab:L2ErrorsLakeAtRest}
\end{table}%

\subsection{Dam break over a discontinuous bump}\label{sec:breakOverBump}
Next we compute the solution of a dam break problem using both the ECDGSEM and the ESDGSEM. This numerical example demonstrates that the entropy stable approximation removes spurious oscillations in the post-shock regions of the flow introduced by the entropy conservative approximation. For the numerical test we set inflow/outlow Dirichlet boundaries along the vertical lines $x = 0$ and $x=10$ and periodic boundaries along the horizontal lines $y=0$ and $y=10$ on the domain $\Omega = [0,10]^2$. We use a rectangular mesh with sizes varying from $20\times 20$, $40\times 40$, $80\times 80$ to $160\times 160$ elements and polynomial degree of $N=4$. The gravitational constant is again set to $g=1$. The initial conditions are
\begin{equation}
\label{eq:dambreakDisc}
h(x,y,0) = \left\{
\begin{aligned}
&3.5-b_3(x,y) \quad\textrm{if }x < 5 \\
&2.5-b_3(x,y) \quad\textrm{if }x > 5
\end{aligned}
\right. ,
\quad u(x,y,0) = v(x,y,0) = 0,
\end{equation}
so we initialise the problem with a height discontinuity on the element interfaces at $x=5.0$.
As a bottom topography we use
\begin{equation}
\label{eq:discBottomBox2}
b_3(x,y) = \left\{
\begin{aligned}
&2.0-(x-5)^2-(y-5)^2, \quad\textrm{if }|x-5| < 1 \textrm{ and } |y-5| < 1 \\
&0, \quad\textrm{otherwise}
\end{aligned}
\right.,
\end{equation}
which is a box with a smooth top that has its center at $(5.0,5.0),$ side lengths of $2,$ and is initialised discontinuously along the edges of the box, which align with cell interfaces.

The results shown in Fig. \ref{fig:DamBreakOverDiscBumpEC} for the purely conserving scheme show that it produces severe ringing in the post-shock region of the approximation. The entropy stable approximation removes these spurious oscillations except near the discontinuity at the shock front. We show the computed entropy stable solution in Fig. \ref{fig:DamBreakOverDiscBumpES}, where we present a grid refinement study for the entropy stable approximation. It is clear that the additional dissipation smoothes the spurious oscillations and the ESDGSEM provides a more physical solution to the dam break problem over the discontinuous bump. Nevertheless, we note that, although stable, the ESDGSEM is not completely oscillation-free \cite{fjordholm2011}.
\begin{figure}[!ht]
   \centering
    {
        \includegraphics[scale=0.25,trim=0 10 0 10, clip]{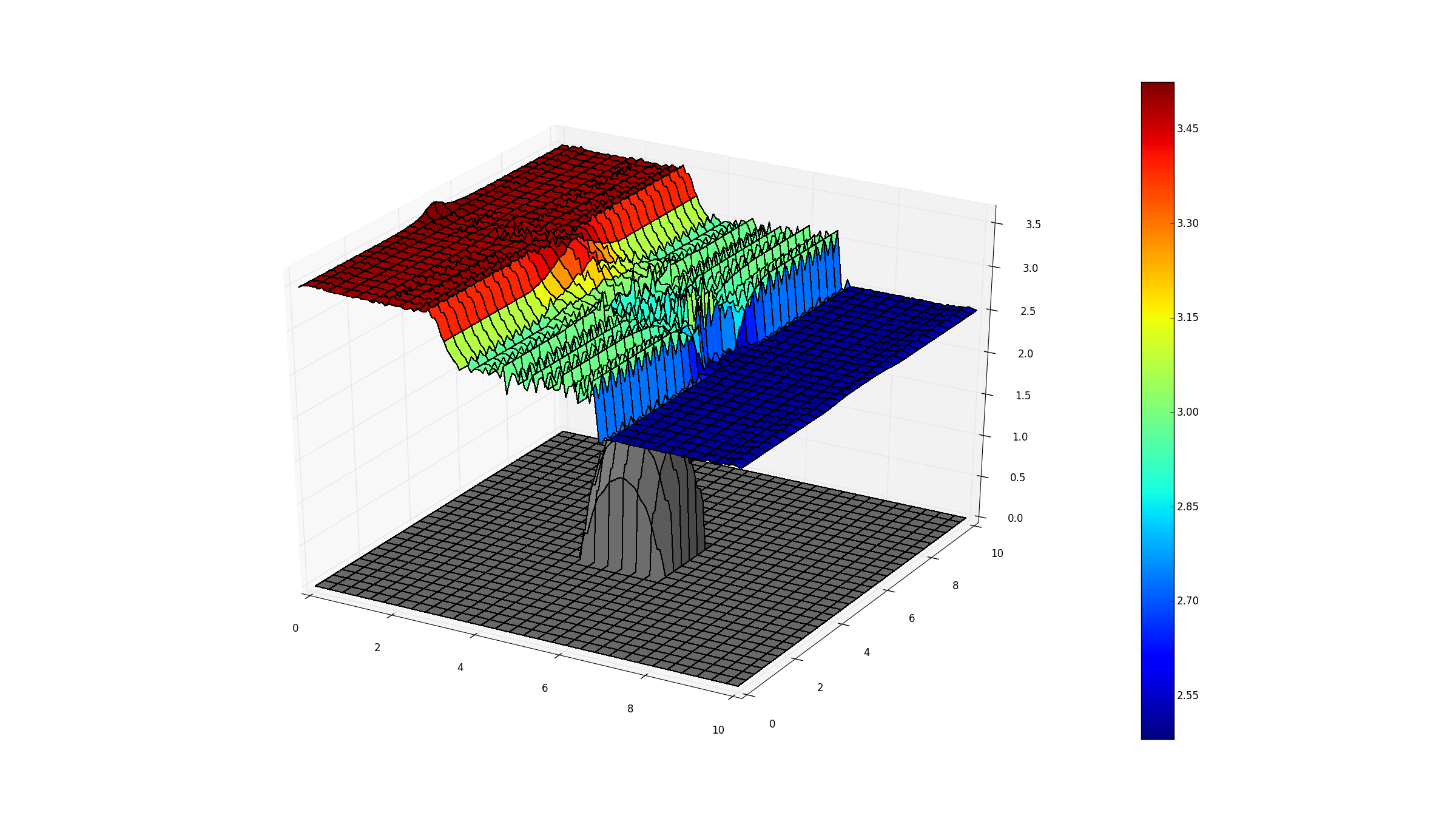}
    }
    \caption{ECDGSEM, dam break over a discontinuous bump on $40\times40$ elements at $T=1$ and $CFL=0.1$.}
    \label{fig:DamBreakOverDiscBumpEC}
\end{figure}

\begin{figure}[!ht]
   \centering
    \subfloat[$20\times 20$]
    {
        \includegraphics[scale=0.1725,trim=250 80 250 80, clip]{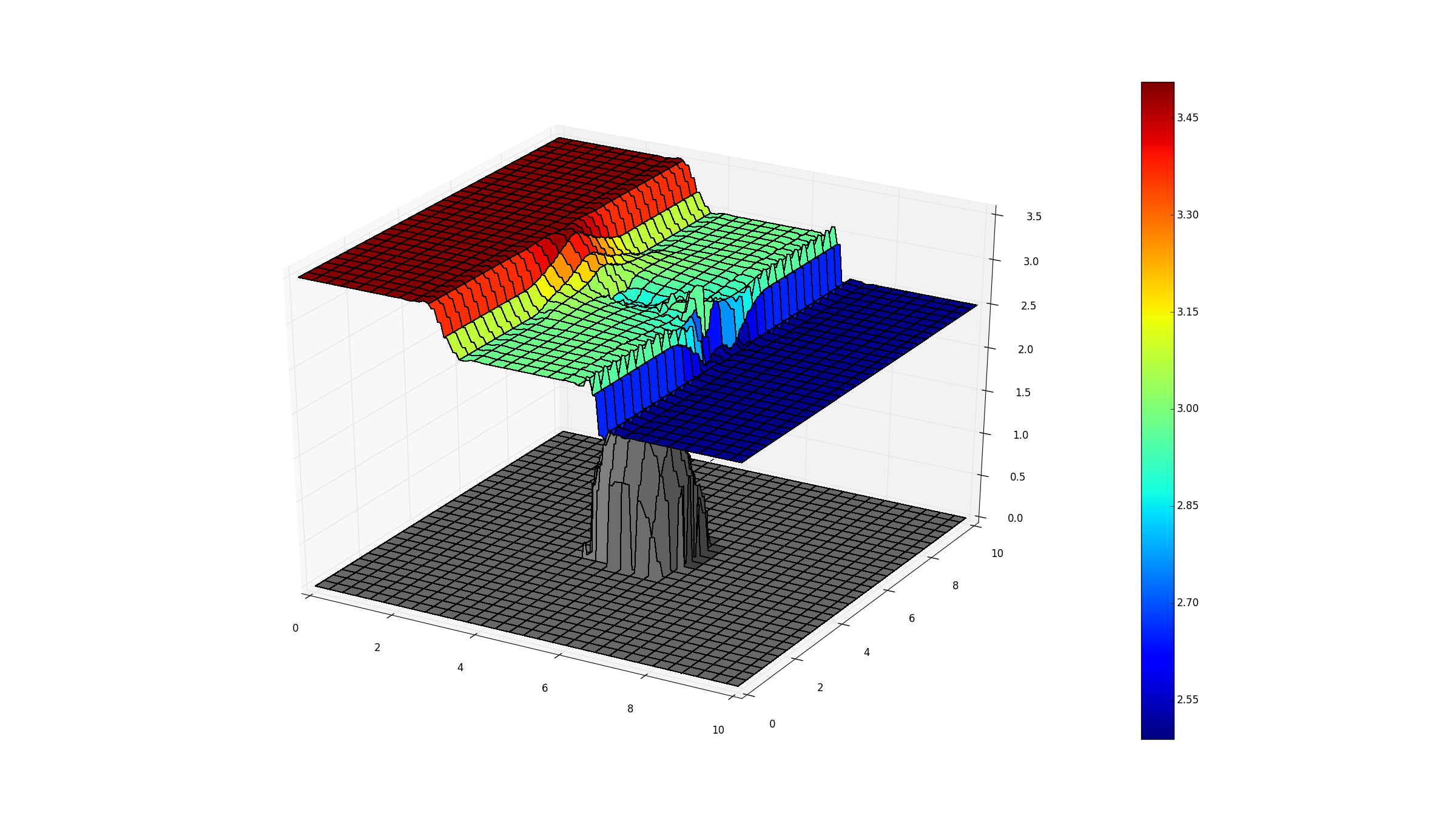}
    }
\subfloat[$40\times 40$]
    {
        \includegraphics[scale=0.1725,trim=250 80 250 80, clip]{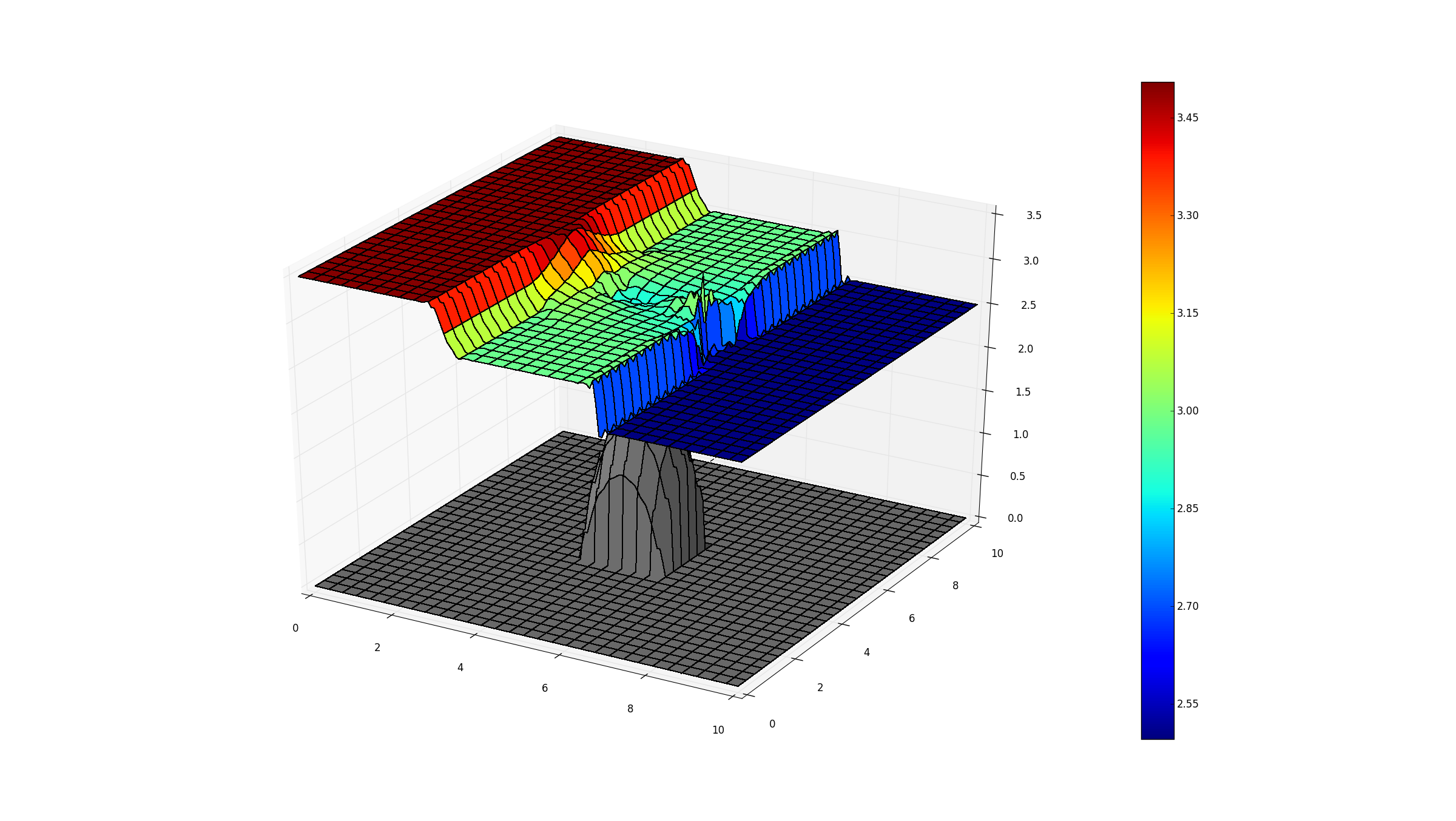}
    }
    \\
\subfloat[$80\times 80$]
    {
        \includegraphics[scale=0.1725,trim=250 80 250 80, clip]{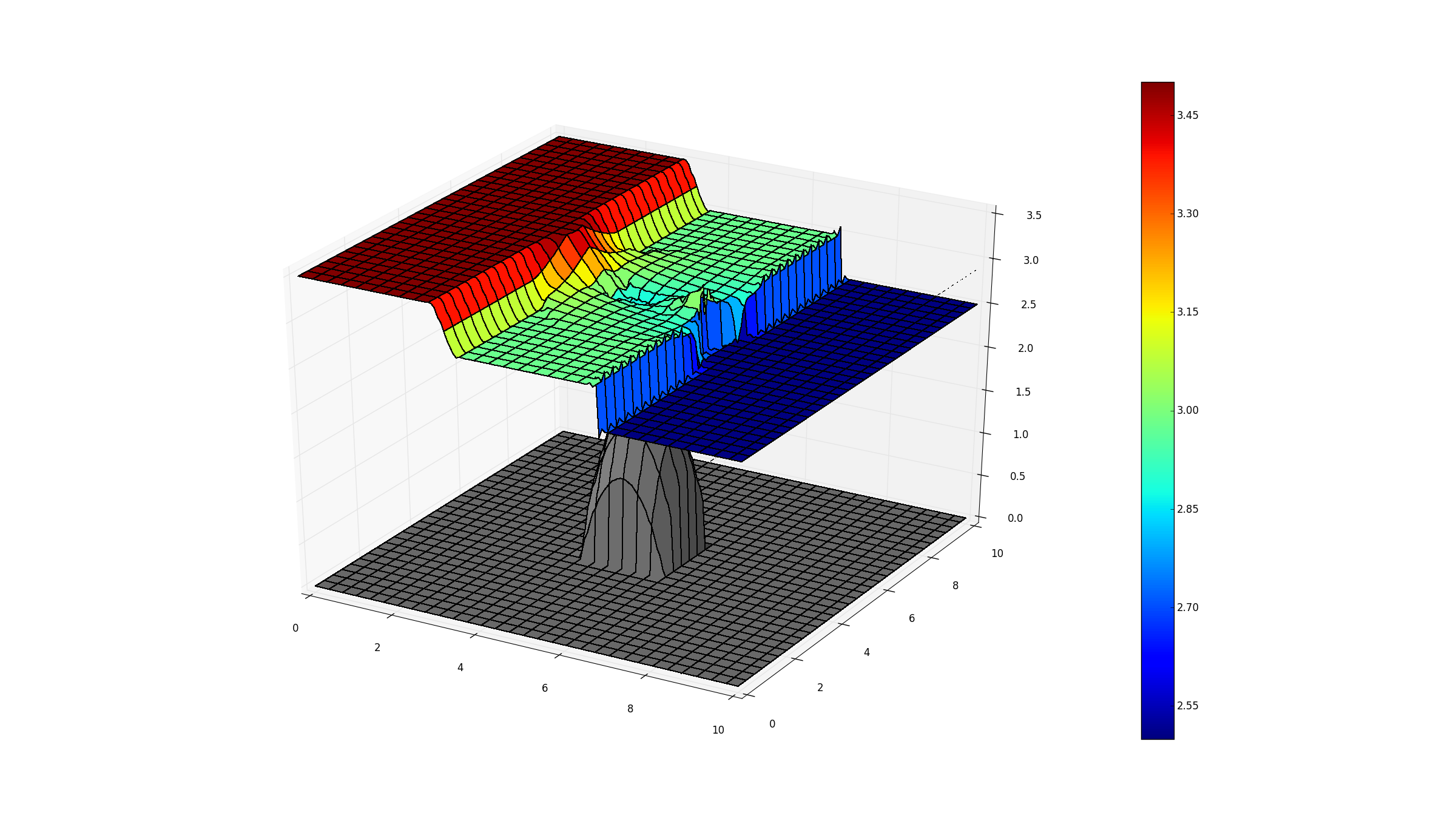}
    }
\subfloat[$160\times 160$]
    {
        \includegraphics[scale=0.1725,trim=250 80 250 80, clip]{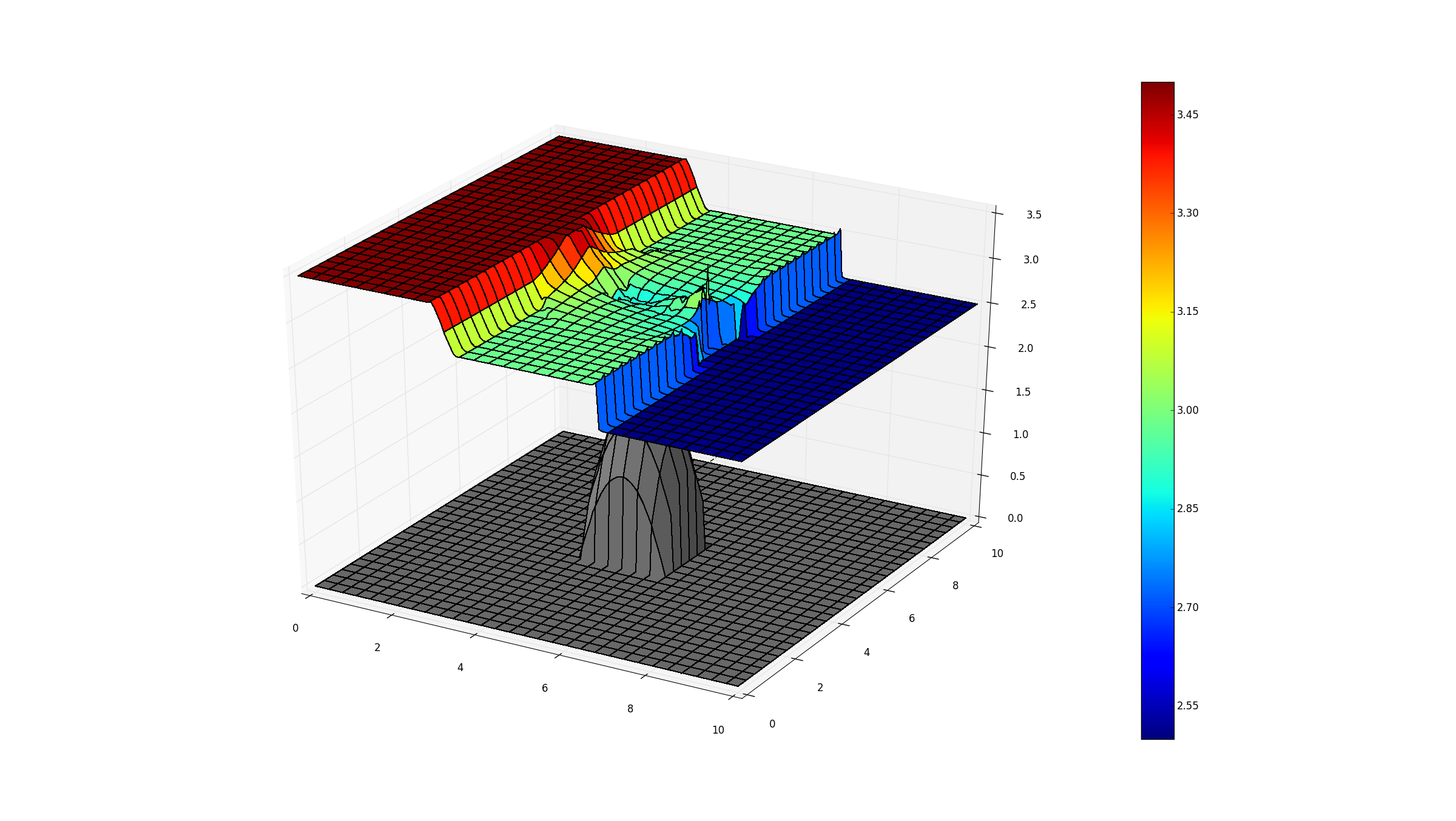}
    }

    \caption{Grid refinement study for the dam break problem modeled by the ESDGSEM at four grid resolutions with $CFL=0.1$ and $N=4$.}
    \label{fig:DamBreakOverDiscBumpES}
\end{figure}

\subsection{Non-linear breaking shallow waters waves}

This test problem simulates the interaction of a hydraulic bore in a shear flow. We also observe how the numerical scheme generates potential vorticity through the passage of a non-uniform bore \cite{Ambati2007,Tassi2007}. For smooth flows the potential vorticity
\begin{equation}
\Pi  = \frac{\Lambda}{h},\quad \Lambda = \pderivative{v}{x} - \pderivative{u}{y},
\end{equation}
is a conserved quantity \cite{Tassi2007,Pedlosky1987}. Hydraulic bores are discontinuities in the flow where energy is dissipated but mass and momentum are conserved. By design the dissipation in the ESDGSEM is generated proportional to the magnitude of the jump in the entropy variables which are large near discontinuities but spectrally small in smooth regions of the flow. Since the dissipation is only applied near bores, potential vorticity can be generated through non-uniform shallow water wave breaking. 

To see how the ESDGSEM generates potential vorticity we begin with an initial flow that has zero vorticity. We consider a flow where the bottom topography is zero, $b\equiv 0$, with the initial linear gravity wave
\begin{equation}\label{eq:gravityWave}
\begin{aligned}
h(x,y) &= 1+ A\sin(ly)\sin(kx)\\
u(x,y) &= -\frac{kAg}{\omega}\sin(ly)\sin(kx)\\
v(x,y) &= \frac{lAg}{\omega}\cos(ly)\cos(kx)
\end{aligned}
\end{equation}
in the rectangular domain $\Omega=[-0.5,0.5]^2$. We set solid wall boundary conditions at $y=\pm0.5$ and periodic boundary conditions in the $x-$direction. The quantities in the initial conditions \eqref{eq:gravityWave} are
\begin{equation}
k =2\pi m,\quad l =(2n+1)\pi,\quad \omega^2 = g(k^2+l^2),\quad A = 0.1,\quad g = 1,\quad m = 2,\quad n = 0.
\end{equation}
We run three simulations for the initial conditions \eqref{eq:gravityWave} each with 40,000 degrees of freedom:
\begin{enumerate}\label{eq:potvort}
\item $N=1$ with a uniform $100\times100$ Cartesian mesh,
\item $N=3$ with a uniform $50\times50$ Cartesian mesh,
\item $N=7$ with a uniform $25\times25$ Cartesian mesh.
\end{enumerate}
The first configuration makes the ESDGSEM a second order spatial approximation which can be directly compared to previous results, namely Fig. 21 in \cite{Ambati2007} and Fig. 4 in \cite{Tassi2007}. 
The other two configurations yield higher order spatial approximations to be used for comparison with the low-order computation.

Due to the non-linearity in the shallow water equations, the higher amplitudes in the gravity waves begin to break at $t\approx 0.5$. The breaking waves create peaks at the crests and troughs of the free surface near to the walls, as shown in Fig.~\ref{fig:bore1}.
\begin{figure}[!ht]
   \centering
    \subfloat[$T=0.0$]
    {
        \includegraphics[scale=0.245]{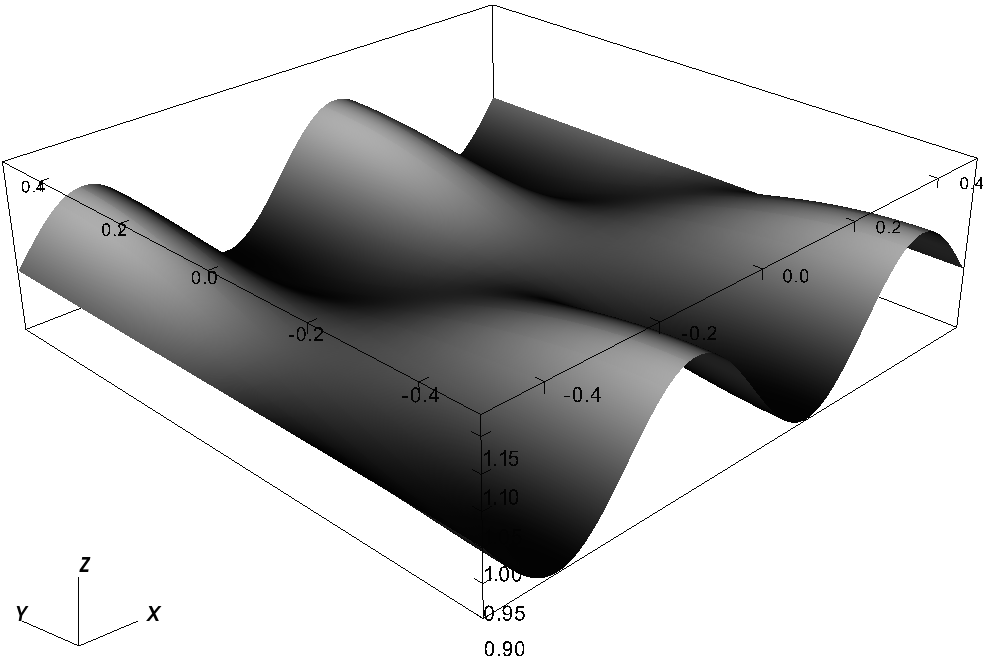}
    }
\subfloat[$T=0.5$]
    {
        \includegraphics[scale=0.245]{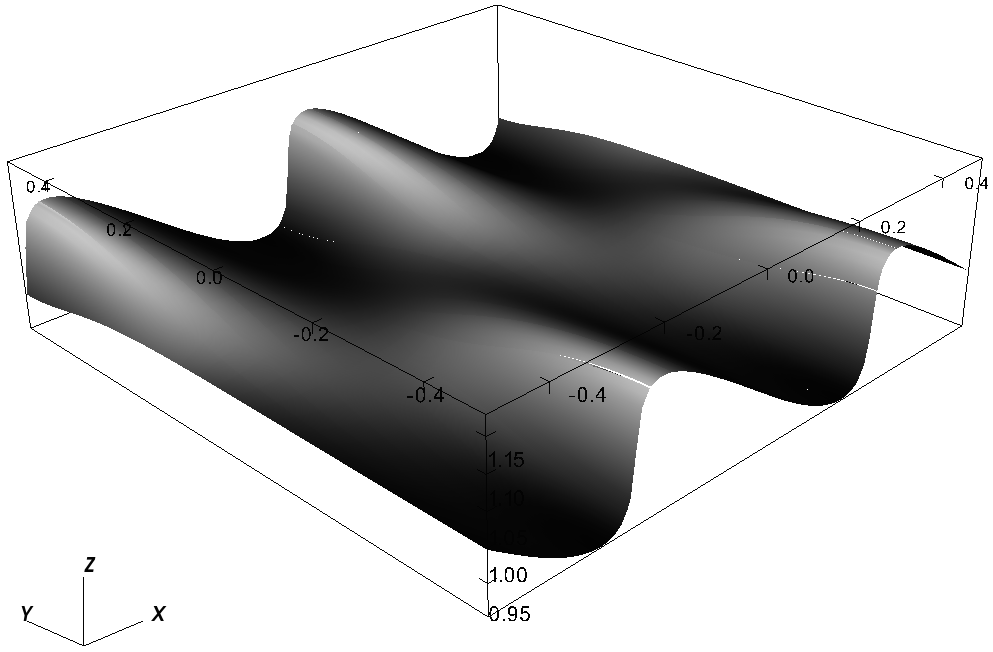}
    }
    \caption{Three dimensional visualization of the initial water height and the water height just before wave breaking occurs at $t\approx 0.5$.}
    \label{fig:bore1}
\end{figure}
The breaking extends to the interior, and bores moving in the negative $x-$direction are formed. The bores are aligned in the $y-$direction with some curvature shown at $T=2$ for $N=1$, $N=3$ and $N=7$ in the left part of Fig.~\ref{fig:bore2}. We also present, in the right part of Fig.~\ref{fig:bore2}, the point-wise potential vorticity at $T=2$ computed using the local derivative to approximate the vorticity $\Lambda$. The computed water height and potential vorticity for $N=1$ compares well with the results of Tassi et al. \cite{Tassi2007}. For the $N=3$ and $N=7$ computations we see there are some oscillations in the vicinity of the discontinuities and more resolution in the smooth parts of the flow. The generated potential vorticity follows the shape of the discontinuities in the solution.
\begin{figure}[!ht]
   \centering
    \subfloat[$h$, $N=1$]
    {
        \includegraphics[scale=0.27]{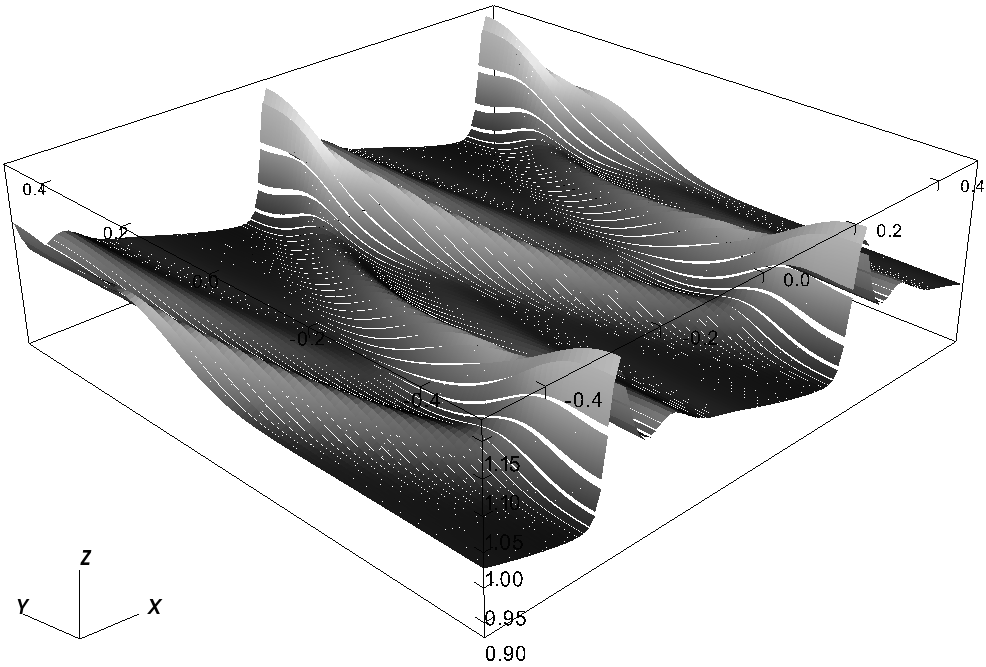}
    }
\subfloat[Potential Vorticity, $N=1$]
    {
        \includegraphics[scale=0.27]{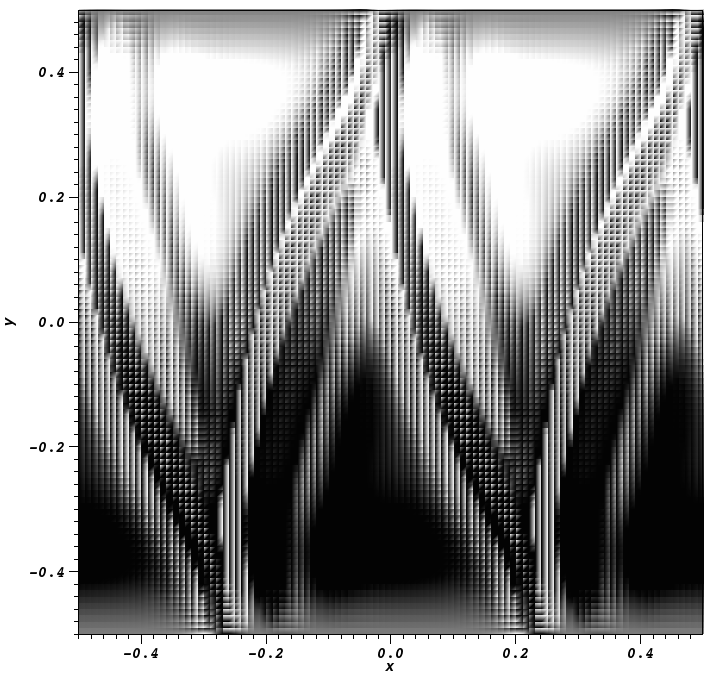}
    }
  \\
      \subfloat[$h$, $N=3$]
    {
        \includegraphics[scale=0.27]{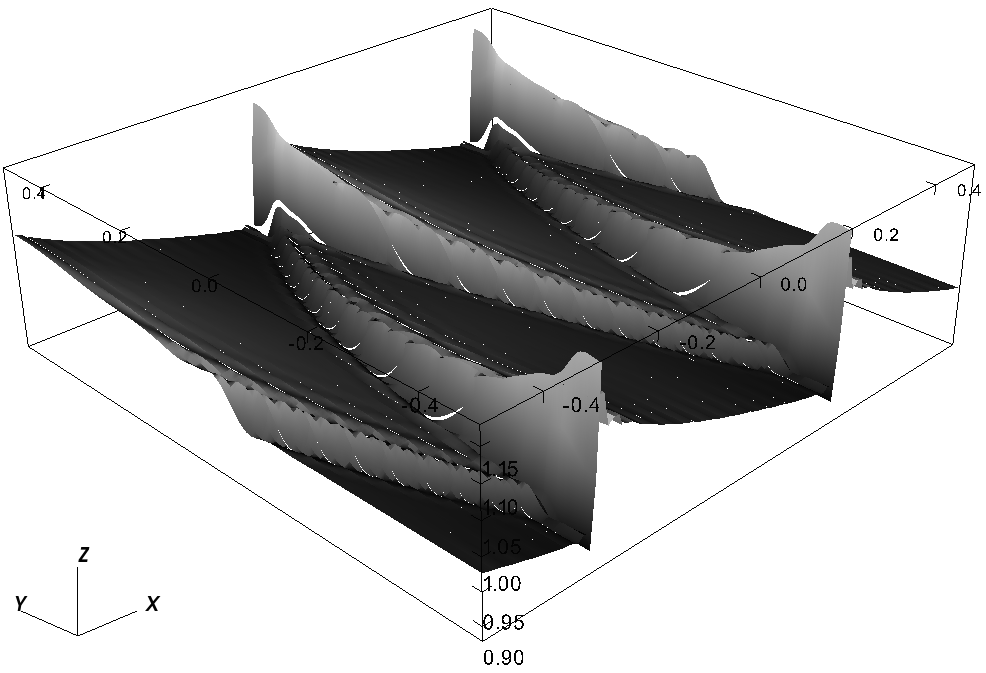}
    }
\subfloat[Potential Vorticity, $N=3$]
    {
        \includegraphics[scale=0.27]{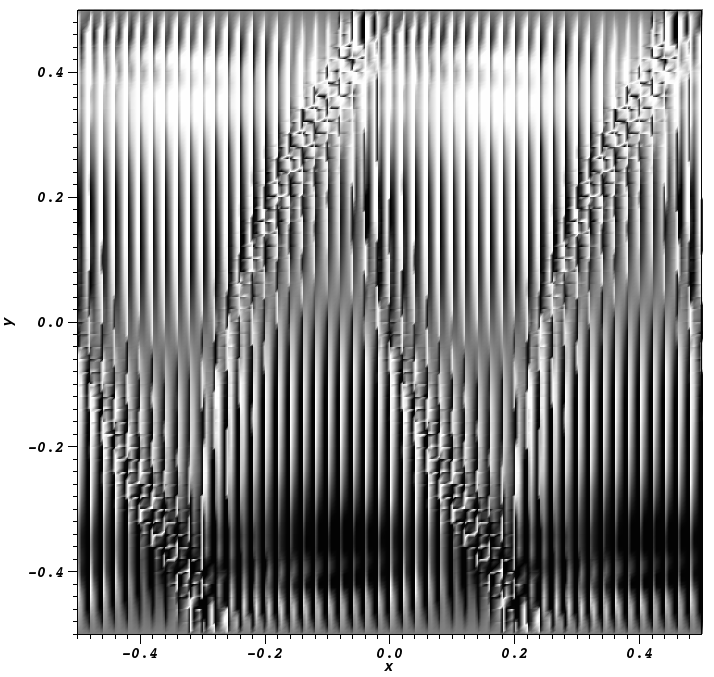}
    }
    \\
        \subfloat[$h$, $N=7$]
    {
        \includegraphics[scale=0.27]{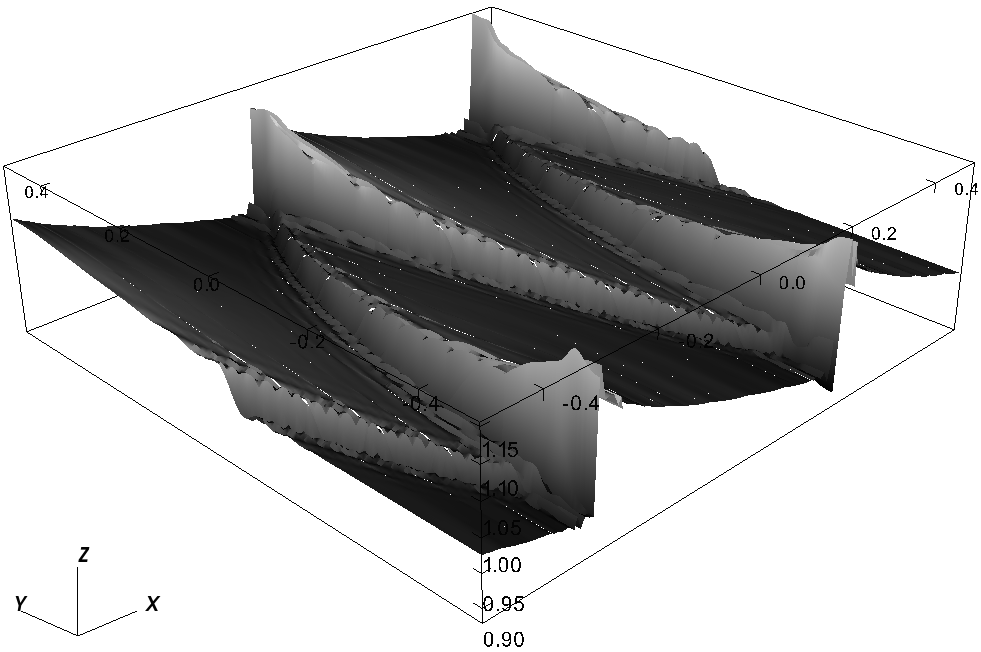}
    }
\subfloat[Potential Vorticity, $N=7$]
    {
        \includegraphics[scale=0.27]{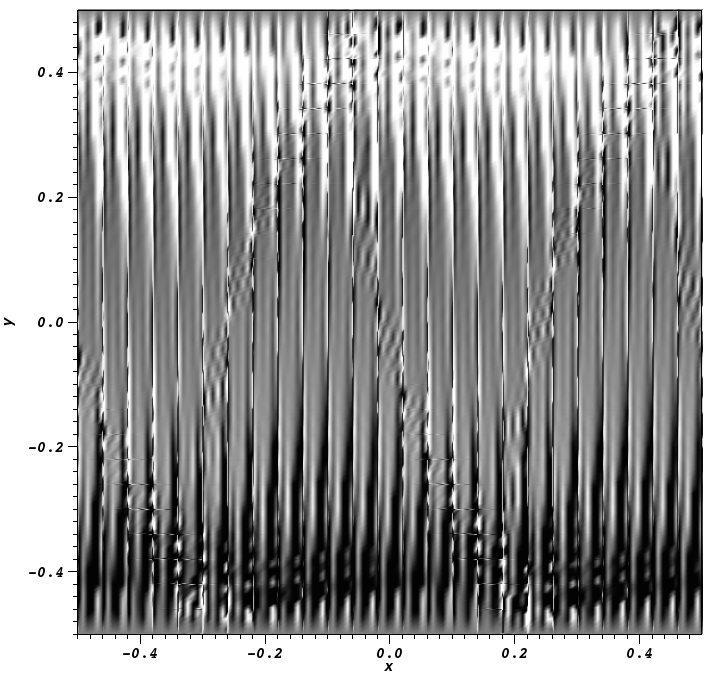}
    }
    \caption{At left is a three dimensional visualization of the water height, $h$, at $T=2$ where wave breaking has occurred and bores are generated. On the right is the approximate potential vorticity \eqref{eq:potvort} where the color ranges between -0.02 (black) to 0.02 (white). Numerical artifacts near element boundaries are visible due to the local derivative computation, without post-processing, of \eqref{eq:potvort}.}
    \label{fig:bore2}
\end{figure}
We again note that ESDGSEM adds dissipation to ensure entropy stability but is not guaranteed to be overshooot free. This explains the noise at the element boundaries in the potential vorticity plots particular for the $N=7$ computation. Without any post-processing of the numerical solution visible numerical artifacts are aligned with the computational grid.

\subsection{Parabolic dam break}
In Secs. \ref{sec:convergence} - \ref{sec:breakOverBump} we have verified the theoretical properties of the EC and ESDGSEM. For the next example we combine each aspect of the numerical scheme and model the fluid flow from the partial break of a parabolic dam. First, we will demonstrate the well-balancedness of the ESDGSEM approximation on curvilinear grids, i.e. numerical verification of Property~\ref{thm:ESDGSEM}.3 of Thm.~\ref{thm:ESDGSEM}. To do so we consider the solution before the failure of the dam, which amounts to two ``lake at rest'' problems on the left and right sides of the dam. On the right side of the dam we also place a discontinuous bottom topography. Then, we allow the dam to fail and examine the flow.

For both numerical tests in this section we use a domain $\Omega = [-5,5]^2$, which is divided into 1600  quadrilateral elements. We model a parabolic dam placed near the center of the domain $\Omega$ with the curve
\begin{equation}\label{dam}
x = \frac{1}{25}y^2 - \frac{1}{4}.
\end{equation}
Finally, for each configuration, we place a discontinuous bottom topography on the downstream side of the dam of the form
\begin{equation}\label{logBottom}
b_4(x,y) = \left\{
\begin{aligned}
&2.0 + \ln(x - 1.25) \quad\textrm{if }x \geq 2.25, \\
&\;\;\quad\qquad 0 \!\;\;\;\quad\qquad\quad\textrm{if }x < 2.25,
\end{aligned}
\right.
\end{equation}

\subsubsection{ESDGSEM well-balancedness}

We first solve the ``lake at rest'' problem on either side of the dam before it fails. We will consider the water height to the left of the dam to be higher than the water on the right of the dam. We consider the initial conditions
\begin{equation}\label{initialDam}
h(x,y,0) = \left\{
\begin{aligned}
&10 - b_4(x) \quad\textrm{if }x < \frac{1}{25}y^2 - \frac{1}{4}, \\[0.1cm]
&\;5 - b_4(x) \quad\;\textrm{if }x > \frac{1}{25}y^2 - \frac{1}{4},
\end{aligned}
\right.,\quad u(x,y,0) = v(x,y,0) = 0.
\end{equation}
We set periodic boundary conditions for each lake individually. This test serves to demonstrate the well-balancedness of the ESDGSEM on a curvilinear mesh, including a discontinuous source term \eqref{logBottom}, Property~\ref{thm:ESDGSEM}.3 of Thm~\ref{thm:ESDGSEM}. For the test problem we take the time step to be $\Delta t = 1/5000$ and integrate to a final time of $T= 5$. We present the $L_2$ error in the approximation of the constant total water height, $H=h+b$, in Table \ref{tab:L2ErrorsLakeAtRest_ESDGSEM}. We find that on either side of the parabolic dam the error in the computed water height is on the order of machine precision.  
\begin{table}[!ht]
\begin{center}
\begin{tabular}{ccc}
\toprule
N & $L_2$-error ESDGSEM (left) & $L_2$-error ESDGSEM (right) \\
\toprule
3 & 5.52E-14 & 4.21E-14 \\\midrule
4 & 7.04E-14 & 6.22E-13 \\\midrule
5 & 1.82E-13 & 1.05E-13 \\\bottomrule
\end{tabular}
\end{center}
\caption{$L_2$-error of the approximation of the total water height, $H=h+b$, to the ``lake at rest'' problem for the parabolic dam at $T = 5$ for various values of $N$. The second column shows the error to the left of the curved dam where there is no bottom topography. The third column gives the error to the right of the curved dam where there is a discontinuous bottom given by \eqref{logBottom}.}
\label{tab:L2ErrorsLakeAtRest_ESDGSEM}
\end{table}%

\subsubsection{Partial dam break with discontinuous bottom topography}

The final demonstration considers the partial failure of a parabolic dam. The initial conditions are given by \eqref{initialDam}. The boundary conditions are periodic along the lines $y=5$ and $y=-5$, Dirichlet along the lines $x=5$ and $x=-5$, and reflecting wall boundary states along the unbroken parts of the parabolic dam. We assume instantaneous failure of the portion of the dam in the region $y\in[-0.5,0.5]$. It is only in this region that the two states interact. 

First, we provide a visual grid convergence study for this complex test problem that has no analytical solution. In Fig. \ref{fig:PartialDamGridConv} we provide the computed solution of the water height for three polynomial orders $N=3$, $N=5$, and $N=7$, with $\Delta t = 1/1500$, integrated to a final time of $T=1.5$. The overlay of quadrilaterals represents the spectral element mesh. We see from the numerical $p$-refinement study that the waves in the approximation are well-resolved for $N=5$ and $N=7$.
\begin{figure}[!ht]
   \centering
    \subfloat[$N=3$]
    {
        \includegraphics[scale=0.1795]{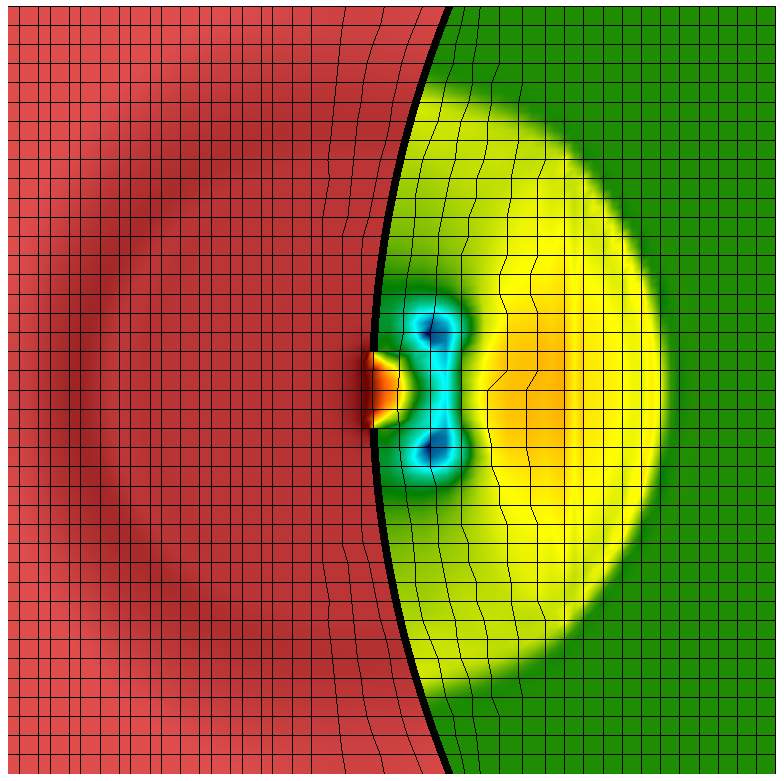} 
    }
\subfloat[$N=5$]
    {
        \includegraphics[scale=0.1795]{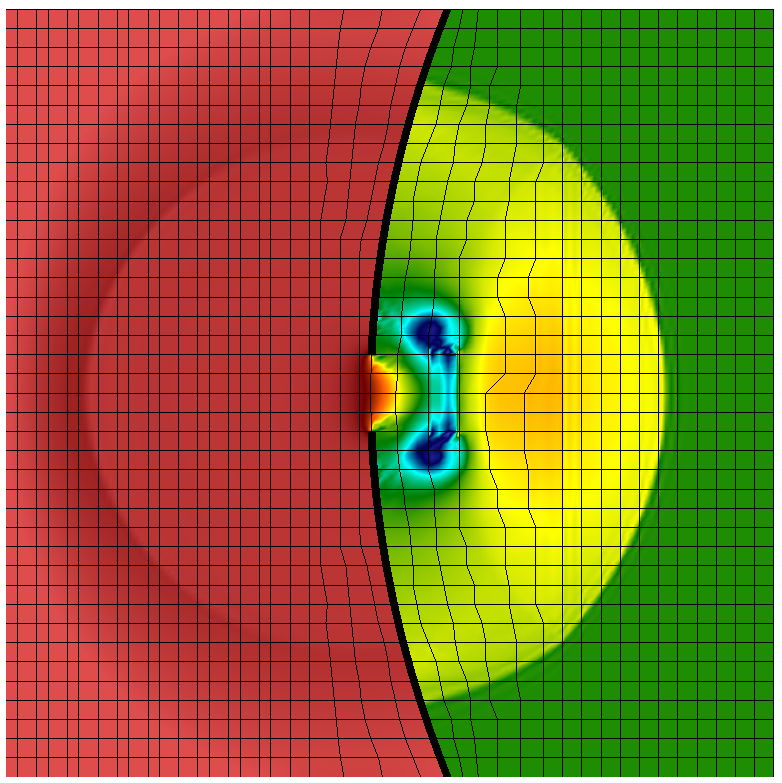}
    }
\subfloat[$N=7$]
    {
        \includegraphics[scale=0.1795]{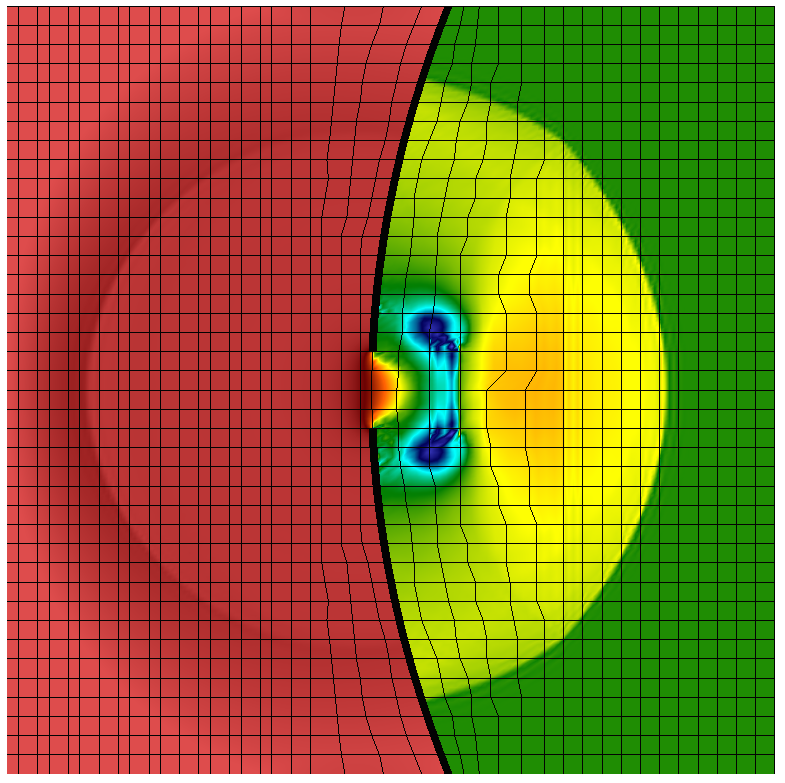}
    }
    \caption{Grid convergence study for the ESDGSEM approximation of the parabolic partial dam break configuration with $\Delta t = 1/1500$ at $T=1.5$. The overlay of quadrilaterals represents the mesh and the thick black line represents the unbroken portion of the parabolic dam. The color ranges between 3 (blue) and 10 (red).}
    \label{fig:PartialDamGridConv}
\end{figure}

From the grid convergence study we know that the computation is sufficiently well resolved with polynomial order $N=5$ in each element. So with $N=5$, we show the evolution of the water height of the partial dam break problem at times $T=0.0$, $T=0.5$, $T=1.0$, and $T=1.5$ in Fig. \ref{fig:PartialDamTime}. Again for this computation, we chose a time step of $\Delta t = 1/1500$. This numerical test combines each aspect of the ESDGSEM approximation, i.e., a discontinuous solution, curvilinear mesh and discontinuous bottom topography. The pseudocolor plots in Fig. \ref{fig:PartialDamTime} show the propagation of eddies near the dam break. Lastly, we provide in Fig. \ref{fig:PartialDam3D} a three dimensional visualization of the partial dam break where we can see on the downstream side of the dam the interaction of the flow with the discontinuous bottom topography \eqref{logBottom}.
\begin{figure}[!ht]
   \centering
    \subfloat[$T=0.0$]
    {
        \includegraphics[scale=0.27]{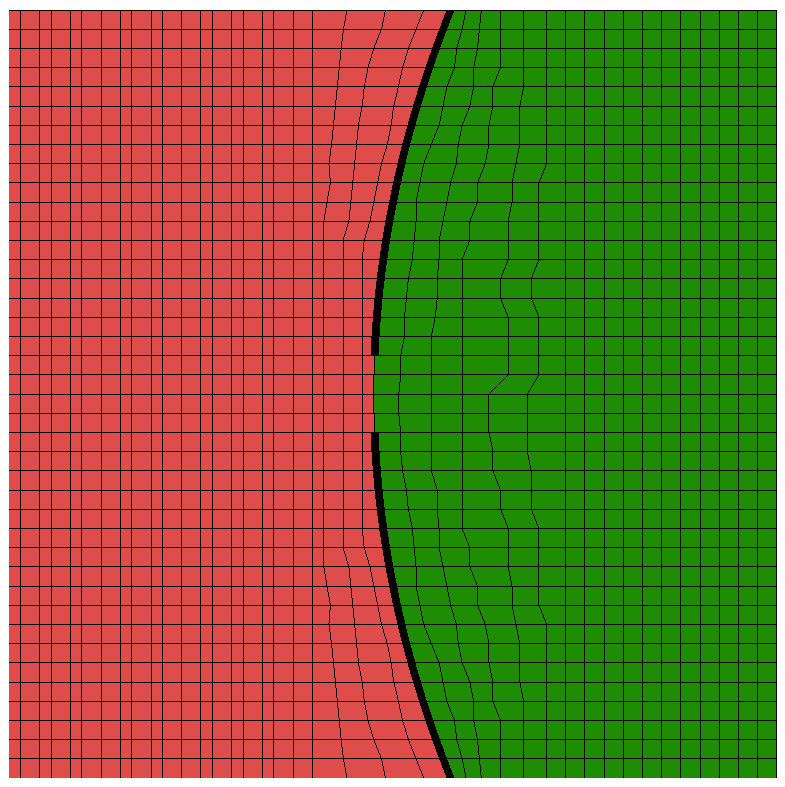}
    }
\subfloat[$T=0.5$]
    {
        \includegraphics[scale=0.27]{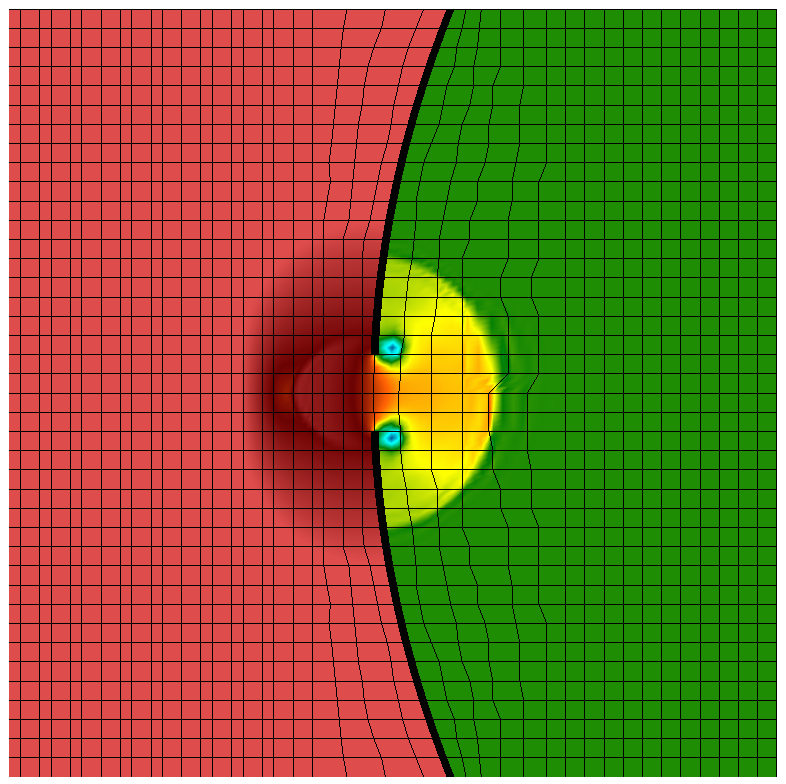}
    }
    \\
\subfloat[$T=1.0$]
    {
        \includegraphics[scale=0.27]{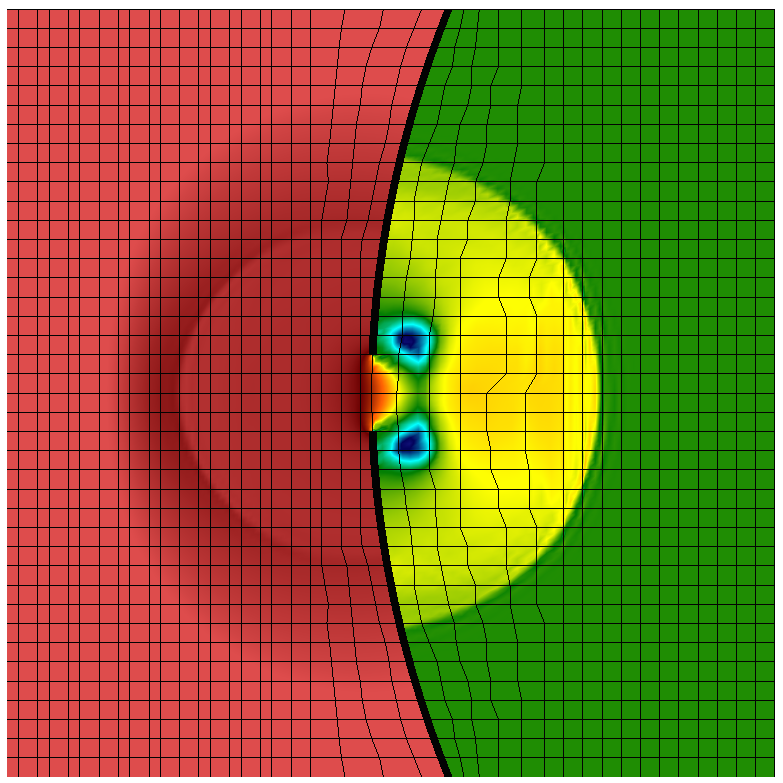}
    }
\subfloat[$T=1.5$]
    {
        \includegraphics[scale=0.27]{N5wd.png}
    }
    \caption{ESDGSEM approximation for the parabolic partial dam break at four times with $N=5$ and $\Delta t = 1/1500$. The overlay of quadrilaterals represents the mesh and the thick black line represents the unbroken portion of the parabolic dam. The color ranges between 3 (blue) and 10 (red).}
    \label{fig:PartialDamTime}
\end{figure}

\begin{figure}[!ht]
   \centering
    \subfloat[$T=0.0$]
    {
        \includegraphics[scale=0.25]{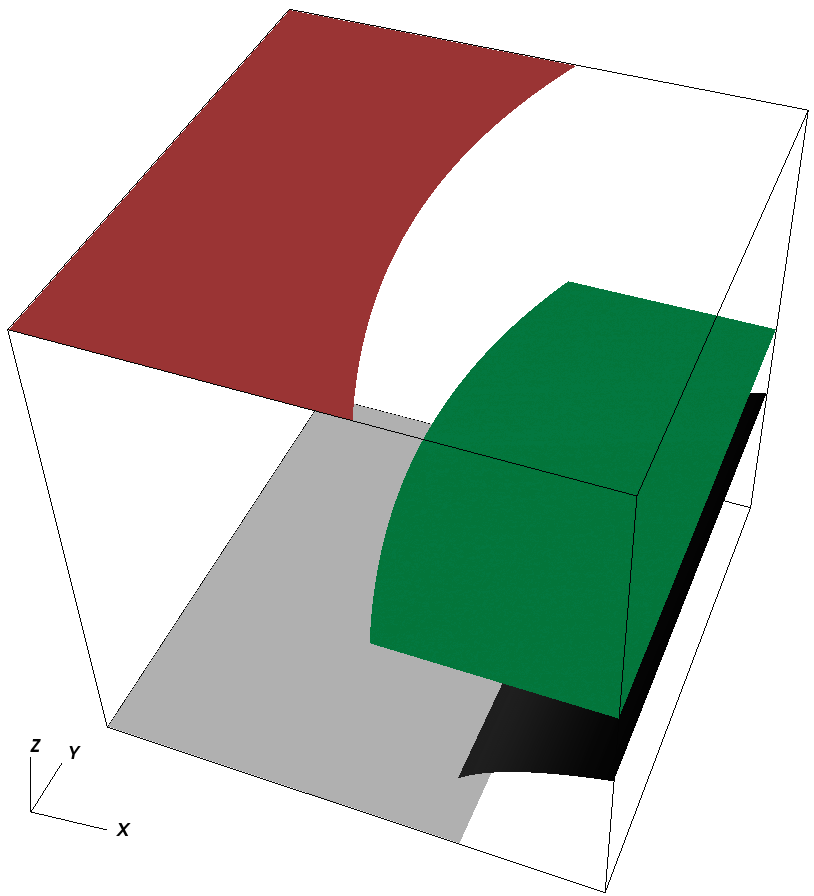}
    }
\subfloat[$T=0.5$]
    {
        \includegraphics[scale=0.25]{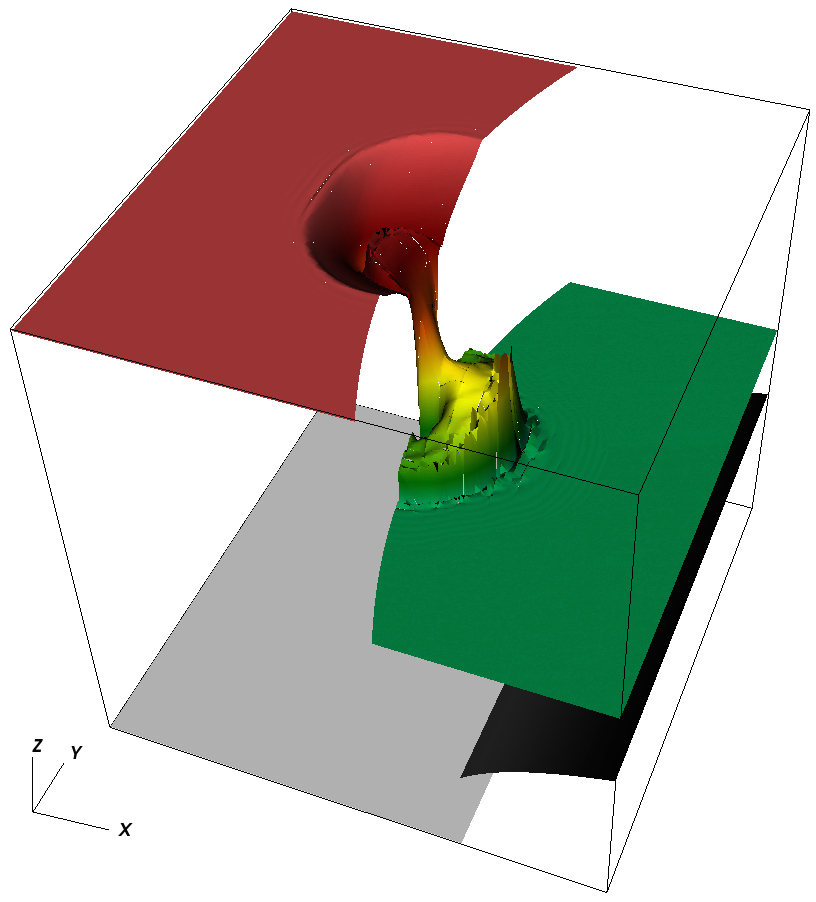}
    }
    \\
\subfloat[$T=1.0$]
    {
        \includegraphics[scale=0.25]{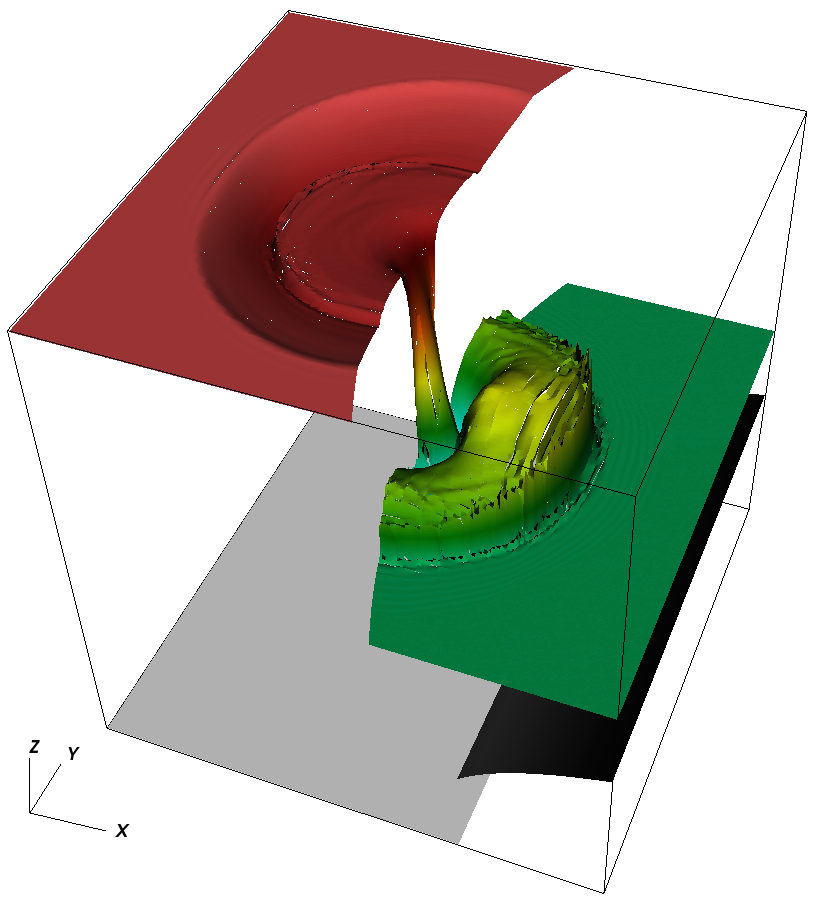}
    }
\subfloat[$T=1.5$]
    {
        \includegraphics[scale=0.25]{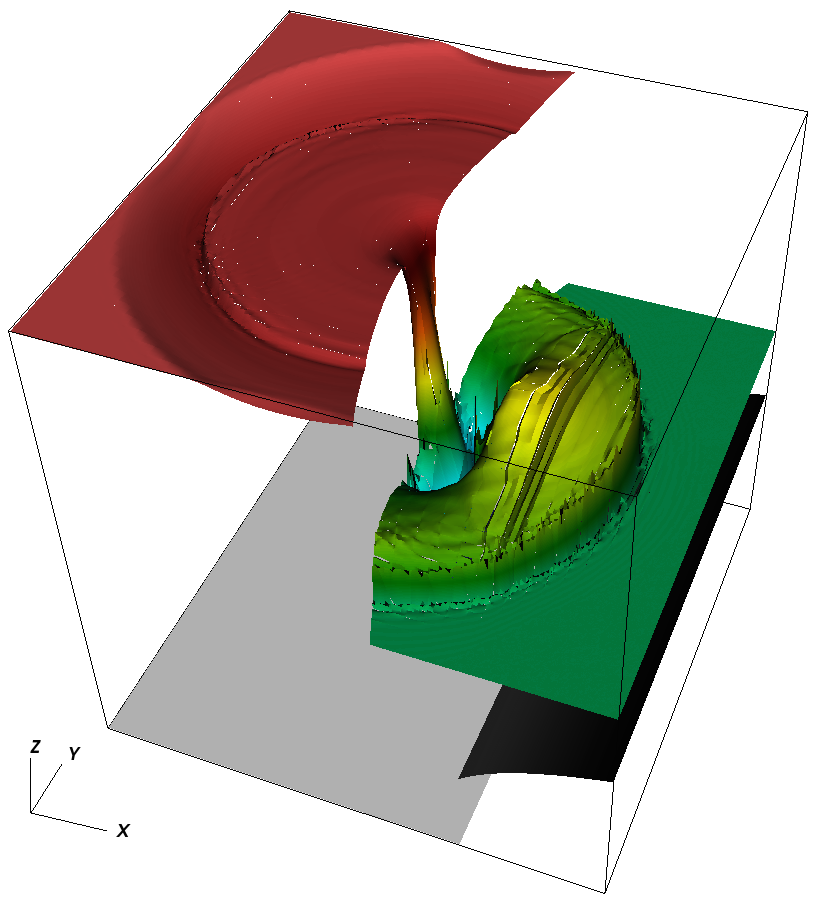}
    }
    \caption{Three dimensional visualization of the ESDGSEM approximation for the parabolic partial dam break at various times with $N=5$ and $\Delta t = 1/1500$. Here the interaction of the flow with the discontinuous bottom topography is clear. The $z-$axis of the plot is from 0 to 10.}
    \label{fig:PartialDam3D}
\end{figure}

The numerical solution of the parabolic dam break problem demonstrates that the entropy stable numerical approximation can capture shock and rarefraction waves. However, the dissipation added to guarantee entropy stability is \textit{not} designed to make the method overshoot free. Additional shock capturing techniques for DG type methods are necessary to remove the remaining oscillations, e.g. \cite{persson2006}.

Lastly, we note that a standard DGSEM scheme is unstable when solving the partial dam break problem from a parabolic dam, even for very small time steps.

\section{Conclusion}\label{sec:conclusion}

In this work we developed a new high-order entropy conserving and entropy stable DGSEM discretisations for the two dimensional shallow water equations on general curvilinear meshes. To highlight the conservative property of the approximation on curvilinear meshes we reformulated the approximation into an equivalent flux differencing form. With this reformulation it is straightforward to demonstrate local conservation. 
Applying results of Fisher and Carpenter \cite{fisher2013}, a careful choice of the numerical volume and surface fluxes leads to an entropy conservative scheme. The flux differencing form also provided additional flexibility to guarantee that the approximation remains well-balanced. For non-constant bottom topographies we found that the recovery of special steady-states of the shallow water equations depends on a special discretisation of the nonlinear source term. By considering a particular source term discretisation we maintained well-balancedness for discontinuous bottom topographies in general curvilinear coordinates. Finally, it is known that energy must be dissipated at shocks, but the entropy conserving scheme is dissipation free modulo any dissipative effects of the time integrator. The numerical solution can therefore capture shocks and rarefactions accurately but at the cost of significant post-shock oscillations. Thus, we provided entropy stable numerical fluxes to add dissipation to the scheme and control overshoots. Note that the dissipation added is merely the amount necessary to guarantee entropy stability and is \emph{not} designed to make the method overshoot free.

We provided six numerical examples to demonstrate and underline the theoretical findings. The simulation of the flow from a parabolic shaped partial dam break  exercised each component of the novel ESDGSEM approximation.

\appendix

\section{Proof of Prop. \ref{prop1}}\label{sec:fluxDiffFormProof}
\begin{proof}
To show the flux difference formula \eqref{eq:fluxDiffForm} we first consider the single difference for the components $i=1$ and $i=2$ and a fixed index for $j$. The argument presented readily extends to the other flux difference components in both the $i$ and $j$ directions. From the high-order flux extension on curvilinear grids \eqref{highOrderFluxCurvilinear} we have for $i=1$
\begin{equation}\label{firstFluxDiffTerm}
\bar{\tilde{F}}_k^{1,j} = \sum_{m=1}^{N}2Q_{0 m}\left(F_k^{vol}(\statevec W^{0,j}, \statevec W^{m,j})\average{Ja^1_1}_{(0,m),j}+G_k^{vol}(\statevec W^{0,j}, \statevec W^{m,j})\average{Ja^1_2}_{(0,m),j}\right),
\end{equation}
and for $i=2$ 
\begin{equation}\label{secondFluxDiffTerm}
\bar{\tilde{F}}_k^{2,j} = \sum_{m=2}^{N}\sum_{\ell=0}^{1}2Q_{\ell m}\left(F_k^{vol}(\statevec W^{\ell,j}, \statevec W^{m,j})\average{Ja^1_1}_{(\ell,m),j}+G_k^{vol}(\statevec W^{\ell,j}, \statevec W^{m,j})\average{Ja^1_2}_{(\ell,m),j}\right).
\end{equation}
We expand the second component \eqref{secondFluxDiffTerm} to find
\begin{equation}\label{secondFluxDiffTermExpand}
\begin{aligned}
\bar{\tilde{F}}_k^{2,j} = \sum_{m=2}^{N}&\left(2Q_{0 m}\left(F_k^{vol}(\statevec W^{0,j}, \statevec W^{m,j})\average{Ja^1_1}_{(0,m),j}+G_k^{vol}(\statevec W^{0,j}, \statevec W^{m,j})\average{Ja^1_2}_{(0,m),j}\right)\right.\\
&\,+\left.2Q_{1 m}\left(F_k^{vol}(\statevec W^{1,j}, \statevec W^{m,j})\average{Ja^1_1}_{(1,m),j}+G_k^{vol}(\statevec W^{1,j}, \statevec W^{m,j})\average{Ja^1_2}_{(1,m),j}\right)\right).
\end{aligned}
\end{equation}
We subtract \eqref{firstFluxDiffTerm} from \eqref{secondFluxDiffTermExpand} and cancel like terms to determine
\begin{equation}\label{FluxDiffTerm2}
\begin{aligned}
\bar{\tilde{F}}_k^{2,j} -\bar{\tilde{F}}_k^{1,j} &= -2Q_{01}\left(F_k^{vol}(\statevec W^{0,j}, \statevec W^{m,j})\average{Ja^1_1}_{(0,m),j}+G_k^{vol}(\statevec W^{0,j}, \statevec W^{m,j})\average{Ja^1_2}_{(0,m),j}\right) \\
&\quad+\sum_{m=2}^N 2Q_{1 m}\left(F_k^{vol}(\statevec W^{1,j}, \statevec W^{m,j})\average{Ja^1_1}_{(1,m),j}+G_k^{vol}(\statevec W^{1,j}, \statevec W^{m,j})\average{Ja^1_2}_{(1,m),j}\right).
\end{aligned}
\end{equation}
We know from the nearly skew-symmetric structure of the SBP matrix $\qmat$ that
\begin{equation}\label{QStuff}
-Q_{01} = Q_{10}, \quad Q_{11} = 0.
\end{equation}
We use the properties \eqref{QStuff} and that the arithmetic mean and the volume flux, by assumption, are symmetric to collect the terms from \eqref{FluxDiffTerm2} into a single sum
\begin{equation}\label{sumDiffProof}
\begin{aligned}
\bar{\tilde{F}}_k^{2,j} -\bar{\tilde{F}}_k^{1,j} = \sum_{m=0}^N 2Q_{1 m}\left(F_k^{vol}(\statevec W^{1,j}, \statevec W^{m,j})\average{Ja^1_1}_{(1,m),j}+G_k^{vol}(\statevec W^{1,j}, \statevec W^{m,j})\average{Ja^1_2}_{(1,m),j}\right).
\end{aligned}
\end{equation}
We then generalise the calculation of the flux difference
\begin{equation}\label{sumDiffProofGeneral}
\begin{aligned}
\bar{\tilde{F}}_k^{{i+1},j} -\bar{\tilde{F}}_k^{i,j} = \sum_{m=0}^N 2Q_{i m}\left(F_k^{vol}(\statevec W^{i,j}, \statevec W^{m,j})\average{Ja^1_1}_{(i,m),j}+G_k^{vol}(\statevec W^{i,j}, \statevec W^{m,j})\average{Ja^1_2}_{(i,m),j}\right),
\end{aligned}
\end{equation}
for $i = 0,\ldots,N$. We then premultiply by the inverse of $\mmat$ to obtain the desired flux differencing result \eqref{eq:fluxDiffForm} for $\diffmat\overline{\mat{\tilde{F}}}$
\begin{equation}\label{sumDiffProofGeneral2}
\begin{aligned}
\frac{\bar{\tilde{F}}_k^{{i+1},j} -\bar{\tilde{F}}_k^{i,j}}{\omega_i} = \frac{1}{\omega_i}\sum_{m=0}^N 2Q_{i m}\left(F_k^{vol}(\statevec W^{i,j}, \statevec W^{m,j})\average{Ja^1_1}_{(i,m),j}+G_k^{vol}(\statevec W^{i,j}, \statevec W^{m,j})\average{Ja^1_2}_{(i,m),j}\right).
\end{aligned}
\end{equation}
An identical strategy can be used in the $j$ index direction to rewrite the flux difference in the $y-$direction, $\overline{\mat{\tilde{G}}}_k\boldsymbol{\Delta}^T\mmat^{-1}$, in the similar indicial form \eqref{eq:fluxDiffFormG}.
\end{proof}

\section{Proof of Thm.~\ref{thm:ECDGSEM}}

\subsection{Proof of Property~\ref{thm:ECDGSEM}.2}\label{sec:ec_proof}
\begin{proof}
In the flux differencing scheme we use different but consistent fluxes for the volume and interface contributions. The interface fluxes \eqref{eq:surfaceFluxes} are known to be entropy conservative \cite{fjordholm2011}. We will demonstrate here that the volume fluxes are also entropy conservative. The volume fluxes we use are
\begin{equation}
\begin{aligned}
\statevec{F}^{vol}(\statevec{W}^{i,j},\statevec{W}^{m,j}) &= \begin{pmatrix}
\average{hu}_{(i,m),j} \\[0.1cm]
\average{hu}_{(i,m),j}\average{u}_{(i,m),j} +g\average{h}_{(i,m),j}^2-\half g \average{h^2}_{(i,m),j} \\[0.1cm]
\average{hu}_{(i,m),j}\average{v}_{(i,m),j}
\end{pmatrix},\\
\\
\statevec{G}^{vol}(\statevec{W}^{i,j},\statevec{W}^{i,m}) &= \begin{pmatrix}
\average{hv}_{i,(j,m)} \\[0.1cm] 
\average{hv}_{i,(j,m)}\average{u}_{i,(j,m)} \\[0.1cm]
\average{hv}_{i,(j,m)}\average{v}_{i,(j,m)} +g\average{h}_{i,(j,m)}^2-\half g \average{h^2}_{i,(j,m)}
\end{pmatrix}.
\end{aligned}
\end{equation}
Similar to \cite{fjordholm2011}, a criterion for discrete entropy conservation is
\begin{equation}
\label{entropyCriterion}
\jump{\statevec{q}\,}^T (\statevec{F}^{vol} + \statevec{G}^{vol})= \jump{\phi} + \average{\statevec{q}}^T\statevec{s} = \jump{\phi} + g\average{hu}\jump{b} + g\average{hv} \jump{b}, \\
\end{equation}
where 
the entropy potential $\phi$ is defined as
\begin{equation}
\phi = \statevec{q}\cdot (\statevec{f}+\statevec{g}) - (\mathcal{F}+\mathcal{G}) = \half g h^2 u + \half g h^2 v, \\
\end{equation}
where $\statevec{f},\,\statevec{g}$ are the physical fluxes \eqref{eq:physicalFluxes}, $\mathcal{F},\,\mathcal{G}$ are the entropy fluxes \eqref{EnergyFluxes} and we have used the consistent auxiliary source term discretisation 
\begin{equation}
\statevec{s}:= \begin{pmatrix}
0\\[0.1cm]
g\frac{\average{hu}}{\average{u}} \jump{b}\\[0.1cm]
g\frac{\average{hv}}{\average{v}} \jump{b}
\end{pmatrix}.\\ 
\end{equation}
The jump in entropy variables is
\begin{equation}
\begin{aligned}
&\jump{\statevec{q}\,} = \begin{pmatrix}
g\jump{h} + g\jump{b} - \average{u}\jump{u}-\average{v}\jump{v}\\[0.1cm]
\jump{u}\\[0.1cm]
\jump{v}
\end{pmatrix}.\\
\end{aligned}
\end{equation}
We show \eqref{entropyCriterion} holds explicitly for the $x-$direction
\begin{equation}
\begin{aligned}
\jump{\statevec{q}\,}^T \statevec{F}^{vol} &= g \average{hu} \jump{h} - \average{hu}\average{u}\jump{u} - \average{hu}\average{v}\jump{v} + \average{hu}\average{u}\jump{u}   \\
&\qquad+ g\average{h}^2\jump{u} - \half g\average{h^2}\jump{u} + \average{hu}\average{v}\jump{v} + g\average{hu}\jump{b} \\
&=g \average{hu} \jump{h}- \half g\average{h^2}\jump{u}+ g\average{h}^2\jump{u} + g\average{hu}\jump{b}\\
&=g \average{hu} \jump{h}- \half g\average{h^2}\jump{u}+ g\average{h}\jump{hu} - g\average{h}\average{u}\jump{h} + g\average{hu}\jump{b}\\
&=g \average{hu} \jump{h}- \half g\average{h^2}\jump{u}+ g\jump{h^2u} - g\average{hu}\jump{h} - g\average{h}\average{u}\jump{h} + g\average{hu}\jump{b}\\
&= g\jump{h^2u} - \half g\average{h^2}\jump{u} - g\average{h}\average{u}\jump{h} + g\average{hu}\jump{b}\\
&= g\jump{h^2u} - \half g\jump{h^2u} + \half g\average{u}\jump{h^2} - g\average{h}\average{u}\jump{h} + g\average{hu}\jump{b}\\
&= \half g\jump{h^2u} + g\average{u}\jump{b} = \jump{\phi} + g\average{hu}\jump{b},
\end{aligned}
\end{equation}
and conclude that $\statevec{F}^{vol}$ is entropy conserving. The $y-$direction is treated analogously to show that $\statevec{G}^{vol}$ is also entropy conserving.
\end{proof}

\subsection{Proof of Property~\ref{thm:ECDGSEM}.3}\label{sec:wb_proof}
\begin{proof}
To verify the well-balancedness of the scheme we need to show that it solves the ``lake at rest'' problem correctly, meaning that the initial conditions $h+b=\textrm{const}$ and $u=v=0$ are preserved for all time. This happens if the discrete time derivatives vanish. It is immediately satisfied the discretised continuity equation since $u$ or $v$ are factors in all the terms. However, it is not immediately clear for the momentum equations. We show here in detail that the discrete time derivative vanishes for the $hu$ equation, the $hv$ equation is handled analogously.

In the following proof we make extensive use of the Hadamard product notation \eqref{DefHadamard} for the component-wise multiplication of matrices and define component-wise powers of nodal values by
\begin{equation}
\mat{W}^k:=\mat{W}^{k-1}\circ\mat{W}.
\end{equation}
If we fully expand the flux-differencing form \eqref{Eq:CurvilinearECDGSEM} of the $hu$ equation by using \eqref{ExtensionEq} and the cubic forms from \cite{gassner2016} we get the following scheme
\begin{equation}
\label{FullyExtendedxMom}
\resizebox{\hsize}{!}{$
\begin{aligned}
\mat{J} (\mat{h} \circ \mat{u})_t   & + \fourth (\dmat(\mat{y}_\eta \circ \mat{h} \circ\mat{u}^2)  + (-\mat{y}_\xi \circ\mat{h}\circ\mat{u}^2)\dmatT  ) \\
							& + \fourth (\mat{y}_\eta \circ\dmat(\mat{h}\circ\mat{u}^2)  + \mat{h}\circ\mat{u}^2 \circ\dmat(\mat{y}_\eta)  - \mat{h}\circ\mat{u}^2\circ(\mat{y}_\xi) \dmatT -\mat{y}_\xi \circ( \mat{h}\circ\mat{u}^2)\dmatT) \\
							& + \fourth  ( \mat{h}\circ \mat{u}\circ\dmat(\mat{y}_\eta\circ\mat{u}) + \mat{y}_\eta \circ\mat{h} \circ\mat{u} \circ\dmat( \mat{u}) - \mat{h} \circ\mat{u}\circ(\mat{y}_\xi \circ\mat{u})\dmatT      -\mat{y}_\xi \circ\mat{h} \circ\mat{u}\circ( \mat{u})\dmatT ) \\
							& +\fourth \left(\mat{u}\circ\dmat(\mat{y}_\eta\circ \mat{h}\circ\mat{u})  + \mat{y}_\eta \circ\mat{u}\circ\dmat (\mat{h} \circ\mat{u})  - (\mat{y}_\xi\circ \mat{h}\circ\mat{u}^2 )\dmatT  -\mat{y}_\xi\circ \mat{u}\circ(\mat{h} \circ\mat{u})\dmatT  \right)\\
							& + \frac{g}{2} ( \mat{h}\circ\dmat(\mat{y}_\eta\circ \mat{h} ) + \mat{y}_\eta \circ \mat{h}\circ\dmat( \mat{h} )  - \mat{h}\circ(\mat{y}_\xi \circ\mat{h})\dmatT  -\mat{y}_\xi \circ\mat{h}\circ ( \mat{h})\dmatT) \\
							& + \fourth  (-\dmat(\mat{x}_\eta \circ\mat{h} \circ\mat{u}\circ\mat{v}) + (\mat{x}_\xi \circ \mat{h} \circ\mat{u}\circ\mat{v})\dmatT    ) \\
							& + \fourth ( -\mat{x}_\eta \circ\mat{h} \circ\mat{v} \circ\dmat(\mat{u})- \mat{h} \circ\mat{v}\circ\dmat(\mat{x}_\eta \circ\mat{u})    + \mat{h}\circ \mat{v} \circ\mat{x}_\xi \circ(\mat{u})\dmatT   + \mat{h} \circ\mat{v}\circ(\mat{x}_\xi\circ \mat{u})\dmatT  ) \\
							& + \fourth (-\mat{x}_\eta \circ\dmat( \mat{h} \circ\mat{u}\circ\mat{v})  -\mat{h}\circ \mat{u}\circ\mat{v} \circ\dmat(\mat{x}_\eta)  +\mat{h} \circ\mat{u}\circ\mat{v}\circ (\mat{x}_\xi  )\dmatT + \mat{x}_\xi \circ( \mat{h}\circ \mat{u}\circ\mat{v})\dmatT ) \\
							&+\fourth\left(  -\mat{u}\circ\dmat( \mat{x}_\eta \circ\mat{h}\circ\mat{v}  ) -\mat{x}_\eta\circ \mat{u} \circ\dmat (\mat{h}\circ\mat{v})   	+ (\mat{x}_\xi \circ\mat{h}\circ\mat{v}\circ\mat{u}  )\dmatT	 +\mat{x}_\xi\circ \mat{u} \circ(\mat{h}\circ\mat{v})\dmatT 	\right)\\
							&+\half g\mat{h}\circ \left(\mat{y}_\eta\circ\dmat (\mat{b}) +\dmat(\mat{y}_\eta \circ \mat{b}) -\mat{y}_\xi\circ (\mat{b})\dmatT - (\mat{y}_\xi \circ \mat{b})\dmatT \right)  \\
	&= -\frac{g}{2} \mat{M}^{-1}\left( \mat{y}_\eta \circ\average{\mat{h}}_{\xi}\circ\jump{\mat{b}}_{\xi} \right) - \frac{g}{2}\left(-\mat{y}_\xi \circ\average{\mat{h}}_{\eta} \circ\jump{\mat{b}}_{\eta} \right)\mat{M}^{-1}
	+\mat{S}\left(\tilde{\mat{F}}^*_2-\tilde{\mat{F}}_2\right) +\left(\tilde{\mat{G}}^*_2-\tilde{\mat{G}}_2\right)\mat{S}.
\end{aligned}$}
\end{equation}

The source term discretisation is a key factor in obtaining a well-balanced scheme. In accordance with our approach for the flux terms, we have also discretised the source term as a quadratic split form. To account for possibly discontinuous bottom topographies on element interfaces we have introduced an additional interface part that vanishes for continuous topographies. The full source term discretisation is
\begin{equation}
\begin{aligned}
\frac{g}{2} \mat{h} \circ \mat{b}_x \approx \quad &\frac{g}{2}\mat{h} \circ \left(\mat{y}_\eta \circ\dmat (\mat{b}) +\dmat(\mat{y}_\eta \circ\mat{b}) -\mat{y}_\xi  \circ(\mat{b})\dmatT - (\mat{y}_\xi \circ\mat{b})\dmatT \right) \\
&+\frac{g}{2} \mat{M}^{-1}\left( \mat{y}_\eta \circ\average{\mat{h}}_{\xi}\circ\jump{\mat{b}}_{\xi} \right) - \frac{g}{2}\left( \mat{y}_\xi  \circ\average{\mat{h}}_{\eta} \circ\jump{\mat{b}}_{\eta} \right)\mat{M}^{-1},
\end{aligned}
\end{equation}
where we use notation \eqref{SourceJmpApproxMatrices} for the jump in bottom topography and the average water height. Using the definitions of the numerical and physical fluxes, the interface terms of the strong form DG discretisation are
\begin{equation}
\resizebox{\hsize}{!}{$
\begin{aligned}
&\mat{S}\left(\tilde{\mat{F}}^*_2-\tilde{\mat{F}}_2\right) +\left(\tilde{\mat{G}}^*_2-\tilde{\mat{G}}_2\right)\mat{S} =\\
&+\smat\left(\mat{y}_\eta  \circ(\average{\mat{u}}^2\circ\average{\mat{h}} + \frac{g}{2}\average{\mat{h}^2}) -\mat{x}_\eta\circ \average{\mat{u}} \circ\average{\mat{v}} \circ\average{\mat{h}} -  \mat{y}_\eta \circ \left(\mat{h}\circ\mat{v}\circ\mat{u}^2+\frac{g}{2}\mat{h}^2\right) +\mat{x}_\eta \circ\mat{h}\circ\mat{u}\circ\mat{v}) \right) \\
&+\left(-\mat{y}_\xi \circ(\average{\mat{u}}^2\circ\average{\mat{h}} + \frac{g}{2}\average{\mat{h}^2}) +  \mat{x}_\xi  \circ\average{\mat{u}} \circ\average{\mat{v}} \circ\average{\mat{h}}   + \mat{y}_\xi\circ\left(\mat{h}\circ\mat{u}^2+\frac{g}{2}\mat{h}^2\right) - \mat{x}_\xi  \circ \mat{h}\circ\mat{u}\circ\mat{v} \right)\smat.
\end{aligned}$}
\end{equation}
Under the ``lake at rest'' initial conditions, $u=v=0$, many terms in the momentum equation \eqref{FullyExtendedxMom} vanish. The remaining terms are
\begin{equation}
\begin{aligned}
\label{LakeAtRestxMom}
\mat{J}\circ (\mat{h}\circ \mat{u})_t   & + \half ( g\mat{h}\circ\dmat(\mat{y}_\eta \circ \mat{h} ) + \mat{y}_\eta \circ g\mat{h}\circ\dmat( \mat{h} )  - g\mat{h}\circ(\mat{y}_\xi \circ \mat{h})\dmatT  - \mat{y}_\xi \circ g\mat{h} \circ ( \mat{h})\dmatT) \\
&+\half g\mat{h} \circ\left(\mat{y}_\eta \circ\dmat (\mat{b}) +\dmat(\mat{y}_\eta \circ \mat{b}) - \mat{y}_\xi \circ (\mat{b})\dmatT - (\mat{y}_\xi \circ \mat{b})\dmatT \right)  \\
&=-\mat{M}^{-1}\left( \mat{y}_\eta \circ \average{\mat{h}}_{\xi} \circ \jump{\mat{b}}_{\xi} \right) +\smat\left(\mat{y}_\eta \circ \left(\average{\mat{h}^2}  -  \mat{h}^2\ \right) \right)  \\
&- \left( - \mat{y}_\xi \circ \average{\mat{h}}_{\eta} \circ \jump{\mat{b}}_{\eta} \right)\mat{M}^{-1}+\left( - \mat{y}_\xi \left( \average{\mat{h}^2}   - \mat{h}^2 \right)\right)\smat.
\end{aligned}
\end{equation}
We first consider the volume terms in \eqref{LakeAtRestxMom}. For constant total water height $h+b=\textrm{const}$, the exactness of the derivative operator implies $\dmat(\mat{h}+\mat{b})=(\mat{h}+\mat{b})\dmatT=0$. The remaining terms cancel due to the metric identities \eqref{metricIDsDiscrete} and we have
\begin{equation}
\begin{aligned}
\frac{g}{2}\mat{h}\circ ( \dmat(\mat{y}_\eta  \circ(\mat{h}+\mat{b}) ) &+ \mat{y}_\eta   \circ \dmat( \mat{h}+\mat{b} )  - (\mat{y}_\xi \circ (\mat{h}+\mat{b}))\dmatT  -\mat{y}_\xi \circ ( \mat{h}+\mat{b})\dmatT) \\
&=\frac{g}{2} \mat{h}\circ( (\mat{h}+\mat{b})\dmat(\mat{y}_\eta  )  - (\mat{h}+\mat{b})\circ(\mat{y}_\xi )\dmatT ) =0.
\end{aligned}
\end{equation}
Since we allow discontinuities in the bottom topography at element interfaces we must account for the jump in water height and we cannot guarantee that $\average{h^2}=h^2$. Instead we have at each interface
\begin{equation}
\label{squareHeightError}
\average{h^2}-h^2= \frac{h_i^2+h_o^2}{2} - h_i^2 = \frac{h_o^2-h_i^2}{2}=\pm\half\jump{h^2} = \pm\average{h}\jump{h},
\end{equation}
where $h_i$ denotes the inner and $h_o$ the outer value. The sign of \eqref{squareHeightError} depends upon which interface we consider. It is positive for the right and top interface and negative otherwise. In combination with the sign applied by the matrix, $\smat= \textrm{diag}(\frac{1}{\omega_0},\ldots,-\frac{1}{\omega_N})$, which gives a negative contribution for the top right component and a positive contribution for the bottom left component, we can restate these terms as
\begin{equation}
\begin{aligned}
\label{SurfaceRemains}
& \smat\left(\mat{y}_\eta \circ\left(\average{\mat{h}^2}_\xi-\mat{h}^2\right)\right) = -\mat{M}^{-1}\left(\mat{y}_\eta \circ\left( \average{\mat{h}}_\xi \circ\jump{\mat{h}}_\xi \right)\right),\\
& \left(-\mat{y}_\xi \circ\left(\average{\mat{h}^2}_\eta-\mat{h}^2\right)\right)\smat= - \left(-\mat{y}_\xi \circ\left(\average{\mat{h}}_\eta \circ\jump{\mat{h}}_\eta\right)\right) \mat{M}^{-1}.
\end{aligned}
\end{equation}
With the reformulation \eqref{SurfaceRemains} we show that the interface contributions are zero
\begin{equation}
\begin{aligned}
&-\mat{M}^{-1}\left( \mat{y}_\eta \circ \average{\mat{h}}_{\xi} \circ\jump{\mat{b}}_{\xi} \right) -\mat{M}^{-1}\left(\mat{y}_\eta \circ\left( \average{\mat{h}}_\xi \circ\jump{\mat{h}}_\xi \right)\right)  \\
&- \left( -\mat{y}_\xi \circ\average{\mat{h}}_{\eta} \circ \jump{\mat{b}}_{\eta} \right)\mat{M}^{-1}- \left(- \mat{y}_\xi \circ\left(\average{\mat{h}}_\eta \circ \jump{\mat{h}}_\eta\right)\right) \mat{M}^{-1} \\
= & -\mat{M}^{-1}\left( \mat{y}_\eta \circ\left( \average{\mat{h}}_{\xi}\circ\jump{\mat{b}+\mat{h}}_{\xi} \right)\right)  \\
&- \left( - \mat{y}_\xi \circ\left( \average{\mat{h}}_{\eta} \circ \jump{\mat{b}+\mat{h}}_{\eta}\right) \right) \mat{M}^{-1} =0,
\end{aligned}
\end{equation}
since $h+b=\textrm{const}$, even on interfaces.

We have shown that the discrete time derivative of the $J(hu)_t$ is zero. The treatment of the $J(hv)_t$ contribution is also zero by an analogous argument. Thus, the scheme is well-balanced.
\end{proof}

\section{Algorithmic Description of the ESDGSEM}\label{sec:efficiency}

We outlined the reformulation of the curvilinear DGSEM in the flux differencing form in Sec. \ref{sec:FluxDiff}. Now we provide specific details and restructure the algorithms of a standard DGSEM implementation to incorporate the entropy stable approximation. The flux differencing form for the curvilinear ESDGSEM \eqref{Eq:CurvilinearESDGSEM} is self-contained, but it seems to be a daunting task to implement. We demonstrate, however, that with a few extra procedures and a slight restructuring of a standard DGSEM time derivative routine it is straightforward to implement the newly proposed entropy stable scheme. 

For this discussion we focus on the computation of the time derivative on a single spectral element. The global time derivative is assembled by looping over every element in a mesh \cite{koprivabook}. To make the discussion concrete we utilize the DGSEM structure outlined in Chap. 8.4 of the book by Kopriva \cite{koprivabook}, but the discussion readily extends to any standard DGSEM implementation. 

To begin, we introduce the notation used throughout this section. We store the computed solution on each element, scaled by the Jacobian, in the array $\{Jw_{i,j,n}\}_{i=0,j=0,n=1}^{N,N,nEqn}$ where $N$ is the polynomial order of the approximation and $nEqn$ is the number of equations. The time derivative, also scaled by the Jacobian, is stored in the array $\{J\dot{w}_{i,j,n}\}_{i=0,j=0,n=1}^{N,N,nEqn}$. Each element stores array information about its mapping (the Jacobian, metric terms, etc.) in the $geom$ object. We adopt the notation of a period to denote access to a component of an object. We store the bottom topography contributions in the volume and on the boundary of an element in separate places (for convenience). The array $\{b_{i,j}\}_{i=0,j=0}^{N,N}$ stores the bottom topography evaluated at the Legendre-Gauss-Lobatto nodes. The array $\{db_{i,j,n}\}_{i=0,j=0,n=1}^{N,N,nEqn}$ stores the volume contributions of the source term and the array $\{\jump{b}_{i,ID}\}_{i=0,ID=1}^{N,4}$  stores the jump in the bottom topography along each edge of a quadrilateral element. To compute the source term at each edge we also store the average of the computed water height in the array $\{\average{h}_{i,ID}\}_{i=0,ID=1}^{N,4}$.

First, we outline the details of the source term discretisation. We divide the computation into two parts: surface and volume contributions. For the surface contributions we alter the routine Alg. 137 (\textit{EdgeFluxes}) from \cite{koprivabook}. This is done out of convenience because the \textit{EdgeFluxes} procedure already has access to local information about an edge, its local ID, and the elements that border an edge. Because we are on an unstructured mesh, care must be taken when computing the jump in the bottom topography term. We take the element to the left of an edge, $e_1$, to be the ``interior'' and the element on the right, $e_2$, to be the ``exterior'' so that definition of the $\jump{b}$ terms is clear. Also, we denote the local side ID of an edge on each element by $s_1$ and $s_2$ respectively. We fill a temporary array from the bottom topography term $\{b_{i,j}\}_{i=0,j=0}^{N,N}$ that depends on the local side index of the edge on the elements $e_1$ and $e_2$, e.g., if $s_1 = 1$ then $b_L =e_1.\{b_{i,0}\}_{i=0}^{N}$ and $s_2 = 3$ then $b_R = e_2.\{b_{i,N}\}_{i=0}^N$. Then, after the normal numerical flux at an edge is computed, one adds to the existing algorithm:
\begin{equation}
\begin{aligned}
e_1.\left\{\jump{b}_{i,s_1}\right\}_{i=0}^N &= b_R - b_L,\\
e_2.\left\{\jump{b}_{i,s_2}\right\}_{i=0}^N &= e_1.\left\{\jump{b}_{i,s_1}\right\}_{i=0}^N, \\
e_1.\left\{\average{h}_{i,s_1}\right\}_{i=0}^N&= \frac{1}{2}(h_R + h_L),\\
e_2.\left\{\average{h}_{i,s_2}\right\}_{i=0}^N &= e_1.\left\{\average{h}_{i,s_1}\right\}_{i=0}^N.
\end{aligned}
\end{equation}
Alg. \ref{Alg:BottomContributions} (\textit{BottomContributions}) is a straightforward implementation of \eqref{eq:sourceTerm} to compute the volume contributions of the source term. We note that one could write the volumetric computation of the source term under BLAS3 architecture standards. Because the bottom topography does not depend on time we precompute and store the quantities $\{db_{i,j,n}\}_{i=0,j=0,n=1}^{N,N,nEqn}$ for later use.
\begin{algorithm}[ht]
  \textbf{Procedure} BottomContributions\;
  \KwIn{$x,y,\{h\}_{i=0,j=0}^{N,N},\{b\}_{i=0,j=0}^{N,N},geom,\{D_{i,j}\}_{i=0,j=0}^{N,N}$ \tcp{polynomial derivative matrix}}\;
 $\{db_{i,j,n}\}_{i=0,j=0,n=1}^{N,N,nEqn} \gets 0$\;
   \For{$i=0$ \mytext{to} $N$}{
    \For{$j=0$ \mytext{to} $N$}{
      $\{sum_m\}_{m=1}^6 \gets 0$\;
       \For{$k=0$ \mytext{to} $N$}{
       $sum_1 \gets sum_1 + D_{i,k}*b_{k,j}$\;
       $sum_2 \gets sum_2 + b_{i,k}*D^T_{k,j}$\;
       $sum_3 \gets sum_3 + D_{i,k}*\left((geom.y_\eta)_{k,j}*b_{k,j}\right)$\;
       $sum_4 \gets sum_4 + \left((geom.y_\xi)_{i,k}*b_{i,k}\right)*D^T_{k,j}$\;
       $sum_5 \gets sum_5 + D_{i,k}*\left((geom.x_\eta)_{k,j}*b_{k,j}\right)$\;
       $sum_6 \gets sum_6 + \left((geom.x_\xi)_{i,k}*b_{i,k}\right)*D^T_{k,j}$\;
       }
       $db_{i,j,2} \gets (geom.y_\eta)_{i,j}*sum_1 - (geom.y_\xi)_{i,j}*sum_2 + sum_3 - sum_4$\;\vspace{1mm}
       $db_{i,j,3} \gets -(geom.x_\eta)_{i,j}*sum_1 + (geom.x_\xi)_{i,j}*sum_2 - sum_5 + sum_6$\;
     }
   }
  \KwOut{$\{db\}_{i=0,j=0,n=1}^{N,N,nEqn}$}\;
  \textbf{End Procedure} BottomContributions\;
 \caption{(\textit{BottomContributions}) Computation of the volumetric source term contributions.}
 \label{Alg:BottomContributions}
\end{algorithm}

We next present Algs. \ref{Alg:HighOrderxFlux} (\textit{HighOrder-xFluxDifference}) and \ref{Alg:HighOrderyFlux} (\textit{HighOrder-yFluxDifference}) needed to reformulate the volume contributions of the ESDGSEM into the subcell flux differencing form presented in \eqref{eq:fluxDiffForm} and \eqref{eq:fluxDiffFormG}. Routines to compute the volume fluxes $\statevec{F}^{vol}$ and $\statevec{G}^{vol}$, given in \eqref{volFlux2D}, are straightforward to implement, so we omit an explicit algorithm. To simplify the indexing in each of the high-order flux differencing algorithms we assume that the procedures is passed an appropriate slice from the solution storage and metric term arrays.

\begin{algorithm}[ht]
  \textbf{Procedure} HighOrder-xFluxDifference\;
  \KwIn{$\{{{W}}_{i,n}\}_{i=0,n=1}^{N,nEqn}$ \tcp{slice of solution on Gauss-Lobatto grid}}\;
  \vspace{-0.325in}\;
  \KwIn{$\left\{(y_\eta)_i\right\}_{i=0}^N,\left\{(x_\eta)_i\right\}_{i=0}^N$ \tcp{metric terms on Gauss-Lobatto grid}}\;
  \vspace{-0.3in}\;
  \KwIn{$Q$ \tcp{SBP matrix}}\;
  
  \For{$i=0$ \mytext{to} $N$}{
    \For{$m=0$ \mytext{to} $N$}{
          $\average{y_\eta} \gets \frac{1}{2}\left((y_\eta)_{i} +(y_\eta)_{m} \right)$\;\vspace{1mm}
          $\average{x_\eta} \gets \frac{1}{2}\left((x_\eta)_{i} +(x_\eta)_{m} \right)$\;
          $\left\{d\bar{\tilde{F}}_{i,n}\right\}_{n=1}^{nEqn} \gets \left\{d\bar{\tilde{F}}_{i,n}\right\}_{n=1}^{nEqn} + 2Q_{i,m}\left(\average{y_\eta}*F^{vol}\left(\{W_{i,n}\}_{n=1}^{nEqn},\{W_{m,n}\}_{n=1}^{nEqn}\right)-\average{x_\eta}*G^{vol}\left(\{W_{i,n}\}_{n=1}^{nEqn},\{W_{m,n}\}_{n=1}^{nEqn}\right)\right)$\;
       }
      }
   \KwOut{$\left\{d\bar{\tilde{{F}}}_{i,n}\right\}_{i=0,n=1}^{N,nEqn}$}\;
  \textbf{End Procedure} HighOrder-xFluxDifference\;
 \caption{(\textit{HighOrder-xFluxDifference}) Computation of the flux difference $\boldsymbol{\Delta}\overline{\mat{\tilde{F}}}$.}
 \label{Alg:HighOrderxFlux}
\end{algorithm}

\begin{algorithm}[ht]
  \textbf{Procedure} HighOrder-yFluxDifference\;
    \KwIn{$\{{W}_{j,n}\}_{j=0,n=1}^{N,nEqn}$ \tcp{slice of solution on Gauss-Lobatto grid}}\;
  \vspace{-0.325in}\;
  \KwIn{$\left\{(y_\xi)_j\right\}_{j=0}^N,\left\{(x_\xi)_j\right\}_{j=0}^N$ \tcp{metric terms on Gauss-Lobatto grid}}\;
  \vspace{-0.3in}\;
  \KwIn{$Q$ \tcp{SBP matrix}}\;
  
  \For{$j=0$ \mytext{to} $N$}{
    \For{$m=0$ \mytext{to} $N$}{
          $\average{y_\xi} \gets \frac{1}{2}\left((y_\xi)_{j} +(y_\xi)_{m} \right)$\;\vspace{1mm}
          $\average{x_\xi} \gets \frac{1}{2}\left((x_\xi)_{j} +(x_\xi)_{m} \right)$\;
          $\left\{d\bar{\tilde{G}}_{j,n}\right\}_{n=1}^{nEqn} \gets \left\{d\bar{\tilde{G}}_{j,n}\right\}_{n=1}^{nEqn}+2Q_{j,m}\left(-\average{y_\xi}*F^{vol}\left(\{W_{j,n}\}_{n=1}^{nEqn},\{W_{m,n}\}_{n=1}^{nEqn}\right)+\average{x_\xi}*G^{vol}\left(\{W_{j,n}\}_{n=1}^{nEqn},\{W_{m,n}\}_{n=1}^{nEqn}\right)\right)$\;
       }
      }
   \KwOut{$\left\{d\bar{\tilde{G}}_{j,n}\right\}_{j=0,n=1}^{N+1,nEqn}$}\;
  \textbf{End Procedure} HighOrder-yFluxDifference\;
 \caption{(\textit{HighOrder-yFluxDifference}) Computation of the flux difference $\overline{\mat{\tilde{G}}}\boldsymbol{\Delta}^T$.}
 \label{Alg:HighOrderyFlux}
\end{algorithm}

Now that we have outlined the source term discretisation and the high-order flux extensions we are prepared to present the main algorithm for the efficient ESDGSEM implementation. We restructure the routine Alg. 144 (\textit{MappedDG2DTimeDerivative}) from \cite{koprivabook} that computes the local time derivative on a curved quadrilateral element. We outline the explicit steps to change a standard DGSEM approximation to implement the ESDGSEM for the shallow water equations:
\begin{enumerate}
\item Begin with Alg. 144 (\textit{MappedDG2DTimeDerivative}) from \cite{koprivabook} that computes the local time derivative on an element. 
\item Remove the standard approach that computes the DG derivative (denoted Alg. 92 (\textit{SystemDGDerivative}) in \cite{koprivabook}).
\item Insert the equivalent flux differencing formulation for the volume terms outlined in Sec. \ref{sec:FluxDiff} and detailed in Algs. \ref{Alg:HighOrderxFlux} and \ref{Alg:HighOrderyFlux}.
\item Use the entropy stable numerical fluxes \eqref{fStabilized} and \eqref{gStabilized} at element interfaces.
\item Multiply the precomputed parts of the volume source term discretisation by the water height and add the source term contributions at each element in the volume.
\item Multiply the precomputed parts of the surface source term discretisation by the average water height and add the source term contributions at each element edge.
\item Alg. \ref{Alg:ESDG2DTimeDerivative} (\textit{ESDG2DTimeDerivative}) summarises the reformulation of a standard DG derivative to the computationally efficient flux difference form.
\end{enumerate}
To reiterate, we assume that the approximations to the bottom derivatives are precomputed and stored and the surface source term contributions are computed in an augmented \textit{EdgeFluxes} procedure.
\begin{algorithm}[ht]
  \textbf{Procedure} ESDG2DTimeDerivative\;
  \KwIn{$\{JW_{i,j,n}\}_{i=0,j=0,n=1}^{N,N,nEqn}$ \tcp{solution scaled by Jacobian}}\;
   \vspace{-0.3in}\;
  \KwIn{$\left\{\tilde{F}^{*,es}_{i,n,ID}\right\}_{i=0,n=1,ID=1}^{N,nEqn,4}$, $\left\{\tilde{G}^{*,es}_{j,n,ID}\right\}_{j=0,n=1,ID=1}^{N,nEqn,4}$ \tcp{entropy stable numerical fluxes}}\;
   \vspace{-0.3in}\;
  \KwIn{$geom$ \tcp{element geometry}}\;
   \vspace{-0.3in}\;
  \KwIn{$Q$ \tcp{SBP matrix}}\;
   \vspace{-0.3in}\;
  \KwIn{$\{\omega_i\}_{i=0}^N$ \tcp{Gauss-Lobatto quadrature weights}}\;
   \vspace{-0.3in}\;
  \KwIn{$\{db_{i,j,n}\}_{i=0,j=0,n=1}^{N,N,nEqn},\{\jump{b}_{i,ID}\}_{i=0,ID=1}^{N,4},\{\average{h}_{i,ID}\}_{i=0,ID=1}^{N,4}$ \tcp{bottom contributions}}\;

  $\{W_{i,j,n}\}_{i=0,j=0,n=1}^{N,N,nEqn} \gets \{JW_n\}_{i=0,j=0,n=1}^{N,N,nEqn}/\{(geom.J)_{i,j}\}_{i=0,j=0}^{N,N}$\;
  \tcp{$\xi-$direction}\;
  \vspace{-0.3in}\;
  \For{$j=0$ \mytext{to} $N$}{
       $\left\{\tilde{F}_{0,n}\right\}_{n=1}^{nEqn} \gets (geom.y_\eta)_{0,j}*xFlux\left(\{W_{0,j,n}\}_{n=1}^{nEqn}\right)-(geom.x_\eta)_{0,j}*yFlux\left(\{W_{0,j,n}\}_{n=1}^{nEqn}\right)$\;
       $\left\{\tilde{F}_{N,n}\right\}_{n=1}^{nEqn} \gets (geom.y_\eta)_{N,j}*xFlux\left(\{W_{N,j,n}\}_{n=1}^{nEqn}\right)-(geom.x_\eta)_{N,j}*yFlux\left(\{W_{N,j,n}\}_{n=1}^{nEqn}\right)$\;
       $\left\{d\bar{\tilde{F}}_{i,n}\right\}_{i=0,n=1}^{N+1,nEqn} \gets HighOrder$-$xFluxDifference\left(\left\{W_{i,j,n}\right\}_{i=0,n=1}^{N,nEqn},Q,\left\{(geom.y_\eta)_{i,j}\right\}_{i=0}^N,\left\{(geom.x_\eta)_{i,j}\right\}_{i=0}^N\right)$\;
          \vspace{-0.1in}\;
      \For{$i=0$ \mytext{to} $N$}{
        $\{J\dot{W}_{i,j,n}\}_{n=1}^{nEqn} \gets \{J\dot{W}_{i,j,n}\}_{n=1}^{nEqn} + \frac{1}{\omega_i}*\left\{d\bar{\tilde{F}}_{i,n}\right\}_{n=1}^{nEqn}$\;
      }
      $\{J\dot{W}_{0,j,n}\}_{n=1}^{nEqn} \gets \{J\dot{W}_{0,j,n}\}_{n=1}^{nEqn} - \frac{1}{\omega_0}\left(\left\{\tilde{F}^{*,es}_{j,n,4}\right\}_{n=1}^{nEqn} - \left\{\tilde{F}_{0,n}\right\}_{n=1}^{nEqn}\right) $\;
      $\{J\dot{W}_{N,j,n}\}_{n=1}^{nEqn} \gets\{J\dot{W}_{N,j,n}\}_{n=1}^{nEqn} +\frac{1}{\omega_N}\left(\left\{\tilde{F}^{*,es}_{j,n,2}\right\}_{n=1}^{nEqn} - \left\{\tilde{F}_{N,n}\right\}_{n=1}^{nEqn}\right) $\;
      $ \{J\dot{W}_{0,j,2}\} \gets \{J\dot{W}_{0,j,2}\} +\frac{g}{2\omega_0}(geom.y_\eta)_{0,j}*\jump{b}_{j,4}*\average{h}_{j,4}$\;
      $ \{J\dot{W}_{0,j,3}\} \gets \{J\dot{W}_{0,j,3}\}-\frac{g}{2\omega_0}(geom.x_\eta)_{0,j}*\jump{b}_{j,4}*\average{h}_{j,4}$\;
      $ \{J\dot{W}_{N,j,2}\} \gets \{J\dot{W}_{N,j,2}\} +\frac{g}{2\omega_N}(geom.y_\eta)_{N,j}*\jump{b}_{j,2}*\average{h}_{j,2}$\;
      $ \{J\dot{W}_{N,j,3}\} \gets \{J\dot{W}_{N,j,3}\} -\frac{g}{2\omega_N}(geom.x_\eta)_{N,j}*\jump{b}_{j,2}*\average{h}_{j,2}$\;
      }
    \tcp{$\eta-$direction}\;
  \vspace{-0.3in}\;
  \For{$i=0$ \mytext{to} $N$}{
       $\left\{\tilde{G}_{0,n}\right\}_{n=1}^{nEqn} \gets -(geom.y_\xi)_{i,0}*xFlux\left(\{W_{i,0,n}\}_{n=1}^{nEqn}\right)+(geom.x_\xi)_{i,0}*yFlux\left(\{W_{i,0,n}\}_{n=1}^{nEqn}\right)$\;
       $\left\{\tilde{G}_{N,n}\right\}_{n=1}^{nEqn} \gets -(geom.y_\xi)_{i,N}*xFlux\left(\{W_{i,N,n}\}_{n=1}^{nEqn}\right)+(geom.x_\xi)_{i,N}*yFlux\left(\{W_{i,N,n}\}_{n=1}^{nEqn}\right)$\;
       $\left\{d\bar{\tilde{G}}_{j,n}\right\}_{j=0,n=1}^{N+1,nEqn} \gets HighOrder$-$yFluxDifference\left(\left\{W_{i,j,n}\right\}_{j=0,n=1}^{N,nEqn},Q,\left\{(geom.y_\xi)_{i,j}\right\}_{j=0}^N,\left\{(geom.x_\xi)_{i,j}\right\}_{j=0}^N\right)$\;
      \vspace{-0.1in}\;
      \For{$j=0$ \mytext{to} $N$}{
        $\{J\dot{W}_{i,j,n}\}_{n=1}^{nEqn} \gets \{J\dot{W}_{i,j,n}\}_{n=1}^{nEqn} + \frac{1}{\omega_j}*\left\{d\bar{\tilde{G}}_{j,n}\right\}_{n=1}$\;
      }
      $\{J\dot{W}_{i,0,n}\}_{n=1}^{nEqn} \gets \{J\dot{W}_{i,0,n}\}_{n=1}^{nEqn} - \frac{1}{\omega_0}\left(\left\{\tilde{G}^{*,es}_{i,n,1}\right\}_{n=1}^{nEqn} - \left\{\tilde{G}_{0,n}\right\}_{n=1}^{nEqn}\right) $\;
      $\{J\dot{W}_{i,N,n}\}_{n=1}^{nEqn} \gets\{J\dot{W}_{i,N,n}\}_{n=1}^{nEqn} +\frac{1}{\omega_N}\left(\left\{\tilde{G}^{*,es}_{i,n,3}\right\}_{n=1}^{nEqn} - \left\{\tilde{G}_{N,n}\right\}_{n=1}^{nEqn}\right) $\;
      $ \{J\dot{W}_{i,0,2}\} \gets \{J\dot{W}_{i,0,2}\} -\frac{g}{2\omega_0}(geom.y_\xi)_{i,0}*\jump{b}_{i,1}*\average{h}_{i,1}$\;
      $ \{J\dot{W}_{i,0,3}\} \gets \{J\dot{W}_{i,0,3}\}+\frac{g}{2\omega_0}(geom.x_\xi)_{i,0}*\jump{b}_{i,1}*\average{h}_{i,1}$\;
      $ \{J\dot{W}_{i,N,2}\} \gets \{J\dot{W}_{i,N,2}\} -\frac{g}{2\omega_N}(geom.y_\xi)_{i,N}*\jump{b}_{i,3}*\average{h}_{i,3}$\;
      $ \{J\dot{W}_{i,N,3}\} \gets \{J\dot{W}_{i,N,3}\} +\frac{g}{2\omega_N}(geom.x_\xi)_{i,N}*\jump{b}_{i,3}*\average{h}_{i,3}$\;
      }
  \For{$i=0$ \mytext{to} $N$}{
    \For{$j=0$ \mytext{to} $N$}{      
      $\{J\dot{W}_{i,j,n}\}_{n=1}^{nEqn} \gets -\left(\{J\dot{W}_{i,j,n}\}_{n=1}^{nEqn} + g*W_{i,j,1}*\{db_{i,j,n}\}_{n=1}^{nEqn}\right)$\;
      }
      }
    \KwOut{$\{J\dot{W}_{i,j,n}\}_{i=0,j=0,n=1}^{N,N,nEqn}$}\;
  \textbf{End Procedure} ESDG2DTimeDerivative\;
 \caption{(\textit{ESDG2DTimeDerivative}) Efficient implementation of two dimensional ESDGSEM in curvilinear coordinates.}
 \label{Alg:ESDG2DTimeDerivative}
\end{algorithm}

{\FloatBarrier}

\section*{References}
\bibliographystyle{elsarticle-num}
\bibliography{References}

\end{document}